\documentclass{article}
\usepackage[utf8x]{inputenc}
\usepackage{color}
\usepackage[left=1in,top=1in,right=1in]{geometry}

\usepackage{
amsmath,
amsfonts,
latexsym, 
amsrefs,
amssymb,
amsthm,
tocloft,
hyperref,
ulem
}

\usepackage{mathrsfs}

\usepackage{bm}

\usepackage{tikz}

\usetikzlibrary{decorations.markings}

\usepackage{dsfont}

\usepackage{wrapfig}

\usepackage{enumerate}

\newcommand{\nc}{\color{black}}

\newcommand{\rd}{{\rm d}}

\newcommand{\e}{{\varepsilon}}
\renewcommand{\epsilon}{{\varepsilon}}

\newcommand{\ba}{\bm{\alpha}}
\newcommand{\bb}{\bm{\beta}}
\newcommand{\bx}{\bm{\xi}}
\newcommand{\bh}{\bm{h}}
\DeclareMathOperator{\U}{U}
\newcommand{\Tr}{ {\rm Tr} }
\newcommand{\ii}{\mathrm{i}} 
\newcommand{\N}{\mathbb N}

\newcommand{\1}{\mathds{1}}

\renewcommand{\P}{\mathbb{P}}
\newcommand{\Q}{\mathbb{Q}}
\newcommand{\E}{\mathbb{E}}
\renewcommand{\Re}{{\rm Re}}

\DeclareMathOperator{\OO}{O}
\DeclareMathOperator{\oo}{o}

\def\author#1{\par
    {\centering{\authorfont#1}\par\vspace*{0.05in}}
}

\def\titlefont{\fontsize{13}{15}\bfseries\boldmath\selectfont\centering{}}
\def\authorfont{\fontsize{13}{15}}
\def\abstractfont{\fontsize{8}{10}}

\let\affiliationfont\rhfont

\def\address#1{\par
    {\centering{\affiliationfont#1\par}}\par\vspace*{11pt}
}

\def\body{
\setcounter{footnote}{0}
\def\thefootnote{\alph{footnote}}
\def\@makefnmark{{$^{\rm \@thefnmark}$}}
}

\def\title#1{
    \thispagestyle{plain}
    \vspace*{-14pt}
    \vskip 79pt
    {\centering{\titlefont #1\par}}%
    \vskip 1em
}

\makeatletter
\renewcommand{\subsection}{\@startsection
{subsection}
{2}
{0mm}
{-\baselineskip}
{0\baselineskip}
{\normalfont\itshape}} 
\makeatother

\renewenvironment{abstract}{\par%
    \vspace*{6pt}\noindent 
    \abstractfont
    \noindent\leftskip18pt\rightskip18pt
}{%
  \par}

\numberwithin{equation}{section}

\newtheorem{thm}{Theorem}[section]
\newtheorem{lem}[thm]{Lemma}
\newtheorem{prop}[thm]{Proposition}
\newtheorem{cor}[thm]{Corollary}

\newtheorem{conj}[thm]{Conjecture}

\makeatletter
\renewcommand{\section}{\@startsection
{section}
{1}
{0mm}
{-2\baselineskip}
{1\baselineskip}
{\normalfont\large\scshape\centering}} 
\makeatother

\begin{document}

\title{Maximum of the characteristic polynomial of random unitary matrices}

\vspace{1.2cm}

\noindent\begin{minipage}[b]{0.38\textwidth}

 \author{Louis-Pierre Arguin}

\address{ Department of Mathematics, Baruch College\\
 Graduate Center, City University of New York\\
  louis-pierre.arguin@baruch.cuny.edu}
\end{minipage}
\noindent\begin{minipage}[b]{0.29\textwidth}
\author{David Belius}

\address{ Courant Institute,\\ New York University\\
  david.belius@cantab.net}
 \end{minipage}
\begin{minipage}[b]{0.29\textwidth}

 \author{Paul Bourgade}

\address{Courant Institute,\\ New York University\\
   bourgade@cims.nyu.edu}

 \end{minipage}

\begin{abstract}
It was recently conjectured by Fyodorov, Hiary and Keating that the maximum of the characteristic polynomial on the unit circle of a $N\times N$ random unitary matrix sampled from the Haar measure grows like $CN/(\log N)^{3/4}$ for some random variable $C$.
In this paper, we verify the leading order of this conjecture, that is, we prove that with high probability the maximum lies in the range $[N^{1 - \varepsilon},N^{1 + \varepsilon}]$,
for arbitrarily small $\varepsilon$. 

The method is based on identifying an approximate branching random walk in the Fourier decomposition of the characteristic polynomial,
and uses techniques developed to describe the extremes of branching random walks and of other log-correlated random fields.
A key technical input is the asymptotic analysis of Toeplitz determinants with dimension-dependent symbols.

The original argument for these asymptotics followed the general idea that the statistical mechanics of $1/f$-noise random energy models is governed by a freezing transition. We also prove the conjectured freezing of the free energy for random unitary matrices.

\end{abstract}

\setcounter{tocdepth}{1}

\tableofcontents

\section{Introduction}
For $N\in\N$, consider a random matrix  ${\rm U}_N$ sampled from the group of $N\times N$ unitary matrices with Haar measure.
This distribution is also known as the {\it Circular Unitary Ensemble} (CUE).
This paper studies the extreme values of the characteristic polynomial ${\rm P}_N$ of ${\rm U}_N$, on the unit circle, as $N\to\infty$.
The main result concerns the asymptotics of
$$
 \max_{h\in[0,2\pi]} |{\rm P}_N(e^{\ii h})|=\max_{h\in[0,2\pi]}|\det(e^{\ii h}-{\rm U}_N)|.
$$
It was shown by Keating and Snaith \cite{KeaSna2000} that for a fixed $h$, $\log|{\rm P}_N(e^{\ii h})|$, converges to a standard Gaussian variable when normalized by $(\frac{1}{2}\log N)^{1/2}$.
A recent conjecture of Fyodorov, Hiary and Keating makes a precise prediction for  the large values of the logarithm of the characteristic polynomial.
\begin{conj}[Fyodorov-Hiary-Keating \cite{FyoHiaKea12,FyoKea14}]
\label{conj: fhk}
For $N\in \N$, let $U_N$ be a random matrix sampled uniformly from the group of $N\times N$ unitary matrices.
Write ${\rm P}_N(z)$, $z\in \mathbb C$, for its characteristic polynomial.
Then
\begin{equation}\label{eqn: fhk}
\max_{h\in[0,2\pi]} \log |{\rm P}_N(e^{\ii h})|
=\log N -\frac{3}{4}\log\log N + \mathcal M_N,
\end{equation}
where $(\mathcal M_N, N\in\N)$ is a sequence of random variables that converges in distribution.
\end{conj}

The main result of this paper is a rigorous verification of the prediction for the leading order.
\begin{thm}
\label{thm: main}
For $N\in \N$, let ${\rm U}_N$ be a random matrix sampled uniformly from the group of $N\times N$ unitary matrices.
Write ${\rm P}_N(z)$, $z\in \mathbb C$, for its characteristic polynomial.
Then
\begin{equation}
\lim_{N\to\infty}\frac{\underset{h\in[0,2\pi]}{\max} \log |{\rm P}_N(e^{\ii h})|}{\log N}=1 \qquad \text{in probability.}
\end{equation}
\end{thm}
It is known that the random field $\big((\log N)^{-1/2} \log |{\rm P}_N(e^{\ii h})|, h\in[0,2\pi]\big)$ 
converges in the sense of finite-dimensional distribution to a Gaussian field,  with independent values at macroscopically separated evaluation points \cite{HugKeaOco2001}. On mesoscopic scales,
the covariance between  two  points $ h_1$ and $ h_2$ at distance $\Delta = |e^{\ii h_1}-e^{\ii h_2}|$ behaves like $\frac{\log (1/\Delta)}{\log N}$ when
$\Delta$ is at least $1/N$, and approaches $1$ for smaller distances \cite{BourgadeMesoscopic}.
This kind of decay of correlations is the defining characteristic of a {\it $\log$-correlated random field}. 
The extrema of such fields have recently attracted much attention, cf. Section \ref{sec: prev work}.

The almost perfect correlations below scale $1/N$ suggest that, to first approximation, one can think of the maximum over $[0,2\pi]$ as
a maximum over $N$ random variables with strong correlations on mesoscopic scales.
Strikingly, the leading order prediction of Conjecture \ref{conj: fhk} is that the maximum is close to that
of $N$ centered independent Gaussian random variables of variance $\frac{1}{2}\log N$, which would lie around $\log N-\frac{1}{4}\log\log N$. 
In other words, despite strong correlations between the values of $\log |{\rm P}_N(e^{\ii h})|$ for different $h$,  an analogy with independent Gaussian random variables 
correctly predicts the leading order of the 
maximum. The constant in front of the subleading correction $\log \log N$, however, differs. But, as we will explain below, it is exactly the constant expected for a log-correlated Gaussian field.

The conjecture was derived from precise computations of the moments of a suitable partition function
and of the measure of high points, 
using statistical mechanics techniques developed for describing the extreme value statistics of disordered systems \cite{FyoBou08,FyoLedRos09, FyoLedRos12}.
It is also supported by strong numerical evidence. A precise form for the distribution of the limiting fluctuations, which is consistent with those of log-correlated fields,  is also predicted. We point out that $\log|{\rm P}_N(e^{\ii h})|$ is believed to be a good model for the local behavior of the Riemann zeta function on the critical line. 
In particular, the authors conjecture a similar behavior for the extremes of the Riemann zeta function on an interval $[T,T+2\pi]$ of the critical line with $N$ replaced by $\log T$, see \cite{FyoHiaKea12,FyoKea14} for details
and \cite{ArgBelHar15,Har13} for rigorous proofs for a different random model of the zeta function.

\begin{figure}
\begin{center}
\vspace{-0.9cm}
\label{fig: unitary}
\includegraphics[width=0.496\linewidth]{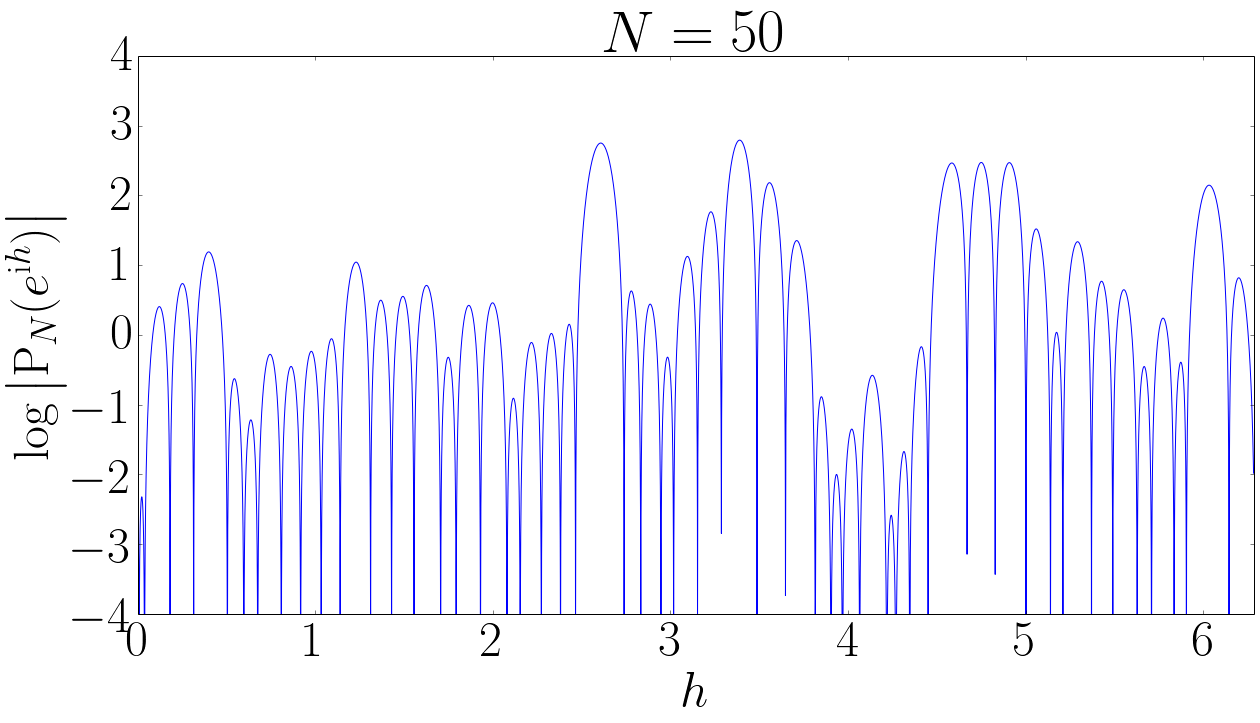}
\includegraphics[width=0.496\linewidth]{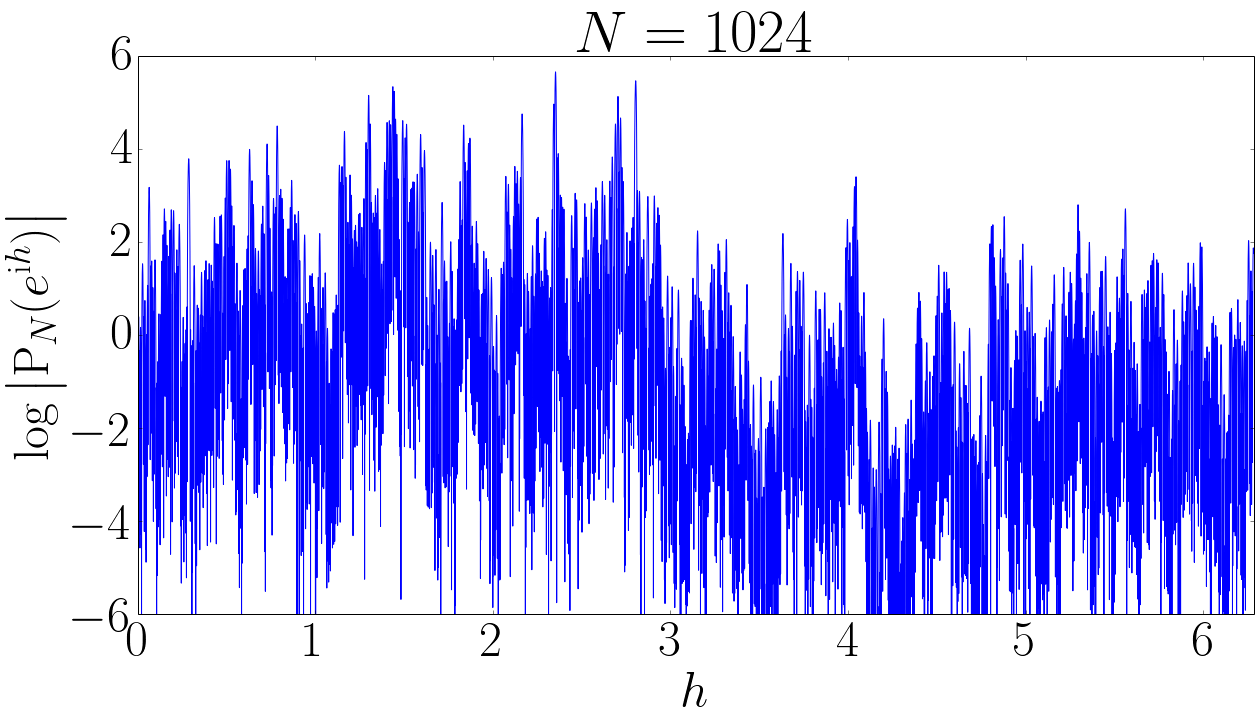}
\caption{Realizations of $\log |{\rm P}_N(e^{\ii h})|$, $0 \le h<2\pi$, for $N=50$ and $N=1024$.
 At microscopic scales, the field is smooth away from the eigenvalues, 
in contrast with the rugged landscape at mesoscopic and macroscopic scales.}
\end{center}
\end{figure}

The proof of Theorem \ref{thm: main} is outlined in Section \ref{subsec: outline} below.
The key conceptual idea is the identification of an approximate branching random walk, or hierarchical field, in the Fourier decomposition of the characteristic polynomial.
This is inspired by a branching structure in the Euler product of the Riemann Zeta function employed in \cite{ArgBelHar15}.
In Section \ref{subsec: outline}, it is explained how branching random walk heuristics provide an alternative justification of Conjecture \ref{conj: fhk}. 
Furthermore, these heuristics can be made rigorous for the leading order, thanks to a robust approach introduced by Kistler in \cite{Kis15}.
Technical difficulties remain to rigorously verify the finer predictions of the conjecture.

It is straightforward to adapt the approach to get information about the measure of {\it high points} for $\gamma \in (0,1)$:
\begin{equation}\label{eq: high points def}
  \mathcal{L}_N(\gamma) = \{ h \in [0,2\pi] : \log | {\rm P}_{N} (e^{\ii h}) | \ge \gamma \log N \}\ .
\end{equation}
We show that with high probability the Lebesgue measure of $\mathcal{L}_N(\gamma)$ is close to $N^{-\gamma^2}$:
\begin{thm}\label{thm: high points}
For $\gamma \in (0,1)$,
\begin{equation}\label{eq: high points}
 \lim_{N \to \infty} \frac{ \log \mbox{\rm Leb}(\mathcal{L}_N(\gamma))}{ \log N} = -\gamma^2 \mbox{, in probability,}
\end{equation}
where $\mbox{\rm Leb}(\cdot)$ denotes the Lebesgue measure on $[0,2\pi]$.
\end{thm}
This was conjectured by Fyodorov \& Keating, see Section 2.4 in \cite{FyoKea14}.
In fact, a more precise expression for the measure of high points was instrumental for their prediction of the subleading order in Conjecture \ref{conj: fhk},
following the ideas of \cite{FyoLedRos12}.
The theorem can be used to obtain the limit of the {\it free energy} 
\begin{equation}
\label{eqn: free energy}
\frac{1}{\log N}\log \left(\frac{N}{2\pi}\int_0^{2\pi} |{\rm P}_{N}(h)|^{\beta} \rd h \right)
\end{equation}
of the random field $\log |{\rm P}_N( e^{\ii h})|$. 
In particular, it is proposed in Section 2.2 of \cite{FyoKea14} that the free energy exhibits {\it freezing}, i.e.~that above a critical temperature $\beta_c$,
the free energy \eqref{eqn: free energy} divided by the inverse temperature $\beta$ becomes constant in the limit. The following, which is essentially
an immediate consequence of Theorem \ref{thm: high points}, proves the conjecture.
\begin{cor}\label{cor: free energy}
For $\beta \ge 0$,
\begin{equation}
 \lim_{N\to \infty} \frac{1}{\log N}\log \left(\frac{N}{2\pi}\int_0^{2\pi} |{\rm P}_{N}(h)|^{\beta} \rd h\right)
=
\begin{cases}
1+\frac{\beta^2}{4}\ & \text{ if $\beta< 2$},\\
\beta & \text{ if $\beta\geq 2$},
\end{cases}
\ \text{ in probability }.
\end{equation}
\end{cor} 

\noindent The work \cite{FyoKea14} contains other interesting conjectures on statistics of characteristic polynomials. One of them, a transition for the second moment of the partition function, was proved in \cite{ClaKra2014}.

\subsection{Relations to Previous Works.\ }\label{sec: prev work}
This paper is part of the current research effort to develop a theory of extreme value statistics of log-correlated fields.
There have been many rigorous works on the subject in recent years, and we give here a non-exhaustive list.
In the physics literature,  most predictions on the extreme value statistics of log-correlated fields can be found in\ \cite{CarLed2001}.
In mathematics, the leading order  of the two-dimensional Gaussian Free Field, was determined in \cite{BolDeuGia01}. 
In a series of impressive work, the form of the subleading correction as well as convergence of the fluctuations have been obtained
\cite{BramZei12, DinZei14, BramDinZei13, BisLou13}.
The approach (with the exception of \cite{BisLou13}) follows closely the one used for branching random walks. This started with the seminal work of Bramson \cite{Bra78} for branching Brownian motion and was later extended to general branching random walks \cite{Bac00, AddRee09, BraZei09, Aid13, BraDinZei14}.
Log-correlated models are closely related to Gaussian Multiplicative chaos, see \cite{RhoVar13} for a review.
In particular, convergence of the maximum of a related model of log-correlated Gaussian field was proved in \cite{Mad14}.
We also refer to \cite{Web15} for connections between the characteristic polynomial of unitary matrices and Gaussian Multiplicative chaos.
From the perspective of spin glasses, Corollary \ref{cor: free energy} suggests that the model exihibits a {\it one-step replica symmetry breaking}.
This was proved for Gaussian log-correlated fields in \cite{DerSpo88,BovKur04,ArgZin14, ArgZin15}.
A general theorem for the convergence of the maximum of log-correlated Gaussian fields was proved in \cite{DinRoyZei15}.
A unifying point of view including non-Gaussian log-correlated fields and their hierarchical structure is developed in \cite{Kis15}.
Important non-Gaussian examples include cover times of the two-dimensional random walk on the torus whose leading order was determined in \cite{DemPerRosZeo04} and subleading order in \cite{BelKis14}. 
Also, the leading and subleading order of the Fyodorov-Hiary-Keating Conjecture are known for a random model of the Riemann zeta function other than CUE \cite{Har13,ArgBelHar15}.
Finally, the analogue conjecture is expected to hold for other random matrix ensembles such as the {\it Gaussian Unitary Ensemble} \cite{FyoSim2015}.\\

\noindent{\it Notation.} Throughout this paper, we use the notation ${\rm O(1)}$ (resp. ${\rm o}(1)$) for a quantity uniformly bounded in $N$ (resp. going to $0$ with $N$). The constants $c$ and $C$ denote universal constants varrying from line to line. The notation $a_N\lesssim b_N$ means that $a_N\leq C b_N$ for some $C$ independent of $N$.\\
 
\subsection{Outline of the Proof: Connection to Branching Random Walk.\ }\label{subsec: outline}
Let $e^{\ii\theta_1}, \ldots, e^{\ii\theta_N}$ be the eigenvalues of ${\rm U}_N$. We are interested in
$$
\log|{\rm P}_N(e^{\ii h})| = \sum_{k=1}^N \log| 1 - e^{\ii(\theta_k-h)} |.
$$
Recall that an integrable $2\pi$-periodic function has a Fourier series which converges pointwise wherever
the function is differentiable (see e.g. \cite[Theorem 2.1]{SteSha03}).
Since the $2\pi$-periodic integrable function $h \mapsto \Re  \log(1-e^{\ii h})$ has Fourier series $- \sum_{j=1}^\infty \frac{\Re\, e^{-\ii j h}}{j}$
we have
\begin{equation}
\label{eqn: fourier}
\log|{\rm P}_N(e^{-\ii h})| = \sum_{k=1}^N \sum_{j=1}^\infty -\frac{\Re (e^{\ii j(\theta_k-h}))}{j}
= \sum_{j=1}^\infty -\frac{\Re (e^{-\ii j h}\Tr {\rm U}_N^j)}{j}, \ \ h\in\mathbb{R}, 
\end{equation}
where $\Tr$ stands for the trace and both right and left-hand sides are interpreted as $-\infty$ if $h$ equals an eigenangle.
The starting point of the approach is to treat the above expansion as a {\it multiscale decomposition} for the process.

Though the traces of powers of ${\rm U}_N$ are not independent, it was shown in \cite{DiaSha1994,DiaEva01} that they
are {\it uncorrelated}
\begin{equation}
\label{eqn: diaconis-shah}
\mathbb{E}\left(\Tr {\rm U}_N^j\ \overline{\Tr {\rm U}_N^k}\right)=\delta_{kj} \min(k,N),
\end{equation}
where $\mathbb{E}$ is the expectation under the Haar measure $\mathbb{P}$ (by rotational invariance also $\mathbb{E}(\Tr {\rm U}_N^j \Tr {\rm U}_N^k)=0$).
At a heuristic level, the covariance structure of the traces explains the asymptotic Gaussianity of $\log |{\rm P}_N(e^{\ii h})|$
as well as the correlation structure for different angles $h_1,h_2$,  see \eqref{eqn: branching 1} below.
It is also the starting point of the connection to branching random walk.

Because of \eqref{eqn: diaconis-shah}, one expects that the contribution to $\log |{\rm P}_N(e^{\ii h})|$ of traces of powers $N$ or greater should be of order $1$ since
$\sum_{j \ge N} \frac{N}{j^2}=\OO(1)$. Moreover, the variance of the powers less than $N$ becomes
$$
\mathbb{E}\left( \left(\sum_{j=1}^{N-1} -\frac{\Re (e^{-\ii jh}\Tr {\rm U}_N^j)}{j}\right)^2\right) = \frac{1}{2}\sum_{j < N}\frac{1}{j}= \frac{1}{2} \log N + \OO(1).
$$
The key idea is to divide the truncated sum into increments
\begin{equation}\label{eq: W incr def}
W_\ell(h)=\sum_{e^{\ell-1} \le j < e^\ell} -\frac{\Re (e^{-\ii j h}\Tr {\rm U}_N^j)}{j} \ \text{ for $\ell=1,\dots, \log N$,}
\end{equation}
which, thanks to \eqref{eqn: diaconis-shah}, are uncorrelated and have variance $\frac{1}{2}\sum_{e^{\ell-1} \le j < e^\ell}\frac{1}{j} \approx \frac{1}{2}$. 
Thus, at a heuristic level, one may think of the partial sums
\begin{equation}\label{eqn: X brw}
X_\ell(h) = \sum_{\ell'=1}^{\ell} W_{\ell'}(h), \ \ell=0,\ldots, \log N,
\end{equation}
as a Gaussian random walk with increments of variance $\frac{1}{2}$, for any fixed $h$. 
Furthermore, for each $\ell$, the collection $(X_{\ell}(h),h\in[0,2\pi])$, defines a random field on the unit circle which can be thought
of as a sequence of regularizations of $\log |{\rm P}_N(e^{ih})|$.
It turns out that for $h$ and $h'$ in $[0,2\pi]$ the corresponding partial sums $X_\ell(h)$ and $X_\ell(h')$ exhibit an approximate {\it branching structure}.
To see this, define the {\it branching scale} of $h$ and $h'$ 
as
\begin{equation}
\label{eqn: branching scale}
h\wedge h'=-\log \|h-h'\|\ ,
\end{equation}
where  $\|h-h'\|=\min\{|h-h'|, (2\pi-|h-h'|)\}$ is the distance on the circle.
Equation \eqref{eqn: diaconis-shah} implies
\begin{equation}
\label{eqn: branching 1}
\mathbb{E}(W_\ell(h)W_\ell(h'))=\sum_{e^{\ell-1} \le j < e^\ell}\frac{\cos(j\|h-h'\|)}{2j}.
\end{equation}
For $j$ where $j \|h-h'\|$ is small the cosine is essentially $1$, and for $j$ such that $j \|h-h'\|$
is large it oscillates, causing cancellation. 
In fact, by expanding the cosine in the first case and by using summation by parts in the other, 
it is not hard to see that
\begin{equation}
\label{eqn: branching 2}
\mathbb{E}(W_\ell(h)W_\ell(h'))=
\begin{cases}
\frac{1}{2} +\OO(e^{\ell-h\wedge h'}) \ &\text{if $\ell\leq h\wedge h'$,}\\
\OO(e^{-2(\ell-h\wedge h')}) 			\ &\text{if $\ell> h\wedge h'$.}\\
\end{cases}
\end{equation}
In other words, the increments $W_\ell(h)$ and $W_\ell(h')$ are almost perfectly correlated for $\ell$ before the branching scale and almost perfectly uncorrelated after the branching scales.
We conclude from \eqref{eqn: branching 2} that, if we restrict the field to the discrete set of $N$ points
\begin{equation}\label{eqn: H_N}
  \mathcal{H}_N = \left\{0, \frac{2\pi}{N}, 2\frac{2\pi}{N}, \dots, (N-1) \frac{2\pi}{N}\right\}\ ,
\end{equation}
the process $\big((X_\ell(h), l\leq \log N), h\in \mathcal H_N\big)$ defined by \eqref{eqn: X brw} is an {\it approximate branching random walk}.
Namely, the random walks $X_\ell(h)$ and $X_\ell(h')$ are almost identical before the branching scale $\ell = h \wedge h'$,
and continue almost independently after that scale, akin to a particle of a branching random walk that splits into two independent walks at time $h \wedge h'$, see Figure \ref{fig: splitting}.
Moreover, if $\| h - h' \| \le e^{-\ell}$, the variance of the difference $X_{\ell}(h)-X_{\ell}(h')$ is of order one. 
Thus, the
``variation'' of $(X_{\ell}(h), h\in[0,2\pi])$ is 
effectively captured by the values at $e^{\ell}$ equally spaced points for each $\ell$.
This is reminiscent of a branching random walk where the mean number of offspring of a particle is $e$ and the average number of particles at time $\ell$ is $e^{\ell}$, see Figure \ref{fig: tree}.

Keeping the connection with branching random walk in mind, the proof of Theorem \ref{thm: main}  is carried out in two steps.
First, we obtain upper and lower bounds for truncated sums restricted to the discrete set $\mathcal H_N$.
For this, we follow  a multiscale refinement of the second method proposed by Kistler \cite{Kis15}.
The second step is to derive from these upper and lower bounds for the entire sum (including large powers of the matrix) over the whole continuous interval $[0,2\pi]$.\\

\begin{figure}
\begin{center}
\vspace{-0.9cm}
\includegraphics[width=0.5\linewidth]{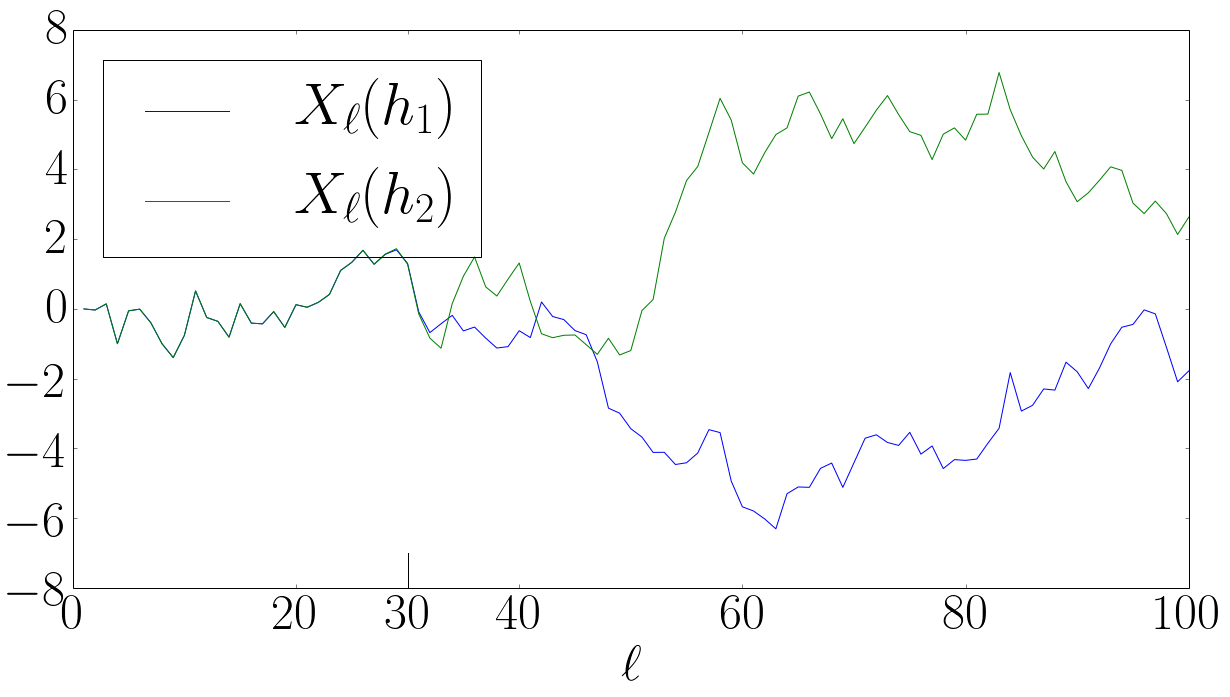}
\end{center}
\caption{Illustration of the processes $X_{\ell}(h_1)$ and $X_{\ell}(h_2)$ for $h_1$ and $h_2$ at distance roughly $e^{-30}$, for a random matrix ${\rm U}_N$ with side-length $N = e^{100}$.
The processes decorrelate at roughly scale $\ell = 30$.  }
\label{fig: splitting}
\end{figure}

\begin{figure}[h]
\begin{center}
\includegraphics[height=4cm]{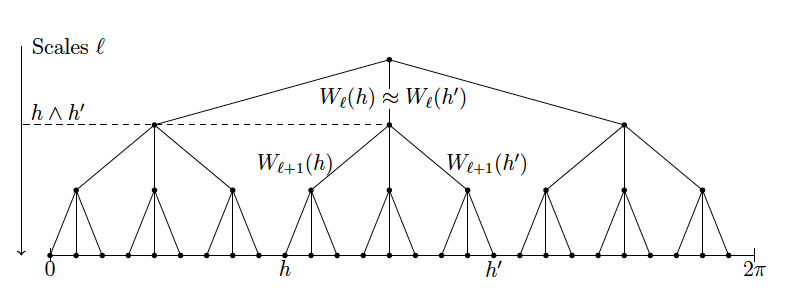}
\caption{Illustration of the approximate branching random walk structure of the multiscale decomposition with $\approx e^\ell$ distinct values at each scale $\ell$.}
\label{fig: tree}
\end{center}
\end{figure}

\noindent{\it First step: the truncated sum on a discrete set.} 
The first result is an upper and lower bound for the maximum of the truncated sum for powers slightly smaller than $N$:
$$
X_{(1-\delta)\log N}(h)=\sum_{ 1 \le \ell < (1-\delta)\log N }W_\ell(h)=\sum_{j < N^{1-\delta}} - j^{-1}\Re (e^{-\ii j h}\Tr {\rm U_N^j}), \ h\in \mathcal H_N\ ,
$$
where the parameter $\delta$ will be taken small enough in terms of $\varepsilon>0$. 
For a given $\varepsilon$, we prove
\begin{equation}\label{eq: UB LB}
\lim_{N\to\infty} \mathbb{P}\left((1-\varepsilon)\log N \le \max_{h \in \mathcal H_N} \sum_{1 \le \ell < (1-\delta)\log N}W_\ell(h) \le \log N \right) = 1\ .
\end{equation}
The point of restricting to powers smaller than $N$ is that for such sums we can obtain sharp large deviation estimates. 

The upper estimate \eqref{eq: UB LB} follows by a union bound
\begin{equation}
\label{eqn: union}
\mathbb{P}\left( \max_{h\in\mathcal H_N}X_{(1-\delta)\log N}(h) > \log N \right)\leq N\ \mathbb{P}\left( X_{(1-\delta)\log N}(h) > \log N \right)
\leq\E(\mathcal Z(0))\ , 
\end{equation}
where
$\mathcal Z(\varepsilon)$ 
is the number of exceedances
\begin{equation}\label{eq: exceedances intro}
\mathcal Z(\varepsilon)=\#\{h\in\mathcal H_N: X_{(1-\delta)\log N}(h) > (1-\varepsilon)\log N\}\ . 
\end{equation}
The expectation $\E(\mathcal Z(0))$ 
goes to zero if $X_{(1-\delta)\log N}(h)$, whose variance is approximatively $(1-\delta)\log N$, admits Gaussian large deviations. 
This is proved by computing the exponential moments of $X_\ell(h)$ using a Riemann-Hilbert approach,
see Proposition \ref{thm:FourierRH}. The Riemann-Hilbert approach to
compute the 
Fourier transform of linear statistics is an idea from \cite{Dei1999}.

The lower bound in \eqref{eq: UB LB}
would 
follow from Chebyshev's inequality (or Paley-Zygmund inequality) if one could show 
$\E(\mathcal Z(\varepsilon)^2) = (1+\oo(1)) \E(\mathcal Z(\varepsilon))^2$.  
However, as for branching random walk, the second moment
$\E(\mathcal Z(\varepsilon)^2)$ 
is in fact exponentially larger than
$\E(\mathcal Z(\varepsilon))^2$ 
, due to rare events.
A way around this is to modify the count by introducing a condition which takes into account the branching structure.
At the level of the leading order, this can be achieved by a {\it $K$-level coarse graining} as explained in \cite{Kis15}.
More precisely, for $K\in\N$ and $\delta=K^{-1}$, consider $K$ large increments of the ``random walk" $X_\ell(h)$:
\begin{equation}
\label{eqn: Y}
  Y_m(h) = \sum_{ \frac{m-1}{K}\log N<\ell\leq \frac{m}{K}\log N} W_\ell(h) = \sum_{ N^{ \frac{m-1}{K} } < j \le N^{ \frac{m}{K} } } -\frac{\Re( e^{-\ii j h}\Tr( {\rm U}_N^j))}{j}\ ,
\end{equation}
for $m=1,\dots, K$, so that $X_{(1-\delta)\log N}(h)=Y_1(h)+\dots+Y_{K-1}(h)$.
It follows from \eqref{eqn: branching 2} that each $Y_m(h)$ has variance roughly $\frac{1}{2K}\log N$. 
Moreover, $Y_m(h)$ and $Y_m(h')$  decorrelate for $\|h-h'\|\gg N^{-(m-1)/K}$.
Because of the self-similar nature of branching random walk, the leading order should be linear in the scales $\ell$.
In particular, for $h$ such that $X_{(1-\delta)\log N}(h)$ is at least $(1-\delta)\log N$, one expects each $Y_m(h)$ to contribute $\frac{1-\delta}{K}\log N$ to the maximum.
Following \cite{Kis15}, this leads us to consider the modified count
\begin{equation}
\label{eqn: Z1}
\tilde{\mathcal{Z}}(\varepsilon) = \#\{h\in \mathcal H_N:Y_m(h)\geq \frac{(1-\varepsilon)}{K}\log N, m=2,\dots, K-1 \}\ . 
\end{equation}
Note that the first increment $Y_1(h)$ is omitted. This is crucial since for
\begin{equation}\label{eq: fm sm intro}
\E(\mathcal{\tilde{Z}}(\varepsilon)^2) = (1+\oo(1))\E(\mathcal{\tilde{Z}}(\varepsilon))^2\ , 
\end{equation}
to hold a sufficiently large proportion of all pairs $h_1$ and $h_2$ must have decorrelated increments,
and the first increments $Y_1(h_1)$ and $Y_1(h_2)$ only decorrelate for macroscopically separated pairs, cf.~\eqref{eqn: branching 2}.
Dropping the first increments ensures enough independence between most pairs $h_1,h_2$ for the moments to match.
Moreover, a separate first moment argument shows that this comes at negligible cost if $K$ is chosen large enough depending on $\varepsilon$.
To prove the estimate \eqref{eq: fm sm intro}, one also needs precise large deviation estimates for the vector $(Y_2(h),\ldots, Y_{K-1}(h))$
for one $h$ and jointly for two different $h$'s. 
Again, such precise estimates on the exponential moments are obtained through the Riemann-Hilbert approach, see  Proposition \ref{thm:FourierRH}.

At this point, a few words on the prediction for the subleading correction are in order.
At that level of precision, the expected number of points exceeding the predicted level $m_N=\log N - 3/4 \log \log N$ diverges in contrast to \eqref{eqn: union}.
However, thanks to the seminal work of Bramson on branching Brownian motion \cite{Bra78}, it is known that
the first moment is inflated by rare events: with large probability, the random walks $X_\ell(h)$ must remain below a linear barrier for the terminal value to be maximal.
To capture this, one is led to consider exceedances with a barrier along the scales
\begin{equation}\label{eq: heur Z}
\tilde{\mathcal{Z}}_B = \#\{ h \in \mathcal H_N: X_{\log N}(h) \ge m_N,\ X_{\ell}(h) \leq \ell  + B \text{ for $\ell=1,\dots, \log N$}\},
\end{equation}
for some suitable $B$. If $(X_\ell(h), \ell\leq \log N)$ were an exact random walk, then by the {\it ballot theorem} (see e.g.~\cite{AddRee08}),
the additional barrier requirement would balance out the divergence so that $\E(\tilde{\mathcal{Z}}_B)=\OO(1)$.
In fact, the restriction on the increments in \eqref{eqn: Z1} is a weak version of this barrier condition:
if a random walk $X_{\ell}(h)$ goes above the linear barrier for small $\ell$, 
its ``descendants'' at subsequent times cannot take advantage of this
by rising at a slower rate, since
to be counted the increments $Y_2,\ldots,Y_{K-1}$ must all make equally large jumps.\\

\noindent{\it Second step: extension to the full sum and the continuous interval. }
Once the result \eqref{eq: UB LB} is established for the truncated sum $X_{(1-\delta)\log N}(h)$ on the discrete set $\mathcal H_N$, it remains to include
traces of high powers to the sum and to consider the full interval $h\in[0,2\pi]$.

For the lower bound, the extension to the full interval is direct.
Controlling the contribution of high powers is harder: it requires large deviation
bounds for $-\Re\sum_{j \ge N^{1-\delta}}\frac{{\rm Tr}(e^{-\ii j h}U_N^j)}{j}$. These come from exponential moment estimates
$$
\mathbb{E}\left(\exp\left(-\alpha\Re\sum_{j \ge N^{1-\delta}}\frac{{\rm Tr}(e^{-\ii j h}U_N^j)}{j}\right)\right)
=
\mathbb{E}\left(\exp\left(\sum_{k=1}^N v(e^{\ii\theta_k})\right)\right),\  {\rm where\ } v(e^{\ii\theta})=-\alpha\Re\sum_{j \ge N^{1-\delta}}\frac{e^{\ii j(\theta-h)}}{j}.
$$
Obtaining a
precise approximation of the above expectation is equivalent to
deriving asymptotics of Toeplitz determinants with a singularity of Fisher-Hartwig type. 
Indeed,
by Heine's formula \cite{Hei1878} we have
\begin{equation}\label{eqn:Heine}
\mathbb{E}\left(\prod_{k=1}^N f(e^{\ii\theta_k})\right)
=
\det((\hat f_{i-j})_{0\leq i,j\leq N})\ {\rm with\ }
f(z)=e^{V(z)}|z-1|^{\alpha},\ V(z)=\alpha\Re \sum_{1 \le k < N^{1-\delta}}\frac{z^k}{k}.
\end{equation}

In other words, we study the distribution of high powers of unitary matrices by superposition of 
a logarithmic singularity with smooth linear statistics.
For a fixed analytic external potential $V$ the Riemann-Hilbert technique from \cite{DeiItsKra2011,DeiItsKra2014}
is a robust method to estimate asymptotics of the above Toeplitz determinant. 
We follow the same approach with an additional technicality: in our regime of interest, $V$ depends on $N$
 with an emerging singularity {\it  of arbitrarily small mesoscopic scale} $N^{-1+\delta}$ (we note that the Riemann-Hilbert method was also used in \cite{FyoKhoSim16} to evaluate determinants associated with merging singularities, in the context of the Gaussian Unitary Ensemble).
The resulting Laplace transform of traces of high powers of unitary matrices, Proposition \ref{lem: mgf ub tail}, may be of independent interest.

For the upper bound, the extension to the full sum and interval is obtained as follows.
First, we obtain an upper bound on the truncated sum over all of $[0,2\pi]$ using a dyadic chaining argument.
This relies on a large deviation control on the difference of values of the truncated sum between two close points (at distance less than $N^{-1}$) obtained from exponential moments via Proposition \ref{thm:FourierRH}.
Second, we control the large values of the traces of high powers 
on a dense discrete set in the same way as for the lower bound (using the aforementioned Proposition \ref{lem: mgf ub tail}).
Combining these, we obtain an upper bound for the full sum on the dense discrete set.
In the last step a bound on small gaps between eigenvalues \cite{BenBou2013} is used,
which gives an appropriate control of the derivative of $\log |{\rm P}_N(e^{\ii h})|$ on the circle, and allows us to conclude
that the maximum over $[0,2\pi]$ is close to the maximum over the dense discrete set.

\subsection{The imaginary part.\ }\label{subsec:imaginary} Theorem \ref{thm: main} has a natural analogue concerning the imaginary part of the logarithm of the characteristic polynomial. More precisely,  if we choose the principal branch 
so that $-\pi/2<{\rm Im}\log (1-e^{\ii \theta})<\pi/2 $, then 
\begin{equation}\label{eqn:imaginary}
\lim_{N\to\infty}\frac{\underset{h\in[0,2\pi]}{\max} {\rm Im} \sum_{k=1}^N\log (1-e^{\ii (\theta_k-h)})}{\log N}=1 \qquad \text{in probability.}
\end{equation}
To illustrate the meaning of this result, let $\mathscr{N}_h({\rm U}_N)=\sum_{k=1}^N\mathds{1}_{(0,h)}(\theta_k)$
denote the number of eigenangles in the range $(0,h)$. 
Following \cite{HugKeaOco2001} for any $-\pi<s<t<\pi$ we have the identity 
$
\mathds{1}_{(s,t)}(\theta)=\frac{t-s}{2\pi}+\frac{1}{\pi}{\rm Im}\log (1-e^{\ii(\theta-t)})-\frac{1}{\pi}{\rm Im}\log (1-e^{\ii(\theta-s)})
$, hence (\ref{eqn:imaginary}) yields
$$
\lim_{N\to\infty}\frac{\underset{h\in[0,2\pi]}{\max}(\mathscr{N}_h({\rm U}_N)-\frac{N h}{2\pi})}{\log N}=\frac{1}{\pi}
\qquad \text{in probability.} 
$$
This easily implies the following optimal rigidity bound for all eigenangles. Note that, in the context of Wigner matrices,
the best known rigidity estimates (for all eigenvalues simultaneously) in the bulk of the spectrum are of type $(\log N)^C/N$ for some non optimal constant $C$ (see e.g. \cite{ErdYauYin2012Rig}, \cite{CacMalSch2015}).  The following gives both the optimal logarithmic exponent and constant
for the CUE. 

\begin{thm}
We label the eigenangles of a Haar-distributed ${\rm U}_N$ so that $0\leq \theta_1\leq\dots\leq \theta_N<2\pi$.
Then, for any $\e>0$, we have
$$
\mathbb{P}\left((2-\e)\frac{\log N}{ N}< \sup_{1\leq k\leq N}\left|\theta_k-\frac{2\pi k}{N}\right|< (2+\e)\frac{\log N}{ N}\right)\underset{N\to\infty}{\to}1.
$$
\end{thm}

For the proof of (\ref{eqn:imaginary}), note that the above choice of principal branch for the logarithm coincides with the Fourier expansion: ${\rm Im}\log (1-e^{\ii \theta})=-{\rm Im}\sum_{k\geq 1}\frac{e^{\ii k\theta}}{k}$, so that our branching technique applies. All elements of the proof for the real part only require notational changes when considering the imaginary part, except Proposition \ref{prop:tail} which bounds the tail of the Fourier series. We omit the straightforward proof adjustments, and we
give in Section \ref{sec: RH} the proof of the tail estimate, for the imaginary part of the series, i.e. the counterpart of Proposition \ref{prop:tail} (see Proposition \ref{prop:tailIm}).

\vspace{0.2cm}

\noindent{\bf Acknowledgements.} The authors thank Nicola Kistler for insightful discussions on the subject of log-correlated fields. L.-P. A. 
is partially supported by NSF grant DMS-1513441, PSC-CUNY Research Award 68784-00 46 and a Eugene M. Lang Junior Faculty Research Fellowship. P. B. is partially supported by NSF grants DMS-1208859 and DMS-1513587. 

\section{Maximum of the truncated sum on a discrete set}
As mentioned in the introduction, we first study the
maximum over the set $\mathcal{H}_N$ (see \eqref{eqn: H_N}) of cardinality $N$.
Precise large deviation bounds are a crucial input for the method.
We have good control on the { large deviations} of the (untruncated) sum
\begin{equation}\label{eq: single sum}
\sum_{j=1}^{\infty} -\frac{\Re(e^{-\ii jh}\Tr{\rm U}_N^j)}{j},\mbox{ for }h \in \mathbb{R},
\end{equation}
since its { characteristic function} can be explicitly computed using the Selberg integral { (cf. Lemma \ref{lem:Selberg})}.
However, the proof of the lower bound, being a second moment calculation, requires precise joint large deviation bounds
for sums at two different $h$'s. Computing exponential moments of two such sums, 
which include traces of powers close to or larger than $N$, is difficult due to the singularity of the logarithm.
If the sum only contains traces with powers less than $N$,
the Riemann-Hilbert techniques of Section \ref{sec: RH} give the needed { characteristic function (and exponential moment)} bounds.
The upper and lower bounds in this section are therefore given for the maximum of the truncated sum
\begin{equation}\label{eq: single sum trunc}
\sum_{\ell=1}^{(1-\delta)\log N} W_\ell(h) = \sum_{1 \le j < N^{1-\delta}} -\frac{\Re(e^{-\ii jh}\Tr{\rm U}_N^j)}{j},
\end{equation}
for small $\delta>0$, where the ``random walk" increments $W_\ell(h)$ are defined in \eqref{eq: W incr def}.
In Section \ref{sec: chaining etc} we extend these to bounds where $\mathcal{H}_N$ is replaced with the full interval $[0,2\pi]$,
and the sum \eqref{eq: single sum trunc} is replaced by the full sum.

\subsection{Upper bound.\ }
An upper bound for the maximum of the truncated sum \eqref{eq: single sum trunc} on $\mathcal{H}_N$ is obtained by a straightforward union bound.
We use the following ``Gaussian'' exponential moment bound for the sum \eqref{eq: single sum trunc},
which is proved in Section \ref{sec: RH} using Riemann-Hilbert techniques.
\begin{lem}
\label{lem: gaussian UB}
  For any $\delta \in (0, 1)$, there is a constant $C$ such that for $h,\xi \in \mathbb{R}$ with $|\xi| \le N^{\delta/10}$, then
  \begin{equation}\label{eq: exp mom of trunc sum}
    \E \left( \exp\left( \xi \sum_{\ell=1}^{(1-\delta) \log N} W_\ell(h) \right) \right)
    \le C\exp \left( \frac{1}{2} \xi^2 \sigma^2 \right),\  \text{ where $\sigma^2 = \frac{1}{2}\sum_{j=1}^{N^{1-\delta}} \frac{1}{j}$.}
  \end{equation}
\end{lem}
This implies a ``Gaussian'' tail bound using the exponential Chebyshev inequality with $\xi=\pm x/\sigma^2$.
\begin{equation}\label{eq: one point trunc tail bound}
 \P\left( \Big| \sum_{\ell=1}^{(1-\delta) \log N} W_\ell(h)\Big| \ge x \right) \le C\exp\left( -\frac{x^2}{2\sigma^2} \right),\ \text{for all $x \le C \sigma^2 $}.
\end{equation}
We thus arrive at the following:
\begin{prop}[Upper bound for truncated sum over discrete set]\label{prop: ub disc trunc}
 For any $\delta>0$,
  \begin{equation}\label{eq: upper bound discrete} 
    \lim_{N \to \infty} \P\left(\max_{ h \in \mathcal{H}_N } \sum_{\ell=1}^{(1-\delta) \log N} W_\ell(h) \le \log N \right) = 1. 
  \end{equation}
\end{prop}
\begin{proof}
Since $\sigma^2 = (1-\delta) \frac{1}{2}\log N + \OO(1)$,  a union bound using \eqref{eq: one point trunc tail bound} gives
\begin{equation}\label{eq: union bound}
\begin{aligned}
  \P\left(\max_{ h \in \mathcal{H}_N } \sum_{\ell=1}^{(1-\delta) \log N} W_\ell(h)\ge  \log N \right) 
  &\le N\ \P\left(  \displaystyle{ \sum_{\ell=1}^{(1-\delta) \log N} W_\ell(h) } \ge  \log N \right)\\
  &\le N \exp\left( - (1+\oo(1))\frac{ (\log N)^2 }{(1-\delta)\log N} \right) \le N^{ -\delta + \oo(1) } = \oo( 1)\ .
\end{aligned}
\end{equation}
\end{proof}

\subsection{Lower bound.\ }
As described in the introduction, we use a truncated second moment argument to prove th following result.
\begin{prop}[Lower bound for truncated sum over discrete set]\label{prop: lb trunc disc} 
For any $\varepsilon>0$, there exists $\delta=\delta(\varepsilon)>0$ such that
\begin{equation}\label{eq: lb trunc disc}
  \lim_{N\to\infty} \P\left( \max_{ h \in \mathcal{H}_N} \sum_{ \ell=1 }^{(1-\delta)\log N} W_\ell(h) \geq (1-\varepsilon) \log N \right)= 1\ .
\end{equation}
\end{prop}
To formulate the truncation, recall the coarse increments $Y_1(h),\ldots,Y_K(h)$, $K\in\mathbb N$, defined in \eqref{eqn: Y}.
Note that by \eqref{eqn: diaconis-shah} these increments have variance
\begin{equation}
 \sigma_m^2   = \E( Y_m(h)^2 ) = \frac{1}{2}\sum_{ N^{ (m-1)/K} < j \le N^{ m/K  } } \frac{1}{j} = \frac{1}{2}\frac{\log N}{K} + \OO(N^{-(m-1)/K}), \qquad \text{$\forall h$,}\\
 \label{eq: Y variance}
 \end{equation}
 and, more generally, covariance
 \begin{equation}
 \rho_m(h_1,h_2) = \E( Y_m(h_1)Y_m(h_2) ) = \frac{1}{2} \sum_{ N^{ (m-1)/K } < j \le N^{ m/K  }} \frac{\cos(j\|h_1-h_2\|)}{j},  \qquad \text{$\forall h_1,h_2$}\ .
 \label{eq: Y covariance}
 \end{equation}
Expanding $e^{-\ii jh}$ for large $h_1 \wedge h_2$ and summing by parts for small $h_1 \wedge h_2$, one arrives at the estimate
\begin{equation}\label{eq: Y covariance bound}
  \rho_m(h_1,h_2) =  
  \begin{cases}    
    \OO(N^{-(m-1)/K} e^{h_1 \wedge h_2})  	&\text{ if $ h_1 \wedge h_2 \leq (m-1)\frac{\log N}{K}$,} \\
    \sigma_m^2 + \OO \left( \left( N^{m/K}e^{-h_1\wedge h_2} \right)^2 \right) 	&\text{ if $h_1 \wedge h_2 \geq m\frac{\log N}{K}$}\ .
  \end{cases}
\end{equation}
Therefore, unless $(m-1) \frac{\log N}{K} \le h_1 \wedge h_2 \le m \frac{\log N}{K}$, the increments $Y(h_1), Y(h_2)$ are almost completely correlated or completely decorrelated.

The second moment method is applied to the counting random variable
\begin{equation}
\label{eqn: Z}
 Z = \sum_{ h \in \mathcal{H}_{N} } \1_{J_x(h)}, \qquad J_x(h)=\{Y_m(h)\geq x, m=2,\dots, K-1\}\ .
\end{equation}
The level $x$ needs to be picked appropriately. For a given $\varepsilon>0$, take
\begin{equation}
\label{eqn: x}
x=(1-\varepsilon/2)\frac{1}{K}\log N\ .
\end{equation}
Proposition \ref{prop: lb trunc disc} is a simple consequence of the following, which will be proved in the remainder of the section.
\begin{prop}\label{prop: fm sm}
For any $K \ge 3$ and $\varepsilon>0$,
\begin{equation}\label{eq: fm sm}
\E\left(Z^2 \right) = (1+\oo_K( 1)) \E(Z)^2, \mbox{ as } N\to \infty.
\end{equation}
\end{prop} 
Here and throughout this section a subscript $K$ on a $\oo_K \left( \cdot \right)$ or $\OO_K\left( \cdot \right)$ term denotes that all constants inside those terms may depend on $K$.
\begin{proof}[Proof of Proposition \ref{prop: lb trunc disc}]
Take $\delta=K^{-1}$. On the event $\{Z\geq 1\}$, there is an $h \in \mathcal{H}_{N}$ such that
\begin{equation}
\label{eqn: sum Y}
  \sum_{ \log N/K < \ell \le (1-1/K)\log N } W_\ell(h) \geq (1-\varepsilon/2)(1-2/K)\log N.
\end{equation}
Together with the Paley-Zygmund inequality
\begin{equation}
\label{eqn: PZ}
\P\left(Z\geq 1\right)\geq \frac{\E\left(Z\right)^2}{\E\left(Z^2\right)}\ ,
\end{equation}
and Proposition \ref{prop: fm sm}, this implies that for any $\varepsilon>0$ and $K \ge 3$,
$$
  \lim_{N\to\infty} \P\left( \max_{h\in\mathcal H_N}\sum_{\log N/K < \ell \le (1-1/K)\log N} W_\ell(h) \geq (1-\varepsilon/2)(1-2/K) \log N \right) = 1.
$$
A union bound as in \eqref{eq: union bound} for $\ell$ up to $K^{-1}\log N$ implies that for $K$ large enough
$$ \lim_{N\to\infty}\mathbb{P}\left( \sum_{\ell=1}^{\log N/K}W_\ell(h) > - \frac{\varepsilon}{3} \log N \mbox{ for all } h\in \mathcal{H}_N \right) = 1\ .$$
The result follows by taking $K$ large enough in terms of $\varepsilon$. 
\end{proof}

The rest of the section is devoted to proving \eqref{eq: fm sm}. The first and second moments can be written as
\begin{equation}
  \E(Z ) = N \P(J_x(0) ) \mbox{ and } \E( Z^2 )= \sum_{h_1,h_2 \in \mathcal{H}_N} \P(J_x(h_1)\cap J_x(h_2) ).
\end{equation}
It suffices to find a lower bound on $\P( J_x(0) )$ and an upper bound on $\P( J_x(h_1)\cap J_x(h_2) )$. 
Exponential moments of the increments $Y_m(h)$ and the 
exponential Chebyshev inequality 
do not yield bounds precise enough to match $\E( Z^2 )$ to $\E(Z)^2$ up to a multiplicative constant tending to one.
In fact, it is necessary to go beyond the level of precision of large deviations
at least for the pairs $h_1$, $h_2$ that contribute the most to $E(Z^2)$, namely points that are
close to being macroscopically separated.
This is done using characteristic function bounds together with Fourier inversion. The Riemann-Hilbert
techniques of Section \ref{sec: RH} can be used to obtain the following bounds on the characteristic function.
\begin{lem}\label{lem: two point fourier bound}
Let $K\ge1$, $h_1,h_2 \in \mathbb{R}$ and write $\bm{Y}_m=\left(Y_{m}\left(h_{1}\right),Y_{m}\left(h_{2}\right)\right)$ for $m=2,\dots, K-1$.
For all ${\bm \xi}_m \in \mathbb{C}^2, m=2,\ldots,K-1$, with $\| {\bm \xi} \| \le N^{1/(10K)}$ we have
$$
\E\left(\exp\left( \sum_{m=2}^{K-1} {\bm \xi}_m \cdot  \bm{Y}_m \right)\right)
=\left(1+\OO(e^{-N^{1/(10K)}})\right)\exp\left(\frac{1}{2} \sum_{m=2}^K \bm \xi_m\cdot \bm \Sigma_m \bm \xi_m\right)\ ,
$$
where	
\begin{equation}
\label{eqn: Sigma}
\bm \Sigma_m=\left(
\begin{array}{cc}
\sigma_m^2 & \rho_m\\
\rho_m & \sigma_m^2
\end{array}
\right) \ , m=2,\dots, K-1,
\end{equation}
for $\sigma_m^2$ and $\rho_m$ defined in \eqref{eq: Y variance}-\eqref{eq: Y covariance}.
\end{lem}

To quantitatively invert the Fourier transform, we use the following crude bound.
\begin{lem}
\label{lem: Fourier Inversion}Let $d\ge1$. There are constants $c=c\left(d\right)$
such that if $\mu$ and $\nu$ are probability measures on $\mathbb{R}^{d}$
with Fourier transforms $\hat{\mu}\left(\bm{t}\right)=\int e^{\ii\bm{t}\cdot\bm{x}}\mu\left({\rm d}x\right)$
and $\hat{\nu}\left(\bm{t}\right)=\int e^{\ii\bm{t}\cdot\bm{x}}\nu\left({\rm d}x\right)$,
then for any $R,T>0$ and any function $f:\mathbb{R}^{d}\to\mathbb{R}$
with Lipschitz constant $C$
\begin{equation}
\left|\mu\left(f\right)-\nu\left(f\right)\right|\le c\frac{C}{T}+\|f\|_{\infty}\left\{ c\left(RT\right)^{d}\|\1_{\left(-T,T\right)^{d}}\left(\hat{\mu}-\hat{\nu}\right)\|_{\infty}+\mu\left(([-R,R]^d)^{c}\right)+\nu\left(([-R,R]^d)^{c}\right)\right\} .\label{eq: fourier inversion bound}
\end{equation}
\end{lem}
\begin{proof}
This follows from the quantitative Fourier inversion estimate (11.26) of Corollary 11.5 \cite{BatRao}.
One uses a smoothing kernel $K_\varepsilon$ whose Fourier transform
is supported on $[-c\varepsilon^{-1}, c\varepsilon^{-1}]^d$ (its
existence is guaranteed by Theorem 10.1 \cite{BatRao}),
with $\varepsilon=T^{-1}$.
The quantity in the curly braces in \eqref{eq: fourier inversion bound}
is the crude upper bound for the integral
$\int_{\mathbb{R}^{d}}\left|\left(\mu-\nu\right)\ast K_{\varepsilon}\right|\left({\rm d}x\right)$ one obtains from using
the point-wise bound $|g(x)| \le cT^{d}\|\1_{\left(-T,T\right)^{d}}\left(\hat{\mu}-\hat{\nu}\right)\|_{\infty}$
on the density $g$ of $\left(\mu-\nu\right)\ast K_{\varepsilon}$,
when $x \in [-R,R]^d$, and a trivial bound for the integral over the complement $([-R,R]^d)^c$.
\end{proof}

Pairs of points $h_1$ and $h_2$ that are macroscopically (or almost macroscopically) separated are the main contribution to $\E(Z^2)$.
For such $h_1$ and $h_2$ we expect the events $J_x(h_1)$ and $J_x(h_2)$ to be essentially independent, and the bounds
\eqref{eq: two point decoupling rough}-\eqref{eq: two point precise bound} that now follow make this quantitative.
\begin{prop}[Two-point bound; Decoupling]
\label{prop: two point decoupling rough}
Let $h_{1},h_{2}\in\mathbb{R}$ be such that $h_1 \wedge h_2\leq\frac{\log N}{K}$. 
Then for $0<x\le \log N$, we have that
\begin{equation}
\label{eq: two point decoupling rough}
\begin{aligned}
\P\left(J_x(h_1)\cap J_x(h_2)\right)\le
C\exp\left(-\sum_{m=2}^{K-1}\frac{x^2}{\sigma_{m}^{2}}\right)\ .
\end{aligned}
\end{equation} 
Furthermore, if $h_1 \wedge h_2 \le \frac{\log N}{2K}$, then we have the more precise bound
\begin{equation}\label{eq: two point precise bound}
\P\left(J_x(h_1)\cap J_x(h_2)\right)
\le 
(1+ \OO_K\left(N^{-c}\right)) e^{-\sum_{m=2}^{K-1}\frac{x^2}{\sigma_{m}^{2}}}\left(\prod_{m=2}^{K-1}\eta_{0,\sigma_{m}^{2}}\left(e^{-\frac{x y}{\sigma_m^2}}\1_{[0,\infty)}(y)\right)\right)^2\ ,
\end{equation} 
where $\eta_{0,\sigma^2}$ denotes the centered Gaussian measure on $\mathbb{R}$ with variance $\sigma^2$.
\end{prop}
Before starting the proof, we note that the Gaussian expectation in parentheses satisfies
\begin{equation}
\label{eqn: eta bounds}
C(\log N)^{-1/2} \leq \eta_{0,\sigma_{m}^{2}}\left(e^{-\frac{x y}{\sigma_m^2}}\1_{[0,\infty)}(y)\right)\leq 1\ ,
\end{equation}
because of the bound
$$
\eta_{0,\sigma_{m}^{2}}\left(e^{-\frac{x y}{\sigma_m^2}}\1_{[0,\infty)}(y)\right)
=e^{\frac{x^2}{2\sigma_m^2}}\int_{x/\sigma_m}^\infty\frac{e^{-z^2/2}}{\sqrt{2\pi}} {\rm d}z
\geq c \frac{\sigma_m}{x} \ ,
$$
the estimate \eqref{eq: Y variance} on $\sigma_m$ and the assumption  $0\leq x\leq \log N$.

\begin{proof}
Consider the probability measure $\mathbb{Q}$ constructed from $\P$ through the density
\begin{equation}
\label{eqn: Q 2 point}
\frac{{\rm d}\mathbb{Q}}{{\rm d}\P}=\frac{e^{\sum_{m=2}^{K-1}\bm{\xi}_m\cdot\bm{Y}_m}}{\E\left(e^{\sum_{m=2}^{K-1}\bm{\xi}_m\cdot\bm{Y}_m}\right)}.
\end{equation}
for $\bm{\xi}_m=\xi_m(1,1)$ to be picked later.
We write $\E_\mathbb{Q}$ for the expectation under $\mathbb{Q}$ and $\E$ for the expectation under $\P$.
For $\bm x=(x,x)$, the probability can be written as
\begin{equation}
\label{eqn: P[J]}
\P\left(J_x(h_1)\cap J_x(h_2)\right)=\E\left(e^{\sum_{m=2}^{K-1}\bm{\xi}_m\cdot\bm{Y}_m}\right)e^{-\sum_{m=2}^{K-1}\bm{\xi}_m\cdot\bm{x}}\E_\mathbb{Q}\left(e^{-\sum_{m=2}^{K-1}\bm{\xi}_m\cdot(\bm{Y}_m-\bm{x})};J_x(h_1)\cap J_x(h_2)\right) \ .
\end{equation}
The first factor is evaluated using Lemma \ref{lem: two point fourier bound} on exponential moments.
For the choice, 
\begin{equation}
\label{eqn: lambda}
\xi_{m}=\frac{x}{\sigma_{m}^{2}},m=2,\ldots,K-1,
\end{equation}
we have $0<\xi_m<2$ by the assumption on $x$, so the Lemma can be applied.
We get
\begin{equation}
\begin{aligned}
\label{eq: fourier bound two point precise}
\E\left(e^{\sum_{m=2}^{K-1}\bm{\xi}_m\cdot\bm{Y}_m}\right)  =\left(1+\OO\left(e^{-N^{1/(10K)}}\right)\right)  \exp\left(\frac{1}{2}\sum_{m=2}^{K-1} \bm \xi_m\cdot \bm \Sigma_m \bm \xi_m \right)\ .
\end{aligned}
\end{equation}
The estimate \eqref{eq: Y covariance bound} on the covariance gives  that $\sum_{m\ge2}\rho_{m}(h_1,h_2)=O\left(1\right)$ for $h_1 \wedge h_2\leq\log N/K$.
Therefore, the quadratic form reduces to
\begin{equation}
\label{eqn: quad form 1}
\begin{aligned}
\frac{1}{2}\sum_{m=2}^{K-1}\bm \xi_m\cdot \bm \Sigma_m \bm \xi_m=\sum_{m=2}^{K-1}\xi^2_m(\sigma_{m}^{2}+\rho_{m})=\sum_{m=2}^{K-1}\xi^2_m\sigma_{m}^{2}+\OO(1)\ .
\end{aligned}
\end{equation}
Putting this in \eqref{eq: fourier bound two point precise}, we get
$$
\E\left(e^{\sum_{m=2}^{K-1}\bm{\xi}_m\cdot\bm{Y}_m}\right)e^{-\sum_{m=2}^{K-1}\bm{\xi}_m\cdot\bm{x}}
=e^{-\sum_{m=2}^{K-1}\xi_m x +\OO(1)}\ .
$$
Equation \eqref{eq: two point decoupling rough} follows from this since the third factor of \eqref{eqn: P[J]} is smaller than $1$ by the definition of the event. 

A more careful analysis of \eqref{eqn: P[J]} is needed to prove \eqref{eq: two point precise bound}.
First, note that if $h_1 \wedge h_2<\log N/(2K)$, then $\sum_{m\ge2}\rho_{m}=O\left(N^{-1/(2K)}\right)$ by \eqref{eq: Y covariance bound}.
Therefore for the same choice of $\xi_m$, we have
$$
\E\left(e^{\sum_{m=2}^{K-1}\bm{\xi}_m\cdot\bm{Y}_m}\right)e^{-\sum_{m=2}^{K-1}\bm{\xi}_m\cdot\bm{x}}
=(1+\OO(N^{-1/(2K)}))e^{-\sum_{m=2}^{K-1}\xi_m x}\ .
$$
We will thus be done once we show
\begin{equation}
\label{eq: suffices to show-1}
\E_\mathbb{Q}\left(e^{-\sum_{m=2}^{K-1}\bm{\xi}_m\cdot(\bm{Y}_m-\bm{x})};J_x(h_1)\cap J_x(h_2)\right)
=\left(\prod_{m=2}^{K-1}\eta_{0,\sigma_{m}^{2}}\left(e^{-\frac{x y}{\sigma_m^2}}\1_{[0,\infty)}(y)\right)+\OO_K\left(N^{-c}\right)\right)^2\ .
\end{equation}
Note that the product in \eqref{eq: suffices to show-1} is the dominant term since it is at least $c\log N^{-(K-1)}$ by \eqref{eqn: eta bounds}.
We prove \eqref{eq: suffices to show-1} using Fourier inversion.

Let $\bm{t}_m=\left(t_{1,m}, t_{2,m}\right)$ and $t_{j,m}\in \mathbb R$ for $m=2,\dots, K-1$ and consider  $\bm{\xi}_m+\ii\bm{t}_m$.
Suppose $|t_{j,m}|<N^{1/(32K)}$ so that $|\xi_{m}+\ii t_{j,m}|<N^{1/(16K)}$.
Let $\mu$ be the law of $(\bm{Y}_m-\bm{x}; m=2,\dots,K-1)$ under $\mathbb{Q}$. Its Fourier transform $\hat{\mu}$ becomes:
\begin{equation}
\label{eqn: fourier 2 point}
\begin{aligned}
&\E_\mathbb{Q}\left(e^{\ii \sum_{m=2}^{K-1}\bm{t}_m\cdot\left(\bm{Y}_m-\bm{x}\right)}\right)
=\frac{\E\left(e^{\sum_{m=2}^{K-1}(\bm{\xi}_m+i\bm{t}_m)\cdot\bm{Y}_m}\right)s}{\E\left(e^{\sum_{m=2}^{K-1}\bm{\xi}_m\cdot\bm{Y}_m}\right)}e^{-\ii \sum_{m=2}^{K-1}\bm{t}_m\cdot \bm{x}}\ .
\end{aligned}
\end{equation}
We apply Lemma \ref{lem: two point fourier bound} with ${\bm \xi}_m+\ii{\bm t}_m$ in place of ${\bm \xi}_m$ to the numerator,
and use \eqref{eq: fourier bound two point precise} to bound the denominator.
After cancellation, we obtain that \eqref{eqn: fourier 2 point} equals
\begin{equation}\label{eqn: fourier 2 point 2}
\left(1+\OO\left(e^{-N^{1/(10K)} }\right)\right)\exp\left(-\sum_{m=2}^{K-1}\frac{1}{2} \bm{t}_m\cdot \bm{\Sigma}_m \bm{t}_m + i \sum_{m=2}^{K-1} \bm{t}_m\cdot \bm{\Sigma}_m \bm{\xi}_m-\ii \sum_{m=2}^{K-1}\bm{t}_m\cdot \bm{x}\right)\ .
\end{equation}
As in \eqref{eqn: quad form 1}, but using here $\sum_{m\ge2} \rho_m = \OO(N^{-1/(2K)})$, we have that
$$
\bm{t}_m\cdot \bm{\Sigma}_m \bm{\xi}_m
=\bm{t}_m\cdot \bm x + \OO\left(\frac{\rho_mx}{\sigma_m^2}\|\bm{t}_m\|\right)
=\bm{t}_m\cdot \bm x+\OO(N^{-1/(4K)})\ .
$$
Thus \eqref{eqn: fourier 2 point 2} in fact equals
\begin{equation}
\label{eqn: fourier 2 point 3}
\left(1+ \OO\left(N^{-1/\left(4K\right)}\right)\right)\exp\left(-\frac{1}{2}\sum_{m=2}^{K-1}(t^2_{1,m}+t^2_{2,m})\sigma_{m}^{2}\right).
\end{equation}
The exponential above is precisely the Fourier transform $\hat{\nu}$ of $\nu=\otimes_{m=2}^{k-1}\eta_{0,\sigma_{m}^{2}}$.
Thus we have shown that 
\begin{equation}
\label{eq: 2 point to show}
\hat{\mu}({\bm t}_2,\ldots, {\bm t}_{K-1})=\left(1+\OO(N^{-1/(4K)})\right)\hat{\nu}({\bm t}_2,\ldots, {\bm t}_{K-1}) ,\ \text{ when $|t_{i,m}| \le N^{1/(32K)}$.}
\end{equation}
This suggests the decoupling in \eqref{eq: suffices to show-1}.
To complete the argument, consider the function $g_{\xi}: \mathbb R\to \mathbb R$ where
\[
g_{\xi}\left(y\right)=\begin{cases}
0 & \text{ if $y\le -N^{-1/(64K^2)}$},\\
e^{-\xi y} & \text{ if $y> 0$},\\
\end{cases}
\]
and $g_\xi$ is linearly interpolated on $[-N^{-1/(64K^2)},0]$.
Note that $g_{\xi}$ is bounded by $1$ and has Lipschitz constant $N^{1/(64K^2)}$.
By definition,
\begin{equation}
\label{eq: domination with smooth approx of indic-1}
\E_\mathbb{Q}\left(e^{-\sum_{m=2}^{K-1}\bm{\xi}_m\cdot\left(\bm{Y}_m-\bm{x}\right)}; J_x(h_1) \cap J_x(h_2)  \right)
\leq \E_\mathbb{Q}\left(\prod_{m=2}^{K-1}\prod_{i=1,2}g_{\xi_m}\left(Y_m(h_i)-x\right)\right)\ .
\end{equation}
Lemma \ref{lem: Fourier Inversion} can be applied with $d=2(K-2)$, $T=N^{1/\left(32K^2\right)}$, $R=N^{1/(32K^2)}$.
The right-hand side of \eqref{eq: domination with smooth approx of indic-1} becomes
\begin{equation}
\label{eqn: estimate after fourier 2}
\begin{aligned}
&\left(\prod_{m=2}^{K-1}\eta_{0,\sigma_{m}^{2}}\left(g_{\xi_m}(y)\right)\right)^2
+ \OO\left(N^{1/\left(64K^2\right) - 1/\left(32K^2\right)}\right)\\
&+ \OO\left( N^{2(K-2)/\left(32K^2\right)} N^{2(K-2)/(32K^2)} N^{-1/(4K)}\right)\\
&+\left(\otimes_{m=2}^{K-1}\eta_{0,\sigma_{m}^{2}}\left( \left(-N^{1/\left(32K^2\right)},N^{1/\left(32K^2\right)}\right) ^{c}\right)\right)^2\\
&+\mathbb{Q}\left(\exists m: |Y_m(h_1)-x|>N^{1/\left(32K^2\right)}\text{ or }  |Y_m(h_2)-x|>N^{1/\left(32K^2\right)}\right)\ .
\end{aligned}
\end{equation}
A standard Gaussian estimate and \eqref{eq: Y variance} show that 
\begin{equation}
\label{eqn: gaussian estimate 2 point}
\otimes_{m=2}^{K-1}\eta_{0,\sigma_{m}^{2}}\left( \left(-N^{1/(32K^2)},N^{1/(32K^2)}\right) ^{c}\right)= \OO\left(Ke^{-cN^{1/(16K^2)}/\log N}\right).
\end{equation}
Lemma \ref{lem: two point fourier bound}
and the definition of $\mathbb{Q}$ imply the exponential moment $\mathbb Q(\exp( \lambda (Y_m(h) - x) ))s \le c \exp( \lambda^2 \sigma_m^2 )$, valid for all
$m$, $h$ and $ 1 \le |\lambda| \le N^{1/(10K)}$, where we have used that $\rho_m \le \sigma_m^2$. The choice $\lambda = N^{1/(32K^2)}/\sigma_m^2$ and the exponential Markov's inequality shows that for all $m$ and $h$,
\begin{equation}\label{eqn: Y estimate 2 point}
\mathbb{Q}\left(|Y_m(h)-x|>N^{1/(32K^2)}\right) \le c\exp( -cN^{1/(16K^2)}/\log N )\ .
\end{equation}
This means that last term of \eqref{eqn: estimate after fourier 2} is also bounded by the right-hand side of \eqref{eqn: gaussian estimate 2 point}.
We conclude that
\begin{equation}
\label{eqn: 2 point almost}
\E_\mathbb{Q}\left(e^{-\sum_{m=2}^{K-1}\bm{\xi}_m\cdot\left(\bm{Y}_m-\bm{x}\right)};J_x(h_1) \cap J_x(h_2) \right)
\leq \left(\prod_{m=2}^{K-1}\eta_{0,\sigma_{m}^{2}}\big(g_{\xi_m}(y)\big)\right)^2 + \OO_K(N^{-c})\ .
\end{equation}
Note that $\left|\eta_{0,\sigma_{m}^{2}}\left(g_{\xi_m}\right)-\eta_{0,\sigma_{m}^{2}}\left(e^{-\xi y}\1_{[0,\infty)}(y)\right)\right|\le N^{-1/\left(64K^2\right)}$
and recall \eqref{eqn: eta bounds}. This together with \eqref{eqn: 2 point almost} shows \eqref{eq: suffices to show-1}, and completes the proof of \eqref{eq: two point precise bound}.
\end{proof}
We now turn to bounding $P\left(J_x(h_1)\cap J_x(h_2)\right)$ when $h_1$ and $h_2$ are ``close''.
In this regime we do not need such a precise bound, so Fourier inversion is not needed.
The bound \eqref{eq: two point rough 2} reflects that if $h_1\wedge h_2 \in [(j-1)\log N/K,j\log N/K]$ the increments $Y_m(h_1),Y_m(h_2),m=2,\ldots,j-1$, are essentially
perfectly correlated, while the increments $Y_m(h_1),Y_m(h_2),m=j+1,\ldots,K-1$ are essentially independent.
The increments $Y_j(h_1),Y_j(h_2)$ are partially correlated, but we ignore this and dominate by the scenario where the
correlation is perfect. 
This leads to a loss in the bound, which turns out to be irrelevant in the second moment computation.
\begin{prop}[Two-point bound; Coupling]
\label{prop: two point bound rough}
Let $h_{1},h_{2}\in\mathbb{R}$ such that $\frac{j-1}{K}\log N< h_1 \wedge h_2\leq \frac{j}{K}\log N$ for some $j=2,\dots, K-1$. 
Then for $0<x\le \log N$, we have that
\begin{equation}
\label{eq: two point rough 2}
\begin{aligned}
&\P\left(J_x(h_1)\cap J_x(h_2)\right)
\leq
C\exp\left({-\sum_{m=2}^{j}\frac{x^2}{2\sigma_{m}^{2}}-\sum_{m=j+1}^{K-1}\frac{x^2}{\sigma_{m}^{2}}}\right)\ .
\end{aligned}
\end{equation} 
\end{prop}
\begin{proof}
The proof is very similar to the proof of \eqref{eq: two point decoupling rough}.
As in \eqref{eqn: P[J]}, we use a change of measure $\mathbb{Q}$ 
\begin{equation}
\label{eqn: P[J] 2}
\begin{aligned}
\P\left(J_x(h_1)\cap J_x(h_2)\right)
=e^{-\sum_{m=2}^{ K-1}\bm \xi_m \cdot x}
\E\left(e^{\sum_{m=2}^{K-1}\bm \xi_m\cdot \bm Y_m} \right)
\E_\mathbb{Q}\left(e^{-\sum_{m=2}^{K-1}\bm \xi_m\cdot (\bm Y_m-\bm x)};J_x(h_1)\cap J_x(h_2)\right)\,
\end{aligned}
\end{equation}
where $\bm \xi_m=(\xi_m,\xi_m)$, $\xi_m\ge0$ and $\bm x=(x,x)$,
Note that the last factor is again smaller than $1$ by the definition of $J_x(h)$.
As in \eqref{eq: fourier bound two point precise}, we have using Lemma \ref{lem: two point fourier bound} that
$$
\E\left(e^{\sum_{m=2}^{K-1}\bm \xi_m\cdot \bm Y_m} \right)
\leq C\exp\left(\frac{1}{2} \sum_{m=2}^{K-1} \bm \xi_m\cdot \bm \Sigma_m \bm \xi_m\right)\ .
$$
By \eqref{eq: Y covariance bound} and the assumption $(j-1)\log N/K<h_1 \wedge h_2\leq j\log N/K$, we have for $m\neq j$
 \begin{equation}
\label{eqn: quad 1}
\begin{aligned}
\frac{1}{2}\bm \xi_m\cdot \bm \Sigma_m \bm \xi_m
=\xi^2_m\big(\sigma_{m}^{2}+\rho_m(h_1,h_2)\big)
=\begin{cases}
\xi^2_m\big(2\sigma_{m}^{2}+ \OO(1)\big) &\text{ if $m\leq j-1$}\\
\xi^2_m\big(\sigma_{m}^{2}+ \OO(1)\big) &\text{ if $m\geq j+1$.}
\end{cases}
\end{aligned}
\end{equation}
For $m=j$, since $\rho_m\leq \sigma_m^2$, we have that
\begin{equation}
\frac{1}{2}\bm \xi_j\cdot \bm \Sigma_j \bm \xi_j\leq 2\xi_j^2 \sigma_j\ .
\end{equation}
To optimize the bound we pick
\begin{equation}
\label{eqn: lambda split}
\xi_m=
\begin{cases}
\frac{x}{2\sigma_m^2} &\text{ if $m\leq j$,} \\
\frac{x}{\sigma_m^2} &\text{ if $m\geq j+1$.} 
\end{cases}
\end{equation}
Using \eqref{eqn: quad 1}-\eqref{eqn: lambda split} in \eqref{eqn: P[J] 2} we obtain \eqref{eq: two point rough 2}.
\end{proof}

Finally, we bound the one point probability
$\P\left(J_x(h)\right)$ from below. Here we again need a precise bound which uses Fourier inversion.
\begin{prop}[One-point bound]
\label{prop: one point precise bound}
For every $h\in \mathbb{R}$ and $0<x\le \log N$,
we have that
\begin{equation}
\label{eq: one point precise bound}
\P\left(J_x(h)\right)\ge \left(1+\OO_K\left(N^{-c}\right)\right)e^{-\sum_{m=2}^{K-1}\frac{x^2}{2\sigma_{m}^{2}}}
\prod_{m=2}^{K-1}\eta_{0,\sigma_{m}^{2}}\left( e^{-\frac{x y}{\sigma_m^2}}\1_{[0,\infty)}(y) \right)
\ ,
\end{equation}
where $\sigma^2_m$ is defined in \eqref{eq: Y variance}.
\end{prop}
\begin{proof}
By rotational invariance, it suffices to consider the case $h=0$. 
We write $Y_m=Y_m(0)$, $m=2,\dots, K-1$ for simplicity.
The proof relies on a change of measure followed by a Fourier inversion as in the proof of Proposition \ref{prop: two point decoupling rough}.
Let $\xi_m=x/\sigma_{m}^{2}$ as in \eqref{eqn: lambda}, so that $0<\xi_m<1$ by the assumption on $x$.
Consider the probability measure $\Q$ constructed from $\P$ via the density
\[
\frac{{\rm d}\Q}{{\rm d}\P}=\frac{e^{\sum_{m=2}^{K-1} \xi_mY_m}}{\E\left(e^{\sum_{m=2}^{K-1}\xi_m Y_m}\right)}.
\]
Again, we write $\E_{\Q}$ for the expectation under $\Q$ and $\E$ for the expectation under $\P$.
We have
\begin{equation}
\label{eqn: P[J] 1}
\P\left(J_x(0)\right)=\E\left(e^{\sum_{m=2}^{K-1}\xi_mY_m}\right)e^{-\sum_{m=2}^{K-1}\xi_m x}\E_\Q\left(e^{-\sum_{m=2}^{K-1}(\xi_mY_m-x)};J_x(0)\right) \ .
\end{equation}
Lemma \ref{lem: two point fourier bound} is applied to evaluate the exponential moment (with $\bm \xi_m = (\xi_m,0)$). It yields
$$
\E\left(e^{\sum_{m=2}^{K-1}\lambda_mY_m}\right)=\left(1+\OO\left(e^{-N^{1/(10K)}}\right)\right)e^{\sum_{m=2}^{K-1}\lambda_{m}^{2}\sigma_{m}^{2}}\ .
$$
In view of \eqref{eqn: P[J] 1} and the above, it remains to show that
\begin{equation}
\label{eq: suffices to show one point}
\E_{\mathbb{Q}}\left(e^{-\sum_{m=2}^{K-1}\xi_m (Y_m-x)};J_x(0)\right)
\geq  \prod_{m=2}^{K-1}\eta_{0,\sigma_{m}^{2}}\left(e^{-\frac{x y}{\sigma_m^2}}\1_{[0,\infty)}(y)\right)+\OO_K\left(N^{-c}\right).
\end{equation}
This is done by Fourier inversion.

Let $t_{m}\in \mathbb R$ for $m=2,\dots, K-1$ with $|t_{m}|<N^{1/(32K)}$. 
Then  $|\xi_m+\ii t_m|<N^{1/(16K)}$ so that Lemma \ref{lem: two point fourier bound} can be applied with $\xi_m+\ii t_m$ in place of $\xi_m$.
The Fourier transform of $(Y_m-x; m=2,\dots,K-1)$ under $\Q$ becomes:
\begin{equation}
\label{eqn: fourier 1 point}
\begin{aligned}
&\E_\Q\left(e^{\ii \sum_{m=2}^{K-1}t_m(Y_m-x)}\right)
=\frac{\E\left(e^{\sum_{m=2}^{K-1}(\xi_m+\ii t_m) Y_m}\right)}{\E\left(e^{\sum_{m=2}^{K-1}\xi_mY_m}\right)}e^{-i\sum_{m=2}^{K-1}t_m\cdot x}\\
&=\left(1+\OO\left(e^{-N^{1/(10K)}}\right)\right) \exp\left(-\frac{1}{2}\sum_{m=2}^{K-1}t^2_{m}\sigma_{m}^{2}\right)\ .
\end{aligned}
\end{equation}
To complete the argument, consider the function $g_{\lambda}: \mathbb R\to \mathbb R$ where
\[
g_{\xi}\left(y\right)=\begin{cases}
0 & \text{ if $y\le 0$},\\
e^{-\xi y} & \text{ if $y> N^{1/(32 K)}$}.\\
\end{cases}
\]
and $g_\xi$ is linearly interpolated on $[0,N^{-1/(32K)}]$.
Note that $g_{\xi}$ is bounded by $1$ and has Lipschitz constant at most $N^{1/\left(32K\right)}$.
By definition,
\begin{equation}
\label{eq: domination 1 point}
\E_\Q\left(e^{-\sum_{m=2}^{K-1}\xi_m\left(Y_m-x\right)};J_x(0)\right)
\geq \E_\Q\left(\prod_{m=2}^{K-1}g_{\xi_m}\left(Y_m-x\right)\right).
\end{equation}
Lemma \ref{lem: Fourier Inversion} can be applied with $T=N^{1/\left(16K\right)}$, $R=N^{(1/16K)}$ and \eqref{eqn: fourier 1 point}.
The right-hand side of \eqref{eq: domination 1 point} becomes
\begin{equation}
\label{eqn: estimate after fourier}
\begin{aligned}
&\prod_{m=2}^{K-1}\eta_{0,\sigma_{m}^{2}}\left(g_{\xi_m}(y)\right) + \OO\left(N^{1/\left(32K\right)-1/\left(16K\right)}\right) + \OO\left( N^{2/\left(16K\right)} N^{2/(16K)} e^{-N^{1/(10K)}}\right)\\
&+\otimes_{m=2}^{K-1}\eta_{0,\sigma_{m}^{2}}\left(\left\{ \left(-N^{1/(16K)},N^{1/(16K)}\right)\right\}^{c}\right)+\Q\left(\exists m: |Y_m-x|>N^{1/(16K)} \right)\ .
\end{aligned}
\end{equation}
We can proceed as in \eqref{eqn: gaussian estimate 2 point} and in \eqref{eqn: Y estimate 2 point} to conclude that
\begin{equation}
\label{eqn: 1 point almost}
\E_\Q\left(e^{-\sum_{m=2}^{K-1}\xi_m(Y_m-x)};J_x(0)\right)
\geq \prod_{m=2}^{K-1}\eta_{0,\sigma_{m}^{2}}\left(g_{\xi_m}(y)\right) + \OO(N^{-1/(32K)})\ .
\end{equation}
Note that $\left|\eta_{0,\sigma_{m}^{2}}\left(g_{\xi_m}\right)-\eta_{0,\sigma_{m}^{2}}\left(e^{-\xi y}\1_{[0,\infty)}(y)\right)\right|\le N^{-1/\left(32K\right)}$.
This with \eqref{eqn: 1 point almost} shows \eqref{eq: suffices to show one point}.
\end{proof}

We now have all estimates needed to prove the second moment estimate in Proposition \ref{prop: fm sm}.

\begin{proof}[Proof of Proposition \ref{prop: fm sm}]
Linearity of expectations and Proposition \ref{prop: one point precise bound} directly imply that
\begin{equation}
\label{eqn: E[Z]^2 lower}
\begin{aligned}
\E(Z)^2 \geq N^2\left(1+\oo_K(1)\right)e^{-\sum_{m=2}^{K-1}\frac{x^2}{\sigma_{m}^{2}}}\left(\prod_{m=2}^{K-1}\eta_{0,\sigma_{m}^{2}}\left(e^{-\frac{x y}{\sigma_m^2}}\1_{[0,\infty)}(y)\right)\right)^2\ .
\end{aligned}
\end{equation}
The choice \eqref{eqn: x} of $x$ and the bounds \eqref{eqn: eta bounds} on the Gaussian probability implies
\begin{equation}
\label{eqn: E[Z] lower}
\E(Z)^2 \geq c(\log N)^{-(K-1)/2}\ N^2e^{-\sum_{m=2}^{K-1}\frac{x^2}{\sigma_{m}^{2}}}
\end{equation}
The second moment can be split in terms of $h_1 \wedge h_2$:
\begin{equation}
\label{eqn: sum split}
\E(Z^2)=
\Big(\sum_{\substack{h_1,h_2\in \mathcal H_N\\ h_1 \wedge h_2\le\frac{1}{2K}\log N}}
+ \sum_{\substack{h_1,h_2\in \mathcal H_N\\ \frac{1}{2K}\log N\leq h_1 \wedge h_2\le \frac{1}{K}\log N}} 
+ \sum_{j=2}^{K}\sum_{\substack{h_1,h_2\in \mathcal H_N\\  \frac{j-1}{K}\log N<h_1 \wedge h_2\leq \frac{j}{K}\log N}}\Big)\ \P\left(J_x(h_1)\cap J_x(h_2)\right)\ .
\end{equation}
The first sum is smaller than $(1+\oo_K(1))\E\left(Z^2\right)$ by \eqref{eq: two point precise bound} and \eqref{eqn: E[Z]^2 lower}, and the fact that there are less than $N^2$ terms in the sum.
It remains to show that the second and third terms of \eqref{eqn: sum split} are $\oo(\E\left(Z\right)^2)$. 
There are $\OO(N^{2-1/(2K)})$ terms in the second sum. We can thus apply the rough estimate \eqref{eq: two point decoupling rough} 
to get that this sum is smaller than 
$$
\OO\left( N^{2-1/(2K)}  \exp\left(-\sum_{m=2}^{K-1}\frac{x^2}{\sigma_m^2}\right) \right) \ .
$$
The estimate \eqref{eqn: E[Z] lower} then directy implies that the second sum is $\oo(\E\left(Z\right)^2)$.
As for the third term of \eqref{eqn: sum split}, for a fixed $j=2,\ldots,K-1$, there are at most $\OO(N^{2-(j-1)/K})$ pairs such that  $\frac{j-1}{K}\log N<h_1 \wedge h_2\leq \frac{j}{K}\log N$, 
Applying the two-point probability estimate \eqref{eq: two point rough 2}, one gets that the contribution of such pairs for a fixed $j$ is 
\begin{equation}\label{eq: third}
\OO\left( \left(N^2 e^{-\sum_{m=2}^{K-1}\frac{x^2}{\sigma_{m}^{2}}} \right) \ N^{-(j-1)/K} \exp\left(\sum_{m=2}^{j}\frac{x^2}{2\sigma_2^2}\right) \right)\ .
\end{equation}
The term in parenthesis is smaller than $C(\log N)^{(K-1)/2}\E(Z)^2$ by \eqref{eqn: E[Z] lower}, whereas the other terms are, by the choice of $x$
and \eqref{eq: Y variance}, of the order of 
$$
N^{-\frac{j-1}{K}}N^{(1-\varepsilon/2)^2\frac{j-1}{K}}=\OO_K(N^{-c(j-1)})\ ,
$$
for $c=c(\varepsilon)$. The sum of \eqref{eq: third} from $j=2$ to $K-1$ is therefore of order $\OO_K(N^{-c})$.
Altogether this implies that the third term of \eqref{eqn: sum split} is $\oo_K( \E\left(Z\right)^2 )$, and concludes the proof of the proposition.
\end{proof}

\section{Extension to the full sum and to the continuous interval}\label{sec: chaining etc}
\subsection{Lower bound.\ }
In this section, we strengthen the lower bound \eqref{eq: lb trunc disc} of the previous section to get:
\begin{prop}
\label{prop: lower bound}
For any $\varepsilon>0$, 
\begin{equation}\label{eq: LB}
  \lim_{N\to\infty} \P\left( \max_{h \in [0,2\pi]} \sum_{j=1}^\infty -\frac{\Re( e^{-\ii jh} \Tr{{\rm U}_N^j})}{j}  \ge (1-\varepsilon) \log N \right) = 1\ .
\end{equation}
\end{prop}
To prove this, we need the following exponential moment bound for the tail of the sum.
\begin{lem}
\label{lem: mgf ub tail}
For any fixed $\delta \in (0,1)$ and any $C>0$ we have for $N$ large enough that for all $h\in\mathbb{R}$ and $|\alpha| \le C$,
\begin{equation}\label{eq: mgf ub tail body}
  \E\left( \exp\left( \alpha \sum_{j \ge N^{1-\delta} } -\frac{\Re( e^{-\ii jh} \Tr \U_N^j)}{j}\right) \right)
 = \exp\left( \left(\frac{1}{4}+\oo(1)\right)\alpha^2 \delta \log N \right),
\end{equation}
\end{lem}
The bound is proved in Section \ref{sec: RH} in the form of Proposition \ref{prop:tail}.

\begin{proof}[Proof of Proposition \ref{prop: lower bound}]
We show that
\begin{equation}\label{eq: LB disc}
  \lim_{N\to\infty} \P\left( \max_{h \in \mathcal{H}_N} \sum_{j=1}^{\infty} -\frac{Re(e^{-\ii jh}\Tr{{\rm U}_N^j})}{j} \ge (1-\varepsilon) \log N \right) = 1,
\end{equation}
from which \eqref{eq: LB} trivially follows.
Using \eqref{eq: mgf ub tail body} with $\alpha = -2x/(\delta \log N)$ we get that
\begin{equation}\label{eq: tail bound for tail}
 \P\left(\sum_{j\ge N^{1-\delta}}-\frac{\Re(e^{-\ii jh}\Tr{{\rm U}_N^j})}{j} \le -x \right) \le \exp( -cx^2/(\delta \log N))
, \ \text{ for all $0 \le x \le \log N$.}
\end{equation}
A simple union bound over the $N$ points of $\mathcal{H}_N$, like in \eqref{eq: union bound}, now shows that
\begin{equation}\label{eq: lb union}
\P\left(\min_{h \in \mathcal{H}_N}  \sum_{j\ge N^{1-\delta}}-\frac{\Re(e^{-\ii jh}\Tr{{\rm U}_N^j})}{j} \le -\frac{1}{2} \varepsilon \log N \right)
\le N^{1 - c\varepsilon^2/\delta}.
\end{equation}
Thus given $\varepsilon>0$, $\delta$ can be set small enough such that \eqref{eq: lb union} is $\oo(1)$, and such
that \eqref{eq: lb trunc disc}, with $\frac{1}{2}\varepsilon$ in place of $\varepsilon$, is satisfied. Combining
these implies \eqref{eq: LB disc}, and thus also \eqref{eq: LB}. 
\end{proof}

\subsection{Upper bound\ }
In this section, we strengthen the upper bound \eqref{eq: upper bound discrete} by removing the discretization and truncation, to arrive at:
\begin{prop}
\label{prop: upper bound}
 For any $\varepsilon>0$,
\begin{equation}\label{eq: UB}
  \lim_{N\to\infty} \P\left(\max_{h \in [0,2\pi]} \sum_{j=1}^\infty -\frac{\Re(e^{-\ii jh}\Tr{{\rm U}_N^j})}{j}\ge (1+\varepsilon) \log N \right) = 0\ .
\end{equation}
\end{prop}
The proof is split into three steps.
The first step is to extend \eqref{eq: upper bound discrete} to a bound for the truncated sum over all of $[0,2\pi]$
using a chaining argument.
In the second step we restrict once again to a discrete set, but one containing $N^C$ equidistant points in $[0,2\pi]$
for a large $C$, and show that the largest error made
in the truncation over this denser discrete set is negligible compared to the leading order $\log N$. Thus we obtain
a bound for the full sum over the denser discrete set.
Finally we use a rough control of the derivative of the characteristic polynomial to show that
the maximum over the denser set is close to the maximum over $[0,2\pi]$.

To carry out the first step, we need a tail estimate for the difference between the truncated sum at two different but close points (at distance at most $N^{-(1-\delta)}$),
that is for
\begin{equation}\label{eq: diff}
\sum_{j=1}^{N^{1-\delta}} -\frac{ \Re( e^{-\ii jh} \Tr {\rm U}_N^j ) }{j} 
- \sum_{j=1}^{N^{1-\delta}} -\frac{\Re( \Tr{\rm U}_N^j )}{j}, \ |h|\leq N^{-1-\delta}.
\end{equation}
Using \eqref{eqn: diaconis-shah} one can compute the covariance matrix ${\bm \Sigma}$ of the two sums in \eqref{eq: diff} exactly; it turns out to be:
\begin{equation}
\label{eqn: Sigma 2}
\bm \Sigma=\left(
\begin{array}{cc}
\sigma^2 & \rho\\
\rho & \sigma^2
\end{array}
\right) \ , \mbox{ for } \sigma^2 = \frac{1}{2}  \sum_{j=1}^{N^{1-\delta}} \frac{1}{j} \mbox{ and } \rho = \frac{1}{2}  \sum_{j=1}^{N^{1-\delta}} \frac{\cos(jh)}{j}.
\end{equation}
If $|h|\le N^{-(1-\delta)}$ then $|jh|\leq 1$ for $j\leq N^{1-\delta}$, which implies that $\cos (jh)=1+\OO(j^2h^2)$ and consequently $\rho \ge \sigma^2 - cN^{2(1-\delta)}h^2$.
This reflects the fact that as $h$ decreases below scale $N^{-(1-\delta)}$ the correlation no longer behaves as the $\log$ of the inverse of $h$, but rather approaches $1$ as a quadratic,
so that the difference \eqref{eq: diff} has variance $cN^{2(1-\delta)}h^2$, decreasing quadratically in $h$.
To obtain a corresponding tail bound, we need the following
exponential moment bound (as the similar Lemma \ref{lem: two point fourier bound}, it follows from
the Riemann-Hilbert techniques of Section \ref{sec: RH}.
\begin{lem}
\label{lem: exp moment continuity}
Let $\delta>0$, $\e\in(0,\delta)$ be fixed. There exists $C>0$ such that for all $h\in \mathbb{R}$ and real $\xi$ with $|\xi h N^{1-\delta}| \le N^{\delta-\e}$, we have 
\begin{equation}\label{eq: exp moment continuity}
\begin{aligned}
    \E\left(
	\exp\left( \xi \left( \sum_{j=1}^{N^{1-\delta}} -\frac{ \Re( e^{-\ii jh} \Tr {\rm U}_N^j ) }{j} 
        - \sum_{j=1}^{N^{1-\delta}} -\frac{\Re( \Tr{\rm U}_N^j )}{j} \right) \right )   
      \right)
   \le C\exp\left(C \xi^2 |\rho-\sigma^2|\nc \right)
\end{aligned},
\end{equation}
where we used the definitions \eqref{eqn: Sigma 2}.
\end{lem}

We have $|\rho-\sigma^2|\leq C\sum_{j=1}^{N^{1-\delta}} \frac{(jh)^2}{j}\leq C N^{2-2\delta}h^2$\nc. Hence, using \eqref{eq: exp moment continuity} with 
$\xi = (h N^{1-\delta})^{-1}\nc$
and the exponential Chebyshev inequality we get that
\begin{equation}\label{eq: diff tail bound}
    \P\left(
	\sum_{j=1}^{N^{1-\delta}} -\frac{ \Re( e^{-\ii jh} \Tr {\rm U}_N^j ) }{j} 
        - \sum_{j=1}^{N^{1-\delta}} -\frac{\Re( \Tr{\rm U}_N^j )}{j} 
        \ge x h N^{1-\delta}
      \right)
   \le C \exp\left( - cx \right),
\end{equation}
for all $x \ge 1$.
The bound \eqref{eq: upper bound discrete} can now be extended to the set $[0,2\pi]$.
\begin{lem}\label{lem: UB step 2}
There is a constant $c$ such that for $0<\delta<1$
  \begin{equation}\label{eq: UB step 2}
  \lim_{N\to\infty}\P\left(
  \max_{ h \in [0,2\pi] }
  \sum_{\ell=1}^{(1-\delta)\log N} W_\ell(h)
  \ge \log N + c \right)=0\ .
  \end{equation}
\end{lem}
\begin{proof}
The proof uses a chaining argument on dyadic intervals. We actually show \eqref{eq: UB step 2}
with the maximum over $[0,2\pi]$ replaced by a maximum over the set $\cup_{n\ge 0} \mathcal H_{N2^n}$. Since
this set is dense in $[0,2\pi]$ and
$h \to \sum_{\ell = 1}^{(1-\delta)\log N } W_{\ell}(h)$
is continuous this implies \eqref{eq: UB step 2}.

For simplicity, define the random variable
$$
X(h)= \sum_{\ell=1}^{(1-\delta)\log N} W_{\ell}(h)\ .
$$
Consider 
$h \in \cup_{n\ge0} H_{N2^n}$.
For $k \ge 0$ consider the sets $\mathcal H_{N2^k}$ with $N2^k$ equidistant points on $[0,2\pi]$. 
Define the sequence 
$(h_k, k\ge0)$ 
as follows: 
if $h\in [\frac{2\pi j}{N2^k}, \frac{2\pi( j+1)}{N2^k})$ for some $j=0,\dots, N2^k-1$, then $h_k=\frac{2\pi j}{N2^k}$.
Note that $h_{k+1}-h_k$ is $0$ or $\frac{2\pi }{N2^{k+1}}$. 
It holds trivially that
\begin{equation}
\label{eqn: chain}
X(h)-X(h_0)=
\sum_{k=0}^{\infty}
\left(X(h_{k+1})-X(h_k)\right)\ .
\end{equation}
Consider the event 
$$
A=\left\{ \left|X(h')-X\left(h'+\frac{2\pi }{N2^{k+1}}\right)\right|\leq  \frac{k+1}{2^k}\  \ \forall h'\in \mathcal H_{N2^k},
\ \forall k \ge 0 \right\}\ .
$$
Since the sequence $(k+1)/2^k$ is summable, it is clear from \eqref{eqn: chain} 
that for all $h\in \cup_{ n \ge 0 } \mathcal H_{N2^n}$,
$$
X(h) = X(h_0) + \OO(1) \ , \text{ on the event $A$.}
$$
Therefore, the maximum over $\cup_{ n \ge 0 } \mathcal H_{N2^n}$ can only differ by a constant from the one on $\mathcal H_N$. 
The conclusion thus follows from Proposition \ref{prop: ub disc trunc} after it is shown that $\P(A^c)$ tends to $0$. 
A straightforward union bound yields
\begin{equation}
\label{eqn: P[A^c]}
\begin{aligned}
\P(A^c)&=
\sum_{k=0}^{\infty} 
\sum_{h' \in \mathcal H_{N2^k}} \P\left(\left|X(h')-X\left(h'+\frac{2\pi }{N2^{k+1}}\right)\right|> \frac{k+1}{2^k}\right)\ .
\end{aligned}
\end{equation}
The bound \eqref{eq: diff tail bound} is used with $h=\frac{2\pi }{N2^{k+1}}$ and $x = (\frac{k+1}{2^k})/(hN^{1-\delta}) \ge c(k+1)N^\delta$
to obtain that
$$
\P\left(\left|X(h')-X\left(h'+\frac{2\pi }{N2^{k+1}}\right)\right|> \frac{k+1}{2^k}\right)
\leq 2\exp\left( - c (k+1)N^{\delta} \right)\ .
$$
Therefore, we get from \eqref{eqn: P[A^c]}  the estimate
$$
\P(A^c) \le
\sum_{k=0}^{\infty}
N2^{k+1} \exp\left( - c (k+1)N^{\delta} \right) \leq ce^{-cN^{\delta}},
$$
which goes to $0$ as $N\to\infty$. 
\end{proof}

Next we bound the ``tail'' on a the set $\mathcal{H}_{N^{100}}$ (the choice $N^{100}$ is somewhat arbitrary,
but since our proof is based on a simple union bound we can not obtain the result over all of $[0,2\pi]$).
\begin{lem}
\label{lem: tail bound}
For every $0<\varepsilon<1$ there exists $\delta=\delta(\varepsilon)<1$ such that,
  \begin{equation}
  \lim_{N\to\infty}\P\left(\max_{h\in \mathcal H_{N^{100}}}\sum_{j \ge N^{1-\delta}}-\frac{\Re(e^{-\ii jh}\Tr{{\rm U}_N^j})}{j} \ge \varepsilon \log N \right)=0\ .
  \end{equation}
\end{lem}
\begin{proof}
By a union bound and rotational invariance, this probability is smaller than
$$
N^{100}
\P\left(\sum_{j \ge N^{1-\delta}} -\frac{\Re(\Tr{{\rm U}_N^j})}{j} \ge \varepsilon \log N \right)\ .
$$
Using the exponential Chebyshev inequality and Lemma \ref{lem: mgf ub tail} with $\alpha = \varepsilon / \delta$ this is at most
$$
N^{100} N^{- c\varepsilon^2/\delta }.
$$
This tends to zero if $\delta$ is chosen small enough, depending on $\varepsilon$.
\end{proof}

Lemmas \ref{lem: UB step 2} and \ref{lem: tail bound} combine to give us
\begin{lem}
\label{lem: UB step 3}
For every $0 < \varepsilon < 1$,
  \begin{equation}
  \lim_{N\to\infty}\P\left(\max_{h\in \mathcal H_{N^{100}}}\sum_{j=1}^{\infty} -\frac{\Re(e^{-\ii jh}\Tr{{\rm U}_N^j})}{j} \ge (1+\varepsilon) \log N \right)=0\ .
 \end{equation}
\end{lem}

It is now possible to prove the upper bound for the maximum from a crude control on the derivative.
\begin{proof}[Proof of Proposition \ref{prop: upper bound}]
To estimate the derivative in a neighborhood of a maximizer, we need to estimate how close the maximizer can be to an eigenvalue.
Let $e^{\ii\theta_1}, e^{\ii\theta_2},\ldots, e^{\ii\theta_N}$ denote the eigenvalues of the random matrix ${\rm U}_N$.
It is helpful to consider the event where the eigenvalues are not too close to each other:
$$
B=\left\{ \inf_{j\neq k} \|\theta_j-\theta_k\|> N^{-90}\right\}\ .
$$
It was shown in \cite{BenBou2013} (see Theorem 1.1) that $\lim_{N\to\infty} \P(B^c)=0$.
(In fact, the authors show that the smallest gap is of order $N^{-4/3}$).

It remains to estimate the probability restricted on the event $B$. 
We will show that on this event
\begin{equation}
\label{eqn: bound deriv}
\frac{{\rm d}}{{\rm d}h} \log |{\rm P}_N(e^{\ii h})|\leq cN^{92}, \qquad \forall h \text{ such that } |h-h^\star|\leq  N^{-100}\ ,
\end{equation}
where $h^\star$ is a maximizer of  $\log |P_N(e^{\ii h})|$. In particular,
$$
 \log |{\rm P}_N(e^{\ii h^\star})|\leq  \log |{\rm P}_N(e^{\ii h})| + cN^{-98}  \qquad \forall h \text{ such that } |h-h^\star|\leq  N^{-100}\ .
$$
Since there must be an $h\in \mathcal H_{N^{100}}$ with $|h-h^\star|\leq N^{-100}$, this implies that
$$
\P\left(\{\log |{\rm P}_N(e^{\ii h^\star})|\geq (1+\varepsilon)\log N\}\cap B\right)
\leq \P\left(\max_{h\in \mathcal H_{N^{100}}}\log |{\rm P}_N(e^{\ii h^\star})|\geq (1+\varepsilon)\log N + cN^{-98}\right)\to 0
$$
by Lemma \ref{lem: UB step 3}. 

To prove \eqref{eqn: bound deriv}, notice that the derivative is
\begin{equation}
\label{eqn: deriv}
\frac{\rd}{\rd h} \log |{\rm P}_N(e^{\ii h})|=
\Re\sum_{j=1}^N   \frac{\ii e^{\ii h}}{e^{\ii h}-e^{\ii \theta_j}}
\end{equation}
Suppose without loss of generality that $\theta_1$ is the closest eigenangle to $h^\star$. Then we must have for $j\neq 2$ that
$|\theta_j-h^\star|> \frac{1}{2}\inf_{j\neq k} |\theta_j-\theta_k|> \frac{1}{2}N^{-90}$ on the event $B$.
Since the derivative is $0$ at $h^\star$, this can be used to bound $|\theta_1-h^\star|$, namely
$$
\left| \Re  \frac{\ii e^{\ii h^\star}}{e^{\ii h^\star}-e^{\ii \theta_1}}\right|=\left|\Re\sum_{j=2}^N   \frac{\ii e^{\ii h^\star}}{e^{\ii h^\star}-e^{\ii \theta_j}}\right|\leq \sum_{j=2}^N \frac{1}{|e^{\ii h^\star}-e^{\ii \theta_j}|}
\leq cN^{91}.
$$
This shows by a Taylor expansion that $|\theta_1-h^\star|> N^{-91}$. 
This also means that for every $h$ such that $|h-h^\star|\leq  N^{-100}$, we must have $|\theta_j-h|> cN^{-91}$. 
Putting this back in \eqref{eqn: deriv} shows that 
$$
\frac{\rd}{\rd h} \log |{\rm P}_N(e^{\ii h})|\leq  cN^{92}\ ,
$$
as claimed. 
\end{proof}

To complete the proof of the main result Theorem \ref{thm: main} it now only remains to show the exponential moment/characteristic function bounds
in Lemmas \ref{lem: gaussian UB}, \ref{lem: two point fourier bound}, \ref{lem: mgf ub tail} and \ref{lem: exp moment continuity}.
These will be proved in Section \ref{sec: RH}.

\section{High points and free energy}
In this section we prove Theorem \ref{thm: high points} about the Lebesgue measure
of high points, and derive from it Corollary \ref{cor: free energy} about
the free energy. 

Recall the definition \eqref{eq: high points def} of the set $\mathcal{L}_{N}\left(\gamma\right)$
of $\gamma-$high points. The first goal is to prove Theorem \ref{thm: high points},
i.e. to show that $\mbox{Leb}\left(\mathcal{L}_{N}\left(\gamma\right)\right)=N^{-\gamma^{2}+\oo\left(1\right)}$ with high probability.
For the upper bound we will be able to work with the full sum (i.e. the logarithm of the characteristic polynomial without truncation). 
The following exponential moment bound for the full sum,
which is obtained from the Selberg integral (see Lemma \ref{lem:Selberg}),
will be used.
\begin{lem}
\label{lem: Full sum exp mom}
For any fixed $C>0$ we have uniformly for $\left|\alpha\right| \le C$ that
\begin{equation}
\mathbb{E}\left(\exp\left(\alpha\sum_{i=1}^{\infty}-\frac{\Re\left(\Tr{\rm U}_{N}^{j}\right)}{j}\right)\right)=e^{\alpha^{2}\left(\frac{1}{4}+\oo\left(1\right)\right)\log N}.\label{eq: full sum exp mom}
\end{equation}
\end{lem}
\begin{proof}[Proof of Theorem \ref{thm: high points}]
 The upper bound is direct: Fubini's theorem and rotational
invariance show that
\[
\mathbb{E}\Big(\mbox{Leb}\big(\mathcal{L}_{N}\left(\gamma\right)\big)\Big)=2\pi\mathbb{P}\left(\sum_{j=1}^{\infty}-\frac{\mbox{Re}\left(\mbox{Tr}{\rm U}_{N}^{j}\right)}{j}\ge\gamma\log N\right).
\]
The exponential Chebyshev inequality and (\ref{eq: full sum exp mom})
with $\alpha=2\gamma$ show that the latter probability is at most
$cN^{-\gamma^{2}+\frac{\varepsilon}{2}}$ for any $\varepsilon>0$ and large enough $N$, so that
\[
\text{for all  $\varepsilon>0$,} \ \ \mathbb{P}\left(\mbox{Leb}\big(\mathcal{L}_{N}\left(\gamma\big)\right)\ge N^{-\gamma^{2}+\varepsilon}\right)\le cN^{-\frac{\varepsilon}{2}}\to0,\mbox{ as }N\to\infty\ .
\]
This gives the upper bound of \eqref{eq: high points}.

The proof of the lower bound is very similar to the proof of Proposition \ref{prop: lb trunc disc},
which gave a lower bound for the maximum of the truncated sum.
Fix a $\gamma\in\left(0,1\right)$ and an $\varepsilon>0$
and recall the coarse increments $Y_{m}\left(h\right),m=1,\ldots,K,$ from  \eqref{eqn: Y},
as well as the event $J_{x}\left(h\right)$ from \eqref{eqn: Z}. Here we will use
\[
x=\frac{\gamma}{K}\left(1+\frac{\varepsilon}{3}\right)\log N,
\]
and apply the second moment method to the measure of the set 
\[
\mathcal{L}_{N}^{K}=\left\{ h\in\left[0,2\pi\right]:J_{x}\left(h\right)\mbox{ occurs}\right\} .
\]
Note that if $h\in\mathcal{L}_{N}^{K}$ and $K$ is large enough depending
on $\varepsilon$, then
\begin{equation}
\sum_{N^{1/K}\le j<N^{1-1/K}}-\frac{\mbox{Re}\left(\mbox{Tr}{\rm U}_{N}^{j}\right)}{j}\ge\frac{K-2}{K}\gamma\left(1+\frac{\varepsilon}{3}\right)\log N\ge\gamma\left(1+\frac{\varepsilon}{4}\right)\log N.\label{eq: sum}
\end{equation}
Consider the  subsets
\[
\mathcal{A}=\left\{ h\in\left[0,2\pi\right]:\sum_{j\ge N^{1-1/K}}-\frac{\mbox{Re}\left(\mbox{Tr}{\rm U}_{N}^{j}\right)}{j}\le-\frac{\varepsilon}{8}\log N\right\} ,
\]
and
\[
\mathcal{B}=\left\{ h\in\left[0,2\pi\right]:\sum_{j<N^{1/K}}-\frac{\mbox{Re}\left(\mbox{Tr}{\rm U}_{N}^{j}\right)}{j}\le-\frac{\varepsilon}{8}\log N\right\} .
\]
Because of (\ref{eq: sum}) we have (for $K$ large enough) 
\[
\mathcal{L}_{N}^{K}\subset\mathcal{L}_{N}\left(\gamma\right)\cup\mathcal{A}\cup\mathcal{B},
\]
so that
\[
\mbox{Leb}\left(\mathcal{L}_{N}\left(\gamma\right)\right)\ge\mbox{Leb}\left(\mathcal{L}_{N}^{K}\right)-\mbox{Leb}\left(\mathcal{A}\right)-\mbox{Leb}\left(\mathcal{B}\right).
\]
The bounds \eqref{eq: one point trunc tail bound} (with $\delta = 1-\log N/K$) and \eqref{eq: tail bound for tail} together with Fubini's theorem show that $\mathbb{E}\left(\mbox{Leb}\left(\mathcal{A}\right)\right)\le cN^{-cK\varepsilon^{2}}$
and $\mathbb{E}\left(\mbox{Leb}\left(\mathcal{B}\right)\right)\le cN^{-cK\varepsilon^{2}}$,
so that if $K$ is chosen large enough depending on $\varepsilon$,
\[
\lim_{N\to\infty}\mathbb{P}\left(\mbox{Leb}\left(\mathcal{A}\right)+\mbox{Leb}\left(\mathcal{B}\right)\le N^{-\gamma^{2}-\varepsilon}\right)=1.
\]
It thus suffices to show that for large enough $K$,
\begin{equation}
\lim_{N\to\infty}\mathbb{P}\left(\mbox{Leb}\left(\mathcal{L}_{N}^{K}\right)\ge2N^{-\gamma^{2}-\varepsilon}\right)=1.\label{eq: suffices to show high points}
\end{equation}
To apply the second moment method to $\mbox{Leb}\left(\mathcal{L}_{N}^{K}\right)$,
we first note that by \eqref{eq: one point precise bound} the first moment satisfies
\begin{equation}
\mathbb{E}\left(\mbox{Leb}\left(\mathcal{L}_{N}^{K}\right)\right) \ge 2\pi\left(1+{\rm o}_{K}\left(1\right)\right)e^{-\sum_{m=2}^{K-1}\frac{x^{2}}{2\sigma_{m}^{2}}}\prod_{m=2}^{K-1}\eta_{0,\sigma_{m}^{2}}\left(e^{-\frac{xy}{\sigma_{m}^{2}}}\1_{[0,\infty)}\left(y\right)\right).\label{eq: high point fm}
\end{equation}
As in the last inequality of \eqref{eq: sum}, the right-hand side is at least $N^{-\frac{K-2}{K}\gamma\left(1+\frac{\varepsilon}{3}\right)^{2}}c\left(\log N\right)^{-\left(K-2\right)}\ge N^{-\gamma^{2}-\frac{\varepsilon}{2}}$,
where the last inequality follows for $K$ large enough depending on $\varepsilon$. 
Therefore,

\[
\mathbb{E}\left(\mbox{Leb}\left(\mathcal{L}_{N}^{K}\right)\right)\ge N^{-\gamma^{2}-\frac{\varepsilon}{2}}\ .
\]
Using the Paley-Zygmund inequality as for the maximum, this proves that
\begin{equation}
 \label{eq: Chebyshev}
\mathbb{P}\left(\mbox{Leb}\left(\mathcal{L}_{N}^{K}\right)\ge N^{-\gamma^{2}-\varepsilon}\right)\geq 
\mathbb{P}\left(\mbox{Leb}\left(\mathcal{L}_{N}^{K}\right)\ge N^{-\frac{\varepsilon}{2}}\ \mathbb{E}\left(\mbox{Leb}\left(\mathcal{L}_{N}^{K}\right)\right)\right)\geq 
(1-N^{-\frac{\varepsilon}{2}})^2\frac{\mathbb{E}\left(\mbox{Leb}\left(\mathcal{L}_{N}^{K}\right)^{2}\right)}{\mathbb{E}\left(\mbox{Leb}\left(\mathcal{L}_{N}^{K}\right)\right)^{2}}\ .
\end{equation}
To obtain the lower bound of \eqref{eq: high points} it only remains to show that for all $K\ge1$, 
\begin{equation}
\mathbb{E}\left(\mbox{Leb}\left(\mathcal{L}_{N}^{K}\right)^{2}\right)=\left(1+{\rm o}_{K}\left(1\right)\right)\mathbb{E}\left(\mbox{Leb}\left(\mathcal{L}_{N}^{K}\right)\right)^{2} \ .
\label{eq: high point second moment}
\end{equation}
To prove \eqref{eq: high point second moment}, we write the second
moment as
\[
\mathbb{E}\left(\mbox{Leb}\left(\mathcal{L}_{N}^{K}\right)^{2}\right)=\int_{\left[0,2\pi\right]^{2}}\mathbb{P}\left(J_{x}\left(h_{1}\right)\cap J_{x}\left(h_{2}\right)\right){\rm d}h_{1}{\rm d}h_{2},
\]
and split the integral in analogy with \eqref{eqn: sum split}:
\[
\begin{array}{l}
\underset{\left({\rm I}\right)}{\underbrace{\int_{h_{1}\wedge h_{2}\le\frac{1}{2K}\log N}(\cdot) \ {\rm d}h_{1}{\rm d}h_{2}}}+\underset{\left({\rm II}\right)}{\underbrace{\int_{\frac{1}{2K}\log N<h_{1}\wedge h_{2}\le\frac{1}{2K}\log N}(\cdot) \ {\rm d}h_{1}{\rm d}h_{2}}}\\
+\underset{\left({\rm III}\right)}{\underbrace{\sum_{j=2}^{K-1}\int_{\frac{j-1}{K}\log N<h_{1}\wedge h_{2}\le\frac{j}{K}\log N}(\cdot) \ {\rm d}h_{1}{\rm d}h_{2}}}+\underset{\left({\rm IV}\right)}{\underbrace{\int_{h_{1}\wedge h_{2}\ge\frac{K-1}{K}\log N}(\cdot) \ {\rm d}h_{1}{\rm d}h_{2}}}\ .
\end{array}
\]
The bound \eqref{eq: two point precise bound} provides a uniform bound on the integrand in $\left({\rm I}\right)$,
showing that 
\[
\left({\rm I}\right)\le\left(2\pi\right)^{2}\left(1+{\rm o}_{K}\left(1\right)\right)  e^{-\sum_{m=2}^{K-1}\frac{x^2}{\sigma_{m}^{2}}}\left(\prod_{m=2}^{K-1}\eta_{0,\sigma_{m}^{2}}\left(e^{-\frac{xy}{\sigma_{m}^{2}}}\1_{[0,\infty)}\left(y\right)\right)\right)^{2}=\left(1+{\rm o}_{K}\left(1\right)\right)\mathbb{E}\left(\mbox{Leb}\left(\mathcal{L}_{N}^{K}\right)\right)^{2},
\]
where the equality follows by (\ref{eq: high point fm}). By \eqref{eq: two point decoupling rough} the integrand in $\left({\rm II}\right)$
is uniformly bounded by $e^{-\sum_{m=2}^{K-1}\frac{x^{2}}{\sigma_{m}^{2}}}\le c\mathbb{E}\left(\mbox{Leb}\left(\mathcal{L}_{N}^{K}\right)\right)^{2}\left(\log N\right)^{\frac{K-2}{2}}$,
and since the measure of the set integrated over is at most $\left(2\pi\right)^{2}N^{-\frac{1}{2K}}$
the integral $\left({\rm II}\right)$ is ${\rm o}\left(\mathbb{E}\left(\mbox{Leb}\left(\mathcal{L}_{N}^{K}\right)\right)^{2}\right)$.
Similarily, the measure of the set integrated over in the $j-$th term
of $\left({\rm III}\right)$ is most $\left(2\pi\right)^{2}N^{-\frac{j-1}{K}}$.
Therefore by bounding the
integrand using \eqref{eq: two point rough 2}, we see similarly to \eqref{eq: third} that the integral $\left({\rm III}\right)$ is ${\rm o}\left(\mathbb{E}\left(\mbox{Leb}\left(\mathcal{L}_{N}^{K}\right)\right)^{2}\right)$.
Lastly $\left({\rm IV}\right)$ can be bounded by 
\[
\int_{h_{1}\wedge h_{2}\ge\frac{K-1}{K}\log N}\mathbb{P}\left(J_{x}\left(h_{1}\right)\right){\rm d}h_{1}{\rm d}h_{2}=\mathbb{E}\left(\mbox{Leb}\left(\mathcal{L}_{N}^{K}\right)\right)\left(2\pi\right)^{2}N^{-\frac{K-1}{K}},
\]
which is also ${\rm o}\left(\mathbb{E}\left(\mbox{Leb}\left(\mathcal{L}_{N}^{K}\right)\right)^{2}\right)$,
see (\ref{eq: high point fm}). This proves (\ref{eq: high point second moment}),
and therefore also the lower bound of \eqref{eq: high points}.
\end{proof}
We can now derive Corollary \ref{cor: free energy} about the free energy.
\begin{proof}[Proof of Corollary \ref{cor: free energy}]
The proof uses Laplace's method. Let $\varepsilon>0$ and fix $M\in \N$. We define levels
$$
\gamma_j=(1+\varepsilon)\frac{j}{M} \qquad j=0,\dots, M\ ,
$$
and the event
$$
E=\bigcap_{j=1}^{M-1} \left\{ N^{-\gamma_j^2-\varepsilon}\leq \mbox{Leb}\left(\mathcal{L}_{N}\left(\gamma_j\right)\right)\leq N^{-\gamma_j^2 + \varepsilon}\right\} 
\cap \{\mbox{Leb}\left(\mathcal{L}_{N}\left(\gamma_M\right)\right)=0\}\ .
$$
Theorems \ref{thm: main} and \ref{thm: high points} imply that $\P( E^c)\to 0$ as $N\to \infty$. Therefore it suffices to prove the result on the event $ E$. 
The integrals $\frac{1}{2\pi}\int_0^{2\pi} |{\rm P}_N(h)|^\beta \rd h$ can be split in terms of the subsets
$\{h\in [0,2\pi]:  |{\rm P}_N(h)| < 1 \}$
and
$\{h\in [0,2\pi]: N^{\gamma_{j-1}} \leq |{\rm P}_N(h)| < N^{\gamma_j} \}$.
We get the upper bound
\begin{equation}
\label{eqn: free upper}
\begin{aligned}
\frac{1}{2\pi}\int_0^{2\pi} |{\rm P}_N(h)|^\beta \rd h
&\leq
\frac{1}{2\pi}\int_0^{2\pi} |{\rm P}_N(h)|^\beta \1_{\{ |{\rm P}_N(h)| <1 \}}\rd h
+
\sum_{j=1}^M \frac{1}{2\pi}\int_0^{2\pi} |{\rm P}_N(h)|^\beta \1_{\{ N^{\gamma_{j-1}}\leq |{\rm P}_N(h)| < N^{\gamma_j} \}}\rd h \\
&\leq
1 + \sum_{j=1}^M N^{\beta \gamma_j} \mbox{Leb}\left(\mathcal{L}_{N}\left(\gamma_{j-1}\right)\right)
\leq 1+\sum_{j=1}^M N^{\beta \gamma_j-\gamma^2_{j-1}+\varepsilon}\ ,
\end{aligned}
\end{equation}
where the last inequality holds on the event $E$. Similarly we have on this event that
\begin{equation}
\label{eqn: free lower}
\frac{1}{2\pi}\int_0^{2\pi} |{\rm P}_N(h)|^\beta \rd
\ge
\sum_{j=1}^M N^{\beta \gamma_{j-1}} \mbox{Leb}\left(\mathcal{L}_{N}\left(\gamma_{j}\right)\right)
\ge
\sum_{j=1}^M N^{\beta \gamma_{j-1}-\gamma^2_{j}-\varepsilon}\ .
\end{equation}
Equations \eqref{eqn: free upper} and \eqref{eqn: free lower} imply for $M$ fixed
$$
\max_{j=1,\dots, M}\left\{ 1+\beta\gamma_{j-1}-\gamma^2_{j}-\varepsilon\right\}+\oo(1)\leq \frac{1}{\log N}\log \left(\frac{N}{2\pi}\int_0^{2\pi} |{\rm P}_N(h)|^\beta \rd h\right)\leq \max_{j=1,\dots, M}\left\{ 1+\beta\gamma_{j}-\gamma^2_{j-1}+\varepsilon\right\} +\oo(1)\ .
$$
After taking the limit $N\to\infty$, $M\to \infty$, and finally $\varepsilon \to 0$, we get
$$
\lim_{M\to\infty} \frac{1}{\log N}\log \left(\frac{N}{2\pi}\int_0^{2\pi} |{\rm P}_N(h)|^\beta \rd h\right)
=\max_{\gamma\in[0,1]}\{1+\beta\gamma-\gamma^2\} =
\begin{cases}
1+\frac{\beta^2}{4} \ & \text{ if $\beta <2$}\\
\beta \ & \text{ if $\beta \geq 2$}\\
\end{cases}
\ ,
$$
which proves the corollary.
\end{proof}

\noindent Up to Lemma \ref{lem: Full sum exp mom} (and the other Gaussian estimates from the previous sections), we have now completed the proof of
Theorem \ref{thm: high points} and Corollary \ref{cor: free energy}.
\section[Estimates on increments and tails]{Estimates on increments and tails}
\label{sec: RH}

This
section
proves two important estimates.
\begin{enumerate}[(i)]
\item Proposition \ref{prop:tail} 
(which is a reformulation of Lemma \ref{lem: mgf ub tail})
justifies the approximation of the characteristic polynomial by partial sums, by 
bounding the contribution of high powers in the Fourier expansion.
We  follow the proof of asymptotics of 
Toeplitz determinants when the symbols have a Fisher-Hartwig singularity as given in \cite{MarMcLSaf2006,DeiItsKra2011,DeiItsKra2014}. In particular we rely on the Riemann-Hilbert problem developped 
in these works. We will also make use of a key differential identity from \cite{DeiItsKra2014}. 
\item Proposition \ref{thm:FourierRH} gives appropriate asymptotics for the
 joint Laplace and Fourier transforms of
sums involving only traces of small powers, at possibly different evaluation points $h$.
 It will be used to prove the previously stated bounds
Lemmas \ref{lem: gaussian UB}, \ref{lem: two point fourier bound} and \ref{lem: exp moment continuity}
involving the truncated sum of traces and the increments $Y_m(h)$.
The proof of (ii) is easier than that of (i): here the symbols have no singularity
\end{enumerate}

There are
several
technical differences with \cite{MarMcLSaf2006,DeiItsKra2011,DeiItsKra2014}. 
First, we only need to consider one Fisher-Hartwig singularity (at $z=1$), which simplifies the analysis. 
Second, our external potential $V$ depends on $N$ and it is close to singular: it corresponds to a  singularity smoothed on the mesoscopic scale $N^{-1+\delta}$, $\delta>0$ arbitrarily small.  Consequently, 
the contour of the Riemann-Hilbert problem is $N$-dependent. Third, a small modification of the Riemann-Hilbert problem considered in \cite{MarMcLSaf2006,DeiItsKra2011,DeiItsKra2014}
will be necessary to obtain better error terms (see (\ref{eqn:conjugacy})).

\subsection{Setting.\ } \label{subsec:Set}
The main task of this section is the following exponential moment estimate.

\begin{prop}
\label{prop:tail}
For any fixed (small) $\delta \in (0,1)$ and (large) $\alpha$, for large enough $N$ we have
\begin{equation}\label{eq: mgf ub tail}
  \E\left( \exp\left(2\alpha \sum_{j \ge N^{1-\delta} } -\frac{\Re(\Tr {\rm U}_N^j)}{j} \right) \right)
 = \exp\left( (1+\oo(1))\alpha^2 \delta \log N \right),
\end{equation}
\end{prop}
\noindent Note that this implies Lemma \ref{lem: mgf ub tail} by rotational invariance. For the proof, we first introduce the following notations.
Let $V$ be analytic in a neighborhood of the unit circle, $\mathscr{C}$. We will actually consider
\begin{equation}\label{eqn:Uext}
V(z)=\lambda\sum_{j=1}^{N^{1-\delta}}\frac{z^j+z^{-j}}{j},\ {\rm i.e.}\ V_j=\frac{\lambda}{|j|}\1_{|j|\leq N^{1-\delta}},
\end{equation}
for $0\leq\lambda\leq\alpha$,
but
most
of subsections \ref{subsec:Set} and \ref{subsec:RH} is independent of this specific form of $V$.
We also define
\begin{align}
&e^{V^{(t)}(z)}=1-t+t e^{V(z)},\ 0\leq t\leq 1\label{eqn:Vt}\\
&f^{(t)}(z)=e^{V^{(t)}(z)}|z-1|^{2\alpha}\ .
\label{eqn:interpol}
\end{align}
Note that  $V^{(1)}(z)=V(z)$ and $V^{(t)}(z)$ is clearly well defined for $|z|=1$, because $1-t+t e^{V(z)}>0$.  With these definitions in mind, instead of (\ref{eq: mgf ub tail}) we will prove the equivalent form (for $\lambda=\alpha$)
\begin{equation}\label{eqn:equivalent}
\E\left(\prod_{k=1}^N f^{(1)}(e^{\ii\theta_k})\right)=\exp\left( (1+\oo(1))\alpha^2 \delta \log N \right)\ .
\end{equation}
For the proof, we will use a Riemann-Hilbert approach and therefore need the following lemma about an analytic extension of $V^{(t)}$.

\begin{lem}\label{lem:analytic}
There exists $\epsilon>0$ small enough (depending on the fixed constants $\delta$ and $\alpha$) such that for any $0\leq t\leq 1$, $V^{(t)}$ admits an analytic extension to
$||z|-1|<\epsilon N^{-1+\delta}$.
\end{lem}

\begin{proof}
We only need to prove 
\begin{equation}\label{eqn:nonZero}
1-t+t e^{V(z)}\neq 0
\end{equation}
for any $||z|-1|<\epsilon N^{-1+\delta}$, $0\leq t\leq 1$.
Note that $\sup_{||z|-1|<N^{-1+\delta}} |V'(z)|=\OO(N^{1-\delta})$ and $V(z)\in\mathbb{R}$ if $|z|=1$, so that for small enough $\epsilon$ we have
\begin{equation}\label{eqn:ImV0}
\sup_{||z|-1|<\epsilon N^{-1+\delta}} |{\rm Im} V(z)|<1/10.
\end{equation}
The above condition clearly implies (\ref{eqn:nonZero}).
\end{proof}

\noindent In the following, we will abbreviate $\kappa=\epsilon N^{-1+\delta}$, with $\e$ so that the above lemma holds.
Define the Wiener-Hopf factorization of $e^{V^{(t)}}$:
\begin{align}
b^{(t)}_{+}(z)&=\exp\left(-V_0+\int_{\mathscr{C}}\frac{1}{2\pi\ii}\frac{V^{(t)}(s)}{s-z}\rd s\right)\ {\rm for}\ |z|<1,\label{b+}\\
b^{(t)}_{-}(z)&=\exp\left(\frac{1}{2\pi\ii}\int_{\mathscr{C}}\frac{V^{(t)}(s)}{s-z}\rd s\right)\ {\rm for}\ |z|>1,\label{b-}
\end{align}
where $\mathscr{C}$ is the unit circle.
Then 
\begin{equation}\label{eqn:b+b-}
b^{(t)}_{+}(z)=e^{\sum_{k=1}^\infty V^{(t)}_k z^k},
b^{(t)}_{-}(z)= e^{\sum_{k=-\infty}^{-1} V^{(t)}_k z^k}
\end{equation}
where $V^{(t)}_k=\frac{1}{2\pi}\int_0^{2\pi}V^{(t)}(e^{\ii\theta})e^{-\ii k\theta}\rd \theta$. By Lemma \ref{lem:analytic}, both functions 
can be slightly extended analytically: $b^{(t)}_{+}$ to $|z|<1+\kappa$ and $b^{(t)}_{-}$ to $|z|>1-\kappa$.
In the domain $||z|-1|<\kappa$, they satisfy the following properties:
\begin{align}
&e^{V^{(t)}(z)}=b^{(t)}_{+}(z)e^{V^{(t)}_0}b^{(t)}_{-}(z),\label{elem1}\\
&b^{(t)}_{+}(z)=\overline{b^{(t)}_{-}(\bar z^{-1})}\label{elem2}.
\end{align}

The proof of (\ref{eqn:equivalent}) proceeds by interpolation through the parameter $t$. 
Indeed, writing 
\begin{equation}\label{eqn:DN}
{\rm D}_N(f^{(t)})=\E\left(\prod_{k=1}^N f^{(t)}(e^{\ii\theta_k})\right),
\end{equation}
a formula from \cite{DeiItsKra2014} gives
$$
\log {\rm D}_N(f^{(1)})-\log{\rm D}_N(f^{(0)})
=
N V_0+\sum_{k=1}^\infty k V_kV_{-k}-\alpha \log (b^{(1)}_{+}(1)b^{(1)}_{-}(1))
+\int_0^1 E(t) \rd t\ ,
$$
where the error term $E(t)$ can be expressed from the solution of a Riemann-Hilbert problem, see Corollary \ref{cor:fundamentalerror}. Hence the main part of the analysis consists in bounding this solution, which is performed in the next two subsections.
The proof is then concluded, up to the initial value $\log{\rm D}_N(f^{(0)})$ in the interpolation.
This term is a Selberg integral, and its asymptotics are given in Lemma \ref{lem:Selberg}.

\subsection{The Riemann-Hilbert problem.\ }\label{subsec:RH}Consider the Szeg{\H o} function
\begin{equation}\label{eqn:Szego}
\exp(g(z))=\exp \left( \frac{1}{2\pi \ii}\int_{\mathscr{C}}\frac{\ln f^{(t)}(s)}{s-z}\rd s \right)
=
\left\{
\begin{array}{ll}
e^{V^{(t)}_0}b^{(t)}_+(z)(z-1)^{\alpha}e^{-\ii\pi\alpha}&{\rm if}\ |z|<1,\\
b^{(t)}_-(z)^{-1}(z-1)^{-\alpha}z^{\alpha}&{\rm if}\ |z|>1.\\
\end{array}
\right.
\end{equation}
In the above definition and for further use, the cut of $(z-1)^\alpha$ goes from $1$ to $\infty$ along $\theta=0$. We fix the branches by $0<\arg (z-1)<2\pi$, and for $z^{\alpha/2}$, $0<\arg (z-1)<2\pi$ as well.
Let 
$$
\zeta=N\log z,
$$
where $\log x>0$ for $x>1$,  and has a cut on the negative half of the real axis. We define the  analytic continuation of the function
$$
h_\alpha(z)=|z-1|^\alpha
$$
through $|z-1|^\alpha=(z-1)^{\alpha/2}(z^{-1}-1)^{\alpha/2}=\frac{(z-1)^\alpha}{z^{\alpha/2}e^{\ii\pi\alpha/2}}$, in a neighborhood of the open arc
$\{|z|=1,|z|\neq 1\}$. The factor $e^{\ii\pi\alpha/2}$ is chosen so that $h_\alpha(z)$ has null argument on the unit circle.
Moreover, let
\begin{equation}\label{eqn:F}
F(z)=\left\{
\begin{array}{ll}
e^{\frac{V^{(t)}(z)}{2}}h_\alpha(z)e^{-\ii\pi\alpha}&{\rm if}\ \xi\in{\rm I}, {\rm II}, {\rm V}, {\rm VI},\\
e^{\frac{V^{(t)}(z)}{2}}h_\alpha(z)e^{\ii\pi\alpha}&{\rm if}\ \xi\in{\rm III}, {\rm IV}, {\rm VII}, {\rm VIII},
\end{array}
\right.
\end{equation}
where we used the definitions of Figure \ref{fig:octants}, and the conventions for cuts and branches were explained previously.

Let $\psi(a,b,x)$ be the confluent hypergeometric function of the second kind. Define $\Psi$ to be the analytic function such that for any $\zeta\in {\rm I}$ (see Figure \ref{fig:octants}) we have
\begin{equation}\label{eqn:Psi}
\Psi(\zeta)=
\left(
\begin{array}{ll}
\zeta^\alpha\psi(\alpha,2\alpha+1,\zeta)e^{\ii\pi\alpha}e^{-\zeta/2}&
-\zeta^\alpha\psi(\alpha+1,2\alpha+1, e^{-\ii\pi}\nc\zeta)e^{\ii\pi\alpha}e^{\zeta/2}\alpha\\
-\zeta^{-\alpha}\psi(-\alpha+1,-2\alpha+1,\zeta)e^{-3\ii\pi\alpha}e^{-\zeta/2}\alpha&
\zeta^{-\alpha}\psi(-\alpha,-2\alpha+1, e^{-\ii\pi}\nc\zeta)e^{-\ii\pi\alpha}e^{\zeta/2}\\
\end{array}
\right)
\end{equation}
and
\begin{equation*}\displaywidth=\parshapelength\numexpr\prevgraf+2\relax
\Psi_+(\zeta)=
\Psi_-(\zeta)
K(\zeta)
\end{equation*}
\setlength{\columnsep}{30pt}
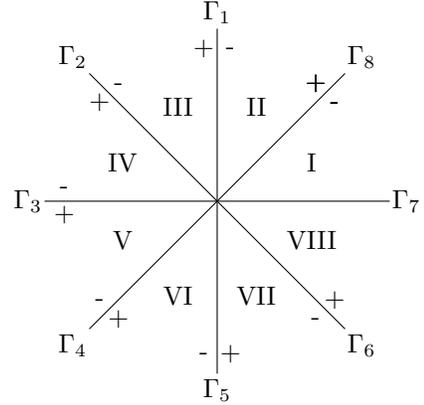
\begin{wrapfigure}[9]{r}{.4\linewidth}
\vspace{-1.5cm}
\begin{center}
\begin{tikzpicture}[xscale=0.017,yscale=0.017]
\clip (-160,-160) rectangle (170,160); 

\node at (113,113) {$\Gamma_8$};
\node at (-113,113) {$\Gamma_2$};
\node at (-113,-113) {$\Gamma_4$};
\node at (113,-113) {$\Gamma_6$};
\node at (0,148) {$\Gamma_1$};
\node at (-148,0) {$\Gamma_3$};
\node at (0,-148) {$\Gamma_5$};
\node at (148,0) {$\Gamma_7$};

\node at (22.5: 80){I};
\node at (67.5: 80){II};
\node at (112.5: 80){III};
\node at (157.5: 80){IV};
\node at (202.5: 80){V};
\node at (247.5: 80){VI};
\node at (-67.5: 80){VII};
\node at (-22.5: 80){VIII};

\node at (40: 120){-};
\node at (50: 120){+};

\node at (85: 120){-};
\node at (95: 120){+};

\node at (130: 120){-};
\node at (140: 120){+};

\node at (175: 120){-};
\node at (185: 120){+};

\node at (220: 120){-};
\node at (230: 120){+};

\node at (265: 120){-};
\node at (275: 120){+};

\node at (310: 120){-};
\node at (320: 120){+};

\node at (40: 120){-};
\node at (50: 120){+};

\node at (40: 120){-};
\node at (50: 120){+};

\draw (-135,0) -- (135,0);
\draw (0,-135) -- (0,135);
\draw (-100,-100) -- (100,100);
\draw (100,-100) -- (-100,100);

\end{tikzpicture}
\end{center}
\vspace{-0.7cm}
\caption{Auxiliary contours in the variable $\zeta$ around $0$ (i.e. a neighborhood of the singularity $z=1$)}\label{fig:octants}
\end{wrapfigure}
where
$$\displaywidth=\parshapelength\numexpr\prevgraf+2\relax
\begin{array}{ll}
K(\zeta)=\left(
\begin{array}{cc}
0&1\\
-1&0
\end{array}
\right)&{\rm if}\ \zeta\in \Gamma_1\cup\Gamma_5\\
K(\zeta)=\left(
\begin{array}{cc}
e^{\ii\pi\alpha}&0\\
0&e^{-\ii\pi\alpha}
\end{array}
\right)&{\rm if}\ \zeta\in \Gamma_3\cup\Gamma_7\\
K(\zeta)=\left(
\begin{array}{cc}
1&0\\
e^{\ii\pi2\alpha}&1
\end{array}
\right)&{\rm if}\ \zeta\in \Gamma_4\cup\Gamma_8\\
K(\zeta)=\left(
\begin{array}{cc}
1&0\\
e^{-\ii\pi2\alpha}&1
\end{array}
\right)&{\rm if}\ \zeta\in \Gamma_2\cup\Gamma_6\ .
\end{array}
$$

Let $\sigma_3=\left(\begin{array}{cc}1&0\\0&-1\end{array}\right)$ and denote $z^{\sigma_3}=
\left(\begin{array}{cc}z&0\\0&z^{-1}\end{array}\right)$. Define
$$
\begin{array}{ll}
N(z)=e^{g(z)\sigma_3}&{\rm if}\ |z|>1\\
N(z)=e^{g(z)\sigma_3}\left(
\begin{array}{cc}
0&1\\
-1&0
\end{array}
\right)&{\rm if}\ |z|<1\ .
\end{array}
$$

\setlength{\columnsep}{30pt}
\begin{wrapfigure}[10]{r}{.4\linewidth}
\vspace{-1.5cm}
\begin{center}
\begin{tikzpicture}[xscale=0.02,yscale=0.02]
\clip (-135,-135) rectangle (170,135); 
\fill [color=black] (0,0) circle (3);

\draw [very thick] (0,0) circle (130);
\draw [very thick] (0,0) circle (70);

\fill [color=white] (100,0) circle (50);
\draw [very thick] (100,0) circle (50);
\fill [color=black] (100,0) circle (3);
\draw [dotted,thick] (0,0) circle (100);
\node at (-30,0) {$z=0$};
\node at (85,0) {$1$};
\node at (114,25) {U};
\node at (-36,-36) {$\Sigma_{\rm out}''$};
\node at (-103,-103) {$\Sigma_{\rm out}$};
\node at (42,-42) {$+$};
\node at (56,-56) {$-$};
\node at (84,-84) {$+$};
\node at (101,-101) {$-$};
\node at (126,-26) {$-$};
\node at (142,-42) {$+$};

\end{tikzpicture}
\end{center}
\vspace{-0.7cm}
\caption{Contour $\Gamma$ for the $R$-Riemann-Hilbert problem}
\label{fig: gamma}
\end{wrapfigure}
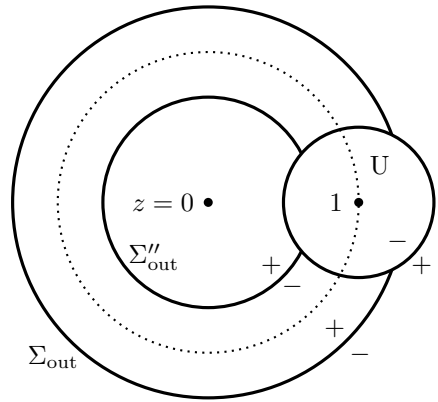

We will also need the notation
\begin{equation}\label{eqn:E}
\begin{array}{ll}
E(z)=N(z)F(z)^{\sigma_3}\left(
\begin{array}{cc}
e^{-\ii\pi\alpha}&0\\
0&e^{2\ii\pi\alpha}
\end{array}
\right)&{\rm if}\ z\in {\rm I,\ II}\\
E(z)=N(z)F(z)^{\sigma_3}\left(
\begin{array}{cc}
e^{-\ii\pi2\alpha}&0\\
0&e^{\ii\pi3\alpha}
\end{array}
\right)&{\rm if}\ z\in {\rm III,\ IV}\\
E(z)=N(z)F(z)^{\sigma_3}\left(
\begin{array}{cc}
0&e^{\ii\pi3\alpha}\\
-e^{-\ii\pi2\alpha}&0
\end{array}\right)&{\rm if}\ z\in {\rm V,\ VI}\\
E(z)=N(z)F(z)^{\sigma_3}\left(
\begin{array}{cc}
0&e^{\ii\pi2\alpha}\\
-e^{-\ii\pi\alpha}&0
\end{array}
\right)&{\rm if}\ z\in {\rm VII,\ VIII}\ .
\end{array}
\end{equation}

Finally, consider the jump matrix
\begin{equation}\label{eqn:matrixM}
M(z)=E(z)\Psi(z)F(z)^{-\sigma_3}z^{\pm N\sigma_3}N(z)^{-1}
\end{equation}
where the plus sign is taken for $|z|<1$, and minus for $|z|>1$.

We define the contour $\Gamma$ in $\mathbb{C}$ as follows (see Figure \ref{fig: gamma}): 
it consists in the boundary $\partial {\rm U}$ of a disk ${\rm U}$ centered at 1 with radius $\kappa$, the arc of circle $\Sigma_{\rm out}$ (resp. $\Sigma_{\rm out}''$) centered at $0$ with radius $1+(2/3) \kappa$ (resp. $1-(2/3) \kappa$),
outside ${\rm U}$ with extremities on ${\rm U}$.

Consider the following Riemann-Hilbert problem for the $2\times 2$ matrix valued function $R$.
\begin{enumerate}
\item $R$ is analytic for $z\in\mathbb{C}\backslash \Gamma$.
\item The boundary values of $R$ are related by the jump condition
\begin{equation*}\displaywidth=\parshapelength\numexpr\prevgraf+2\relax
\begin{array}{ll}
R_+(z)=R_-(z)N(z)\left(\begin{array}{cc}1&0\\f^{(t)}(z)^{-1}z^{-N}&1\end{array}\right)N(z)^{-1}&{\rm if\ } z\in\Sigma_{\rm out}\\
R_+(z)=R_-(z)N(z)\left(\begin{array}{cc}1&0\\f^{(t)}(z)^{-1}z^{N}&1\end{array}\right)N(z)^{-1}&{\rm if\ } z\in\Sigma''_{\rm out}\\
R_+(z)=R_-(z)M(z),&{\rm if\ } z\in \partial {\rm U}\ .
\end{array}
\end{equation*}
\item $R(z)={\rm Id}+{\rm O}(1/z)$ as $z\to\infty$.
\end{enumerate}

\begin{wrapfigure}[14]{r}{.6\linewidth}
\vspace{-1cm}
\begin{center}
\tikzset{->-/.style=
{decoration={
markings,
mark=at position #1 with
{\arrow{>}}}, postaction=
{decorate}
}
}
\tikzset{-<-/.style=
{decoration={
markings,
mark=at position #1 with
{\arrow{<}}}, postaction=
{decorate}
}
}
\begin{tikzpicture}[xscale=0.04,yscale=0.04]
\clip (20,-75) rectangle (170,75);

\draw [very thick] (0,0) circle (130);
\draw [very thick] (0,0) circle (70);
\draw [very thick] (0,0) circle (100);

\fill [color=white] (100,0) circle (50);
\draw [very thick] (100,0) circle (50);
\node at (40,0) {$\mathscr{C}''_1$};
\node at (160,0) {$\mathscr{C}_1$};

\node at (65,50) {$\mathscr{C}''_2$};
\node at (65,-48) {$\mathscr{C}''_2$};

\node at (106,57) {$\mathscr{C}_2$};
\node at (106,-57) {$\mathscr{C}_2$};

\node at (47,-45) {$+$};
\node at (54,-52) {$-$};

\node at (73,-61) {$+$};
\node at (81,-67) {$-$};

\node at (77,49) {$+$};
\node at (79,41) {$-$};

\node at (77,-49) {$+$};
\node at (81,-41) {$-$};

\node at (94,-55) {$+$};
\node at (94,-47) {$-$};

\node at (106,-66) {$+$};
\node at (115,-73) {$-$};
\node at (142,-15) {$-$};
\node at (48,-20) {$+$};
\node at (58,-15) {$-$};
\node at (152,-20) {$+$};
\fill [color=black] (100,0) circle (2);

\draw[very thick,->-=.5] ([shift=(-30:50)]100,0) arc (-30:30:50);
\draw[very thick,->-=.5] ([shift=(210:50)]100,0) arc (210:150:50);
\draw[very thick,->-=.5] ([shift=(80:50)]100,0) arc (80:95:50);
\draw[very thick,->-=.5] ([shift=(-95:50)]100,0) arc (-95:-80:50);
\draw[very thick,->-=.5] ([shift=(130:50)]100,0) arc (130:115:50);
\draw[very thick,->-=.5] ([shift=(-115:50)]100,0) arc (-115:-130:50);

\draw[very thick,->-=.5] ([shift=(45:70)]0,0) arc (45:60:70);
\draw[very thick,->-=.5] ([shift=(35:100)]0,0) arc (35:40:100);
\draw[very thick,->-=.5] ([shift=(20:130)]0,0) arc (20:35:130);

\draw[very thick,->-=.5] ([shift=(-60:70)]0,0) arc (-60:-45:70);
\draw[very thick,->-=.5] ([shift=(-40:100)]0,0) arc (-40:-30:100);
\draw[very thick,->-=.5] ([shift=(-30:130)]0,0) arc (-30:-20:130);

\node at (34,52) {$\Sigma_{\rm out}''$};
\node at (64,66) {$\Sigma_{\rm out}'$};
\node at (124,68) {$\Sigma_{\rm out}$};

\draw[very thick,->-=.5] plot [smooth, tension=1] coordinates { (100,0) (115,21)(121.5,46) };
\draw[very thick,-<-=.5] plot [smooth, tension=1] coordinates { (100,0) (115,-21)(121.5,-46) };

\draw[very thick,->-=.5] plot [smooth, tension=1] coordinates { (100,0) (85,20)(62,33) };
\draw[very thick,-<-=.5] plot [smooth, tension=1] coordinates { (100,0) (85,-20)(62,-33) };

\node at (127,24) {$\Sigma_{\rm in}$};
\node at (127,-24) {$\Sigma_{\rm in}$};

\node at (75,16) {$\Sigma''_{\rm in}$};
\node at (77,-11) {$\Sigma''_{\rm in}$};

\node at (96,15) {$-$};
\node at (88,9) {$+$};

\node at (96,-15) {$-$};
\node at (88,-9) {$+$};

\node at (107,16) {$+$};
\node at (115,11) {$-$};

\node at (107,-16) {$+$};
\node at (115,-11) {$-$};

\end{tikzpicture}
\end{center}
\vspace{-0.6cm}
\caption{Contours definition and orientation in Proposition \ref{prop:fundamentalerror}}
\label{fig:orientation}
\end{wrapfigure}

It was proved in \cite{DeiItsKra2011} that there exists a unique solution to this Riemann-Hilbert problem and the solution satisfies
\begin{equation}\label{eqn:det1}\displaywidth=\parshapelength\numexpr\prevgraf+2\relax
\det R(z)=1.
\end{equation}

The following decomposes the ratio ${\rm D}_N(f^{(1)})/{\rm D}_N(f^{(0)})$ into the main contribution and some error term. It is just a restatement of a differential identity from \cite{DeiItsKra2014}, keeping all error terms explicit.
We denote $\dot f=\partial f/\partial t$ and $f'=\partial f/\partial z$.

\begin{prop}\label{prop:fundamentalerror}
We have
\begin{equation}\label{eqn:diffid}
\log {\rm D}_N(f^{(1)})-\log{\rm D}_N(f^{(0)})
=
N V_0+\sum_{k=1}^\infty k V_kV_{-k}-\alpha \log (b^1_+(1)b^1_-(1))
+\int_0^1 E(t) \rd t.
\end{equation}
The error term $E(t)$ is defined as (see Figure \ref{fig:orientation} for the definition and orientation of the contours)
\begin{align*}
E(t)=&-\int_{\mathscr{C}_1}(R_{11}R_{22}'-R_{12}'R_{21})_-\frac{\dot f^{(t)}}{f^{(t)}}\frac{\rd z}{\ii 2\pi }
+\int_{\mathscr{C}_1''}(R_{11}'R_{22}-R_{12}R_{21}')_-\frac{\dot f^{(t)}}{f^{(t)}}\frac{\rd z}{\ii2\pi }\\
&-\int_{\mathscr{C}_2}(R_{11}R_{22}'-R_{12}'R_{21})_+\frac{\dot f^{(t)}}{f^{(t)}}\frac{\rd z}{\ii 2\pi }
+\int_{\mathscr{C}''_2}(R_{11}'R_{22}-R_{12}R_{21}')_+\frac{\dot f^{(t)}}{f^{(t)}}\frac{\rd z}{\ii2\pi }\\
&+\int_{\Sigma_{\rm out}} z^{-N} \frac{e^{-2g}}{f^{(t)}}\left(R_{22}'R_{12}-R_{12}'R_{22}\right)_+\frac{\dot f^{(t)}}{f^{(t)}}\frac{\rd z}{\ii2\pi }
+\int_{\Sigma''_{\rm out}}z^{N} \frac{e^{2g}}{f^{(t)}}\left(R_{21}'R_{11}-R_{11}'R_{21}\right)_-\frac{\dot f^{(t)}}{f^{(t)}}\frac{\rd z}{\ii2\pi }\\
&+\int_{\Sigma'_{\rm out}}\left((R_{11}'R_{22}-R_{12}R_{21}')-(R_{11}R_{22}'-R_{12}'R_{21})\right)\frac{\dot f^{(t)}}{f^{(t)}}\frac{\rd z}{\ii 2\pi}\ .
\end{align*}
\end{prop}

\begin{proof}
Define 
$$
S(z)=R(z)N(z),\ I=(S_{22}'S_{12}-S_{12}'S_{22})/f^{(t)},\ J=S_{22}'S_{11}-S_{12}'S_{21}. 
$$
Simple calculations give 
$$
\begin{array}{ll}
I=\frac{e^{-2g}}{f^{(t)}}(R_{22}'R_{12}-R_{22}R_{12}')&{\rm for}\ |z|>1\\
I=\frac{e^{2g}}{f^{(t)}}(R_{11}R_{21}'-R_{11}'R_{21})&{\rm for}\ |z|<1\\
J=-g'+(R_{11}R_{22}'-R_{12}'R_{21})&{\rm for}\ |z|>1\\
J=g'+(R_{11}'R_{22}-R_{12}R_{21}')&{\rm for}\ |z|<1\\
\end{array}
$$
From \cite[equations (5.69), (5.70), (5.71)]{DeiItsKra2014} we have
\begin{equation}\label{eqn:partial}
\frac{\partial}{\partial t}\log D_N(f^{(t)})=
N\int_{\mathscr{C}}\frac{\dot f^{(t)}}{f^{(t)}}\frac{\rd z}{\ii 2\pi z}
+\int_{\Sigma}(-J_++z^{-N} I_+)\frac{\dot f^{(t)}}{f^{(t)}}\frac{\rd z}{\ii 2\pi}
+\int_{\Sigma''}(J_-+z^{N} I_-)\frac{\dot f^{(t)}}{f^{(t)}}\frac{\rd z}{\ii 2\pi},
\end{equation}
where $\Sigma=\Sigma_{\rm out}\cup\Sigma_{\rm in}$, $\Sigma''=\Sigma''_{\rm out}\cup\Sigma''_{\rm in}$
We first consider the contribution from $\Sigma_{\rm in}$ and $\Sigma''_{\rm in}$.
From \cite[equations (5.73), (5.77)]{DeiItsKra2014},
\begin{equation}\label{eqn:ugly1}
\int_{\Sigma''_{\rm in}}(J_-+z^{N} I_-)\frac{\dot f^{(t)}}{f^{(t)}}\frac{\rd z}{\ii 2\pi}
=
\int_{\mathscr{C}''_1}J\frac{\dot f^{(t)}}{f^{(t)}}\frac{\rd z}{\ii 2\pi}
=
\int_{\mathscr{C}''_1}g'\frac{\dot f^{(t)}}{f^{(t)}}\frac{\rd z}{\ii 2\pi}
+
\int_{\mathscr{C}''_1}\left(R_{11}'R_{22}-R_{12}R_{21}'\right)_-\frac{\dot f^{(t)}}{f^{(t)}}\frac{\rd z}{\ii 2\pi}.
\end{equation}
In the same way we have
\begin{equation}\label{eqn:ugly2}
\int_{\Sigma_{\rm in}}(-J_++z^{-N} I_+)\frac{\dot f^{(t)}}{f^{(t)}}\frac{\rd z}{\ii 2\pi}
=
\int_{\mathscr{C}_1}g'\frac{\dot f^{(t)}}{f^{(t)}}\frac{\rd z}{\ii 2\pi}
-
\int_{\mathscr{C}_1}\left(R_{11}'R_{22}'-R_{12}'R_{21}'\right)_-\frac{\dot f^{(t)}}{f^{(t)}}\frac{\rd z}{\ii 2\pi}.
\end{equation}

Concerning the contribution from $\Sigma_{\rm out}$ and $\Sigma''_{\rm out}$ in (\ref{eqn:partial}), we first 
note that
\begin{align}\label{eqn:ugly3}
&\int_{\Sigma_{\rm out}''}z^{-N} I_+\frac{\dot f^{(t)}}{f^{(t)}}\frac{\rd z}{\ii 2\pi}
=
\int_{\Sigma_{\rm out}''}z^{-N} \frac{e^{-2g}}{f^{(t)}}(R_{22}'R_{12}-R_{22}R_{12}')\frac{\dot f^{(t)}}{f^{(t)}}\frac{\rd z}{\ii 2\pi},\\
&\int_{\Sigma_{\rm out}}z^{N} I_-\frac{\dot f^{(t)}}{f^{(t)}}\frac{\rd z}{\ii 2\pi}\label{eqn:ugly4}
=
\int_{\Sigma_{\rm out}''}z^{N} \frac{e^{2g}}{f^{(t)}}(R_{11}R_{21}'-R_{11}'R_{11})\frac{\dot f^{(t)}}{f^{(t)}}\frac{\rd z}{\ii 2\pi}.
\end{align}
Finally, by a contour deformation, 
\begin{align}
&\int_{\Sigma_{\rm out}}(-J_+)\frac{\dot f^{(t)}}{f^{(t)}}\frac{\rd z}{\ii 2\pi}+\int_{\Sigma''_{\rm out}}(J_-)\frac{\dot f^{(t)}}{f^{(t)}}\frac{\rd z}{\ii 2\pi}\notag\\
=&
\int_{\Sigma'_{\rm out}}(J_+-J_-)\frac{\dot f^{(t)}}{f^{(t)}}\frac{\rd z}{\ii 2\pi}
-
\int_{\mathscr{C}_2}J_+\frac{\dot f^{(t)}}{f^{(t)}}\frac{\rd z}{\ii 2\pi}
+\notag
\int_{\mathscr{C}_2''}J_+\frac{\dot f^{(t)}}{f^{(t)}}\frac{\rd z}{\ii 2\pi}\\
=&\int_{\Sigma'_{\rm out}}(g'_++g'_-)\frac{\dot f^{(t)}}{f^{(t)}}\frac{\rd z}{\ii 2\pi}
-
\int_{\mathscr{C}_2\cup \mathscr{C}_2''}g'\frac{\dot f^{(t)}}{f^{(t)}}\frac{\rd z}{\ii 2\pi}
+\notag
\int_{\Sigma'_{\rm out}}\left((R_{11}'R_{22}-R_{12}R_{21}')-(R_{11}R_{22}'-R_{12}'R_{21})\right)\frac{\dot f^{(t)}}{f^{(t)}}\frac{\rd z}{\ii 2\pi}\notag \\
&-\int_{\mathscr{C}_2}\left(R_{11}R_{22}'-R_{12}'R_{21}\right)\frac{\dot f^{(t)}}{f^{(t)}}\frac{\rd z}{\ii 2\pi}
+\int_{\mathscr{C}'_2}\left(R_{11}'R_{22}-R_{12}R_{21}'\right)\frac{\dot f^{(t)}}{f^{(t)}}\frac{\rd z}{\ii 2\pi}\ .
\label{eqn:ugly5}
\end{align}
Injecting the estimates (\ref{eqn:ugly1}), (\ref{eqn:ugly2}), (\ref{eqn:ugly3}), (\ref{eqn:ugly4}), (\ref{eqn:ugly5}) into the integrated form of (\ref{eqn:partial}), we obtain
\begin{align*}
&\log {\rm D}_N(f^{(1)})-\log{\rm D}_N(f^{(0)})\\
=&
N\int_0^1\int_{\mathscr{C}}\frac{\dot f^{(t)}}{f^{(t)}}\frac{\rd z}{\ii 2\pi}\rd t+
\int_0^1\left(\int_{\mathscr{C}_1\cup \mathscr{C}_1''\cup \mathscr{C}_2\cup \mathscr{C}_2''}g'\frac{\dot f^{(t)}}{f^{(t)}}\frac{\rd z}{\ii 2\pi}\right)\rd t
+
\int_0^1\left(\int_{\Sigma'_{\rm out}}(g'_++g'_-)\frac{\dot f^{(t)}}{f^{(t)}}\frac{\rd z}{\ii 2\pi}\right)\rd t
+\int_0^1 E(t) \rd t\\
=& N V_0+\int_0^1\left(\int_{\mathscr{C}_{1-\kappa}\cup \mathscr{C}_{1+\kappa}}g'\frac{\dot f^{(t)}}{f^{(t)}}\frac{\rd z}{\ii 2\pi}\right)\rd t
+\int_0^1 E(t) \rd t,
\end{align*}
where the last equation holds by contour deformation, and we denoted $\mathscr{C}_r$ the circle entered at $0$ with radius $r$.
Finally, it was proved in \cite[equations (5.81) to (5.93)]{DeiItsKra2014} that
the above double integral coincides with $\sum_{k=1}^\infty k V_kV_{-k}-\alpha \log (b^0_+(1)b^0_-(1))$. This concludes the proof.
\end{proof}

Let
\begin{equation}\label{eqn:conjugacy}
X(z)=
\left(
\begin{array}{cc}
1&0\\
0&e^{V^{(t)}_0}
\end{array}
\right)
R(z)
\left(
\begin{array}{cc}
1&0\\
0&e^{-V^{(t)}_0}
\end{array}
\right)
\end{equation}
As we will see in the following corollary, this conjugacy establishes symmetry between $|z|<1$ and $|z|>1$. This symmetry was initially broken in (\ref{eqn:Szego}). This small adjustment will be important to us to optimize error terms, as mentionned after Corollary \ref{cor:fundamentalerror}.

The matrix $X$ satisfies the following Riemann-Hilbert problem:
\begin{enumerate}
\item $X$ is analytic for $z\in\mathbb{C}\backslash \Gamma$.
\item The boundary values of $R$ are related by the jump condition
$$
X_+(z)=X_-(z)Q(z)
$$
where the jump matrix $Q$ is given by
\begin{equation*}\displaywidth=\parshapelength\numexpr\prevgraf+2\relax
\begin{array}{ll}
Q(z)=\left(
\begin{array}{cc}
1&0\\
0&e^{V^{(t)}_0}
\end{array}
\right)N(z)\left(\begin{array}{cc}1&0\\f^{(t)}(z)^{-1}z^{-N}&1\end{array}\right)N(z)^{-1}\left(
\begin{array}{cc}
1&0\\
0&e^{-V^{(t)}_0}
\end{array}
\right)&{\rm if\ } z\in\Sigma_{\rm out}\\
Q(z)=\left(
\begin{array}{cc}
1&0\\
0&e^{V^{(t)}_0}
\end{array}
\right)N(z)\left(\begin{array}{cc}1&0\\f^{(t)}(z)^{-1}z^{N}&1\end{array}\right)N(z)^{-1}\left(
\begin{array}{cc}
1&0\\
0&e^{-V^{(t)}_0}
\end{array}
\right)&{\rm if\ } z\in\Sigma''_{\rm out}\\
Q(z)=\left(
\begin{array}{cc}
1&0\\
0&e^{V^{(t)}_0}
\end{array}
\right)M(z)\left(
\begin{array}{cc}
1&0\\
0&e^{-V^{(t)}_0}
\end{array}
\right),&{\rm if\ } z\in \partial {\rm U}
\end{array}
\end{equation*}
\item $X(z)={\rm Id}+{\rm O}(1/z)$ as $z\to\infty$.
\end{enumerate}

\vspace{0.3cm}

Proposition \ref{prop:fundamentalerror} can be written in terms of $X$ as follows.\\

\begin{cor}\label{cor:fundamentalerror}
The identity (\ref{eqn:diffid}) holds
with the error term expressed as
\begin{align*}
E(t)=&-\int_{\mathscr{C}_1}(X_{11}X_{22}'-X_{12}'X_{21})_-\frac{\dot f^{(t)}}{f^{(t)}}\frac{\rd z}{\ii 2\pi }
+\int_{\mathscr{C}_1''}(X_{11}'X_{22}-X_{12}X_{21}')_-\frac{\dot f^{(t)}}{f^{(t)}}\frac{\rd z}{\ii2\pi }\\
&-\int_{\mathscr{C}_2}(X_{11}X_{22}'-X_{12}'X_{21})_+\frac{\dot f^{(t)}}{f^{(t)}}\frac{\rd z}{\ii 2\pi }
+\int_{\mathscr{C}''_2}(X_{11}'X_{22}-X_{12}X_{21}')_+\frac{\dot f^{(t)}}{f^{(t)}}\frac{\rd z}{\ii2\pi }\\
&+\int_{\Sigma_{\rm out}} z^{-N} \frac{e^{-2g+V^{(t)}_0}}{f^{(t)}}\left(X_{22}'X_{12}-X_{12}'X_{22}\right)_+\frac{\dot f^{(t)}}{f^{(t)}}\frac{\rd z}{\ii2\pi }
+\int_{\Sigma''_{\rm out}}z^{N} \frac{e^{2g-V^{(t)}_0}}{f^{(t)}}\left(X_{21}'X_{11}-X_{11}'X_{21}\right)_-\frac{\dot f^{(t)}}{f^{(t)}}\frac{\rd z}{\ii2\pi }\\
&+\int_{\Sigma'_{\rm out}}\left((X_{11}'X_{22}-X_{12}X_{21}')-(X_{11}X_{22}'-X_{12}'X_{21})\right)\frac{\dot f^{(t)}}{f^{(t)}}\frac{\rd z}{\ii 2\pi}\ .
\end{align*}
\end{cor}

One advantage of this writing of $E$ is on $\Sigma_{\rm out}$ and $\Sigma''_{\rm out}$: the terms
$e^{-2g+V^{(t)}_0}/f^{(t)}$ and $e^{2g-V^{(t)}_0}/f^{(t)}$ have smaller order than the corresponding terms in Proposition \ref{prop:fundamentalerror}.
As we will see in the next subsection, the conjugacy (\ref{eqn:conjugacy}) also gives better bounds on the jump matrix $Q-{\rm Id}$ than on $M-{\rm Id}$.

Before performing these estimates, we will use the following rewriting of $|e^{-2g+V^{(t)}_0}/f^{(t)}|$ (when $|z|>1$) and 
$|e^{2g-V^{(t)}_0}/f^{(t)}|$ (when $|z|<1$).

\begin{lem}\label{eqn:rewriting}
We have
\begin{align*}
&\left|\frac{e^{-2g(z)+V^{(t)}_0}}{f^{(t)}(z)}\right|=\left|\frac{b^{(t)}_-(z)}{b^{(t)}_-(\bar z^{-1})}z^{-2\alpha}\right|,\ \mbox{for}\ 1<|z|<1+\kappa,\\
&\left|\frac{e^{2g(z)-V^{(t)}_0}}{f^{(t)}(z)}\right|=\left|\frac{b^{(t)}_-(\bar z^{-1})}{b^{(t)}_-(z)}\right|\ \mbox{for}\ 1-\kappa<|z|<1.
\end{align*}
\end{lem}

\begin{proof}
This is an elementary combination of equations (\ref{eqn:interpol}), (\ref{b+}), (\ref{b-}), (\ref{elem1}), (\ref{elem2}), (\ref{eqn:Szego}).
\end{proof}

\subsection{Asymptotic analysis of the Riemann-Hilbert problem.\ } \label{subsec:anaRH}
The proof of Proposition \ref{prop:tail}
relies on bounding the error estimate in Corollary \ref{cor:fundamentalerror}. For this, we will show that the matrix $X(z)$ is close to the constant ${\rm Id}$, by first proving that the jump matrix $Q(z)$ is approximately {\rm Id}.
Before that, we need to prove that terms appearing in the previous Lemma \ref{eqn:rewriting} are close to 1.

\begin{lem}\label{lem:estlem}
There is a $C>0$ such that for any  $||z|-1|<\kappa$ we have
\begin{align*}
&C^{-1}<\left|\frac{e^{-2g(z)+V^{(t)}_0}}{f^{(t)}(z)}\right|<C\ \mbox{if}\ |z|>1,\\
&C^{-1}<\left|\frac{e^{2g(z)-V^{(t)}_0}}{f^{(t)}(z)}\right|<C\ \mbox{if}\ |z|<1.
\end{align*}
\end{lem}

\begin{proof}
Assume that $|z|>1$. From Lemma \ref{eqn:rewriting}, we need to estimate $\left|\frac{b^{(t)}_-(z)}{b^{(t)}_-(\bar z^{-1})}z^{-2\alpha}\right|$. Clearly, $\log |z^{-2\alpha}|<\alpha N^{-1+\delta}=\oo(1)$.
Moreover, let $\mathscr{C}'$ be the circle centered at $0$ with radius $r=1+\kappa$. Then by analyticity we can write
\begin{equation}\label{eqn:ugly1st}
\log b^{(t)}_-(z)-\log b^{(t)}_-(\bar z^{-1})=\frac{1}{2\pi\ii}\int_{\mathscr{C}'}\left(\frac{V^{(t)}(s)}{s-z}-\frac{V^{(t)}(s)}{s-\bar z^{-1}}\right)\rd s.
\end{equation}
We denote $s= re^{\ii\theta}$. Note that
\begin{equation}\label{eqn:ugly2nd}
\int_{\mathscr{C}'}|V^{(t)}(r e^{\ii\theta})-V^{(t)}(e^{\ii\theta})|\frac{|z-\bar z^{-1}|}{|z-e^{\ii\theta}||\bar z^{-1}-e^{\ii\theta}|}  \rd\theta
\leq C \kappa\sup_{||s|-1|<\kappa}|{V^{(t)}}'(s)|.
\end{equation}

Moreover, 
using (\ref{eqn:Vt}) and (\ref{eqn:ImV0}), we have 
\begin{equation}\label{eqn:ugly3rd}
\sup_{||s|-1|<\kappa}|{V^{(t)}}'(s)|\leq C  N^{1-\delta}.
\end{equation}
Using (\ref{eqn:ugly1st}), (\ref{eqn:ugly2nd}), and (\ref{eqn:ugly3rd}), we proved
$$
\log b^{(t)}_-(z)-\log b^{(t)}_-(\bar z^{-1})
=\frac{1}{2\pi\ii}\int_{\mathscr{C}'}\left(\frac{V^{(t)}(s/r)}{s-z}-\frac{V^{(t)}(s/r)}{s-\bar z^{-1}}\right)\rd s
+\OO(1).
$$
Define $z_0=z/|z|$. Then from the above equation we also have
$$
\log b^{(t)}_-(z)-\log b^{(t)}_-(\bar z^{-1})
=\frac{1}{2\pi\ii}\int_{\mathscr{C}'}\left(\frac{V^{(t)}(s/r)-V^{(t)}(z_0)}{s-z}-\frac{V^{(t)}(s/r)-V^{(t)}(z_0)}{s-\bar z^{-1}}\right)\rd s
+\OO(1),
$$
so that the result will be proved if
\begin{equation}\label{eqn:enough}
\int_{\mathscr{C}'}|V^{(t)}(s/r)-V^{(t)}(z_0)|\frac{|z-\bar z^{-1}|}{|s-z|^2}\rd s
=\OO(1).
\end{equation}
Note that from (\ref{eqn:Vt}) we have
$$
e^{V^{(t)}(s/r)-V^{(t)}(z_0)}=\frac{1-t+te^{V(s/r)}}{1-t+te^{V(z_0)}},
$$
so that
$$
e^{V^{(t)}(s/r)-V^{(t)}(z_0)}\leq C+C e^{V(s/r)-V(z_0)}.
$$
Here we used that $(a+b)/(c+d)<\max(a/c,b/d)$ for positive numbers, and that on the unit circle 
$V$ has real values. This implies that
$$
|V^{(t)}(s/r)-V^{(t)}(z_0)|<1+|V(s/r)-V(z_0)|.
$$
Moreover, 
Abel summation gives
\begin{equation}\label{eqn:Abel}
V_0(z)=-\lambda\log\max(|1-z|,\kappa)+\OO(1)\ \ \text{uniformly in $||z|-1|<\kappa$.}
\end{equation}
Both previous equations give (let $z_0=e^{\ii\theta_0}$)
$$
\int_{\mathscr{C}'}|V^{(t)}(s/r)-V^{(t)}(z_0)|\frac{|z-\bar z^{-1}|}{|s-z|^2}\rd s\\
\leq
C \int_{\mathscr{C}'}(1+|\log\max(\theta,\kappa)-\log\max(\theta_0,\kappa) |)\frac{\kappa}{|e^{\ii\theta}-z|^2}\rd \theta+\OO(1)=\OO(1).
$$
This proves (\ref{eqn:enough}) and therefore the lemma for $|z|>1$.
The proof for $|z|<1$ is similar.
\end{proof}

\begin{lem}\label{lem:jumpestimate}
Entries of the jump matrix satisfy the bounds
$$
\begin{array}{ll}
Q(z)-{\rm Id}={\rm O}\left(N^{-\delta}\right)&{\rm uniformly\ in\ } z\in\partial {\rm U},\\
Q(z)-{\rm Id}={\rm O}\left(e^{-cN^{\delta}}\right)&{\rm uniformly\ in\ } z\in\Sigma_{\rm out}\cup \Sigma''_{\rm out}.
\end{array}
$$
\end{lem}

\begin{proof}
We first consider the jump matrix on $\partial U$. It was proved in \cite[Proposition 3]{MarMcLSaf2006} that
the jump matrix of the $R$-Riemann-Hilbert problem satisfies the asymptotic expansion (notice that we changed the notations from \cite{MarMcLSaf2006} to our setting), in the case $\delta=1$, i.e. when $V$ and the contour do not depend on $N$: if $|z|>1$,
$$
M(z)={\rm Id}+\sum_{k=1}^\infty\frac{\ii ^k}{2^{k+1}\zeta^k}
\left(
\begin{array}{cc}
s_{\alpha,k}&(-1)^k \left(\frac{e^{-2g(z)}}{f^{(t)}(z)}\right)^{-1} t_{\alpha,k}\\
\frac{e^{-2g(z)}}{f^{(t)}(z)}t_{\alpha,k}&(-1)^ks_{\alpha,k}
\end{array}
\right)
$$
meaning that the remainder associated with partial sums does not exceed the first neglected term in absolute value. Here, the coefficients are
$
s_{\beta,k}=(\alpha+\frac{1}{2},k)+(\alpha-\frac{1}{2},k)$, 
$t_{\beta,k}=(\alpha+\frac{1}{2},k)-(\alpha-\frac{1}{2},k)$
where
$
(\nu,k)=\frac{(4\nu^2-1)(4\nu^2-9)\dots(4\nu^2-(2k-1)^2)}{2^{2k}k!}$. 
The above expansion is equivalent (by simple conjugacy) to
\begin{equation}\label{eqn:Q}
Q(z)={\rm Id}+\sum_{k=1}^\infty\frac{\ii ^k}{2^{k+1}\zeta^k}
\left(
\begin{array}{cc}
s_{\alpha,k}&(-1)^k \left(\frac{e^{-2g+V^{(t)}_0}}{f^{(t)}}\right)^{-1} t_{\alpha,k}\\
\frac{e^{-2g+V^{(t)}_0}}{f^{(t)}}t_{\alpha,k}&(-1)^ks_{\alpha,k}
\end{array}
\right).
\end{equation}
If we assume  $|z-1|\asymp \kappa$ (hence $|\zeta|\asymp N^{\delta}$), and that $V$ depends on $N$ (through (\ref{eqn:Uext}) and (\ref{eqn:Vt})), this expansion still holds: the proof in
 \cite{MarMcLSaf2006} only requires $(1)$ known asymptotics of confluent hyperbolic fuctions (these still hold in our context as $\alpha$ is a fixed parameter independent of $N$, like in  \cite{MarMcLSaf2006}) and (2) that the above coefficient $\frac{e^{-2g+V^{(t)}_0}}{f^{(t)}}$ is of order one (this property holds as proved in Lemma \ref{lem:estlem}).

Thus (\ref{eqn:Q}) holds in our regime of interest and at first order, this approximation is
$
Q(z)={\rm Id}+{\rm O}\left(N^{-\delta}\right),
$
as expected. This approximation still holds when $z\in\partial{\rm U}$ but $|z|<1$, by just changing the coefficient 
$e^{-2g+V^{(t)}_0}/f^{(t)}$ by
$e^{2g-V^{(t)}_0}/f^{(t)}$ in the reasoning.
 
On $\Sigma_{\rm out}$, the result follows from
\begin{equation}\label{eqn:QbounOut}
Q(z)-{\rm Id}=
\left(
\begin{array}{cc}
0&0\\
\frac{e^{-2g+V^{(t)}_0}}{f^{(t)}}z^{-N}&0
\end{array}
\right)
=\OO(z^{-N})=\OO(e^{-c N^\delta}).
\end{equation}
A similar calculation gives the estimate on $\Sigma''_{\rm out}$.
\end{proof}

\begin{prop}\label{prop:Qestimate} 
The matrix $X$ satisfies the bounds
\begin{equation}\label{eqn:1stX}
\begin{array}{ll}
X(z)_{\pm}-{\rm Id}={\rm O}\left(\frac{1}{N|z-1|}\right)&{\rm uniformly\ in\ } z\in\Gamma.
\end{array}
\end{equation}
The matrix $X'=\frac{\partial}{\partial z}X$ satisfies the bounds
\begin{equation}\label{eqn:2ndX}
\begin{array}{ll}
X'(z)_{\pm}={\rm O}\left(\frac{N^{-\delta}}{|z-1|}\right)&{\rm uniformly\ in\ } z\in\Gamma.
\end{array}
\end{equation}
The same bounds hold for $X(z)$ and $X'(z)$ away from $\Gamma$ uniformly in $||z|-1|<\kappa$.
\end{prop}

\begin{proof}
Consider the Cauchy operator
$$
{\rm C} f(z)=\frac{1}{2\pi\ii}\int_{\Gamma}\frac{f(\xi)}{\xi-z}\rd \xi,\ z\in\mathbb{C}\backslash \Gamma,
$$
where the orientation of inegration is clockwise on $\Sigma_{\rm out}$ and $\Sigma_{\rm out}''$ and clockwise on 
{\rm U}. For $\xi\in \Gamma$, let ${\rm C}_ f(\xi)=\lim_{z\to\xi^-} {\rm C} f(z)$. It is well known (see e.g. \cite{Ste1970}) that this limit exists in ${\rm L}^2(\Gamma)$ and 
\begin{equation}\label{eqn:boundedinL2}
\|C_-f\|_{{\rm L}^2(\Gamma)}\leq c_\Gamma \|f\|_{{\rm L}^2(\Gamma)}.
\end{equation}
For our contour, the constant $c_\Gamma$ can actually be chosen uniform in $N$ as shown by a simple scaling argument (the radii of our circles do not matter).

Let $\Delta_X(z)=Q(z)-{\rm Id}$. From Lemma $\ref{lem:jumpestimate}$,
$\|\Delta_X\|_{L^\infty(\Gamma)}=\OO(N^{-\delta}+e^{-cN^\delta})$ converges to $0$. Together with 
(\ref{eqn:boundedinL2}) it implies that
$$
{\rm C}_{\Delta_X}:f\mapsto C_-(f \Delta_X)
$$
is a bounded operator from ${\rm L}^2(\Gamma)$ to itself with norm $\|{\rm C}_{\Delta_X}\|\to 0$. 
Hence 
$1-{\rm C}_{\Delta_X}$ has an inverse for large enough $N$, and we have
\begin{equation}\label{eqn:usefulineq}
\|(1-{\rm C}_{\Delta_X})^{-1}f\|_{{\rm L}^2(\Gamma)}\leq C 
\|f\|_{{\rm L}^2(\Gamma)},
\end{equation}
still for large enough $N$.
Let
$$
\mu_X=(1-{\rm C}_{\Delta_X})^{-1}({\rm C}_- \Delta_X)\in {\rm L}^2(\Gamma).
$$
It is well known (see e.g. \cite[Theorem 7.8]{DeiKriMcLVenZho1999}) that for any $z\not\in\Gamma$
\begin{equation}\label{eqn:X}
X(z)={\rm Id}+{\rm C}(\Delta_X+\mu_X\Delta_X)(z).
\end{equation}
We first bounds similar to (\ref{eqn:1stX}) and (\ref{eqn:2ndX}) under a stronger assumption on $z$, namely ${\rm dist}(z,\Gamma)= N^{-1+\delta}/100$. To bound the first term in (\ref{eqn:X}), note that
\begin{equation}\label{eqn:113}
|{\rm C}\Delta_X(z)|\lesssim
\int_{\partial {\rm U}}\frac{|\Delta_X(\xi)|}{|z-\xi|} |\rd\xi|+\int_{\Sigma_{\rm out}\cup \Sigma_{\rm out}''}\frac{|\Delta_X(\xi)|}{|z-\xi|} |\rd\xi|
\lesssim
\frac{1}{N|z-1|}
+
e^{-cN^\delta}.
\end{equation}
where we used Lemma \ref{lem:jumpestimate} to bound the jump matrix $\Delta_X$. Concerning the other term from 
(\ref{eqn:X}),  the Schwarz inequality yields
\begin{equation}\label{eqn:mu1}
|{\rm C}(\mu_X\Delta_X)(z)|
\lesssim
\|\mu_X\|_{{\rm L}^2(\partial {\rm U})}
\left(\int_{\partial {\rm U}}\frac{|\Delta_X(\xi)|^2}{|z-\xi|^2}|\rd \xi|\right)^{1/2}
+
\|\mu_X\|_{{\rm L}^2(\Sigma_{\rm out}\cup\Sigma_{\rm out}'')}
\left(\int_{\Sigma_{\rm out}\cup\Sigma_{\rm out}''}\frac{|\Delta_X(\xi)|^2}{|z-\xi|^2}|\rd \xi|\right)^{1/2}.
\end{equation}
From (\ref{eqn:usefulineq}), we have
\begin{multline}\label{eqn:mu2}
\|\mu_X\|_{{\rm L}^2(\partial {\rm U})}\lesssim \|C_-\Delta_X \|_{{\rm L}^2(\partial {\rm U})}
\lesssim
\|C_-(\Delta_X\1_{\partial{\rm U}}) \|_{{\rm L}^2(\partial {\rm U})}
+
\|C_-(\Delta_X\1_{\Sigma_{\rm out}\cup\Sigma_{\rm out}''}) \|_{{\rm L}^2(\partial {\rm U})}\\
\lesssim
\|\Delta_X \|_{{\rm L}^2(\partial {\rm U})}
+
\|\Delta_X\|_{{\rm L}^\infty(\Sigma_{\rm out}\cup\Sigma_{\rm out}'')}
\left(\int_{{\partial {\rm U}}} |\log {\rm dist}(\xi,\Sigma_{\rm out}\cup\Sigma_{\rm out}'')|^2|\rd \xi|\right)^{1/2}\\
\lesssim N^{-\delta}N^{-(1-\delta)/2}+e^{-cN^\delta}
\end{multline}
where we used (\ref{eqn:boundedinL2}) for the third inequality.
In the same manner we have
\begin{equation}\label{eqn:mu3}
\|\mu_X\|_{{\rm L}^2(\Sigma_{\rm out}\cup\Sigma_{\rm out}'')}
\lesssim
e^{-N^\delta}+N^{-\delta}N^{-(1-\delta)/2}\log N.
\end{equation}
Equations (\ref{eqn:mu1}), (\ref{eqn:mu2}) and (\ref{eqn:mu3}) imply that
\begin{equation}\label{eqn:110}
|{\rm C}(\mu_X\Delta_X)(z)|\lesssim \frac{N^{-\delta}}{N|z-1|}+e^{-c N^{\delta}}.
\end{equation}
Equations  (\ref{eqn:113}) and (\ref{eqn:110})  conclude the proof of (\ref{eqn:1stX}) when $z$ is far enough from $\Gamma$.
As ${\rm dist}(z,\Gamma)=N^{-1+\delta}/100$, the estimate (\ref{eqn:2ndX}) follows from (\ref{eqn:1stX}) by 
Cauchy's integral formula.

Remarkably, these estimates also hold up to $z\in\Gamma_\pm$ thanks to the classical contour deformation argument 
as explained in the proof of \cite[Corollary 7.77]{DeiKriMcLVenZho1999}. This concludes the proof.
\end{proof}

\begin{cor}\label{cor:boundE}
We have
$$
\int_0^1 E(t)\rd t={\rm O}\left(N^{-\delta}(\log N)^3\right).
$$
\end{cor}

\begin{proof}
By Lemma \ref{lem:estlem}, $|e^{-2g+V^{(t)}_0}/f^{(t)}|=|e^{2g-V^{(t)}_0}/f^{(t)}|=\OO(1)$ in Corollary \ref{cor:fundamentalerror}. Moreover, $\dot f^{(t)}/f^{(t)}$ has a constant sign for fixed $z$ and $\int_0^1 \dot f^{(t)}/f^{(t)}\rd t=V(z)$. These observations injected in Corollary \ref{cor:fundamentalerror} give
\begin{multline*}
\left|\int_0^1E(t)\rd t\right|\lesssim
\sup_{0\leq t\leq 1}\left(\|X\|_{{\rm L}^\infty(\partial{\rm U})} \|X'\|_{{\rm L}^\infty(\partial{\rm U})}N^{-1+\delta}
+
\int_{\Sigma_{\rm out}\cup\Sigma'_{\rm out}\cup\Sigma''_{\rm out}}
|X(z)||X'(z)||\rd z|\right)\sup_{z\in\Gamma}\int_0^1\left|\frac{\dot f^{(t)}}{f^{(t)}}\right|\\
\lesssim
N^{-\delta}(\log N) \sup_{z\in\Gamma}\int_0^1\left|\frac{\dot f^{(t)}}{f^{(t)}}\right|,
\end{multline*}
where we used Proposition \ref{prop:Qestimate}. Moreover, from (\ref{eqn:ugly3rd}) and (\ref{eqn:Abel})
\begin{equation}\label{eqn:Abel2}
V(z)=-\lambda\log|1-z|+\OO(1)
\end{equation}
uniformly in $\Gamma$, so that
if $t\leq e^{-(\log N)^2}$ we have $|\dot f^{(t)}/f^{(t)}|\leq C N^{\lambda}$. If $t>e^{-(\log N)^2}$ we use (\ref{eqn:ImV0})
to conclude that
$$
\left|\frac{\dot f^{(t)}}{f^{(t)}}\right|\leq C\left|\frac{1+e^{V}}{t e^{V}}\right|\leq \frac{C}{t}.
$$
All together, we obtain that for any $||z|-1|<\epsilon N^{-1+\delta}$ we have
$$
\int_0^1\left|\frac{\dot f^{(t)}}{f^{(t)}}(z)\right|\rd t\leq \int_0^{e^{-(\log N)^2}}N^\delta+C\int_{e^{-(\log N)^2}}^1\frac{\rd t}{t}\leq C (\log N)^2,
$$
which concludes the proof.
\end{proof}

\subsection{Proof of Proposition \ref{prop:tail}.\ }\label{subsec:proof}
We choose $\lambda=\alpha$. Note that 
\begin{align*}
V_0&=0,\\
\sum_{k\geq 1}k V_k V_{-k}&=\alpha^2(1-\delta)\log N+{\rm O}(1),\\
\log(b^1_+(1)b^1_-(1))&=2\alpha(1-\delta)\log N+{\rm O}(1).
\end{align*}
Therefore Corollary \ref{cor:fundamentalerror} and Corollary \ref{cor:boundE} yield
$$
\log D_N(f^{(1)})-\log D_N(f^{(0)})
=
-\alpha^2(1-\delta)\log N +{\rm O}(1).
$$
Together with the following Lemma \ref{lem:Selberg}, this gives
$$
\log D_N(1)=\alpha^2\delta\log N +\OO(1),
$$
which concludes the proof.

\begin{lem}\label{lem:Selberg}
The Toeplitz determinant assocated to the pure Fisher-Hartwig symbol satisfies the estimates
$$
\log D_N(f^{(0)})=\alpha^2\log N +{\rm O}(1).
$$
\end{lem}

\begin{proof}
This is an elementary consequence of Selberg's integral. 
The following exact formula holds \cite{KeaSna2000}:
$$
{\rm D}_N(f^{(0)})
=
\prod_{k=1}^N\frac{\Gamma(k)\Gamma(k+2\alpha)}{\Gamma(k+\alpha)^2}.
$$
Let $G$ be the Barnes function, which satisfies in particular $G(z+1)=\Gamma(z)G(z)$, $G(1)=1$.
Then
\begin{equation}\label{eqn:asympt1}
\log D_N(f^{(0)})=\log G(1+N+2\alpha)+\log G(1+N)-2\log G(1+N+\alpha)-\log G(1+2\alpha)+2\log G(1+\alpha).
\end{equation}
The Barnes function satisfies the asymptotic expansion \cite{Vor1987}
\begin{equation}\label{eqn:asympt2}
\log G(1+z)=\frac{z^2}{2}\log z-\frac{3}{4}z^2+\frac{1}{2}\log(2\pi) z-\frac{1}{12}\log z+\zeta'(-1)+{\rm O}\left(\frac{1}{z^2}\right).
\end{equation}
From (\ref{eqn:asympt1}) and (\ref{eqn:asympt2}) we get
the result of the lemma.
\end{proof}
\noindent Note that Lemma \ref{lem:Selberg} also proves Lemma \ref{lem: Full sum exp mom}.

\subsection{Gaussian approximation of the increments.\ }\label{subsec:Gauss}
We now prove Lemmas \ref{lem: gaussian UB}, \ref{lem: two point fourier bound} and \ref{lem: exp moment continuity}, which will all follow from  the following Proposition \ref{thm:FourierRH}.

Let $m\in\mathbb{N}$ be fixed,  let $\bh=(h_1,\dots,h_m)$ (for any $1\leq k\leq m$, $h_1,\dots,h_m\in[0,2\pi]$), and let  $\ba=(\alpha_1,\dots,\alpha_m)$, 
$\bb=(\beta_1,\dots,\beta_m)$ (for any $1\leq k\leq m$, $0<\alpha_k<\beta_k<1$) have real entries. Let $\bx=(\xi_1,\dots,\xi_m)$ have possibly complex entries.
In this subsection, we change the definition of $V$ and are interested in functions of type
\begin{equation}\label{eqn:fRH}
V(z)=\frac{1}{2}
\sum_{k=1}^m \xi_k \sum_{j=N^{\alpha_k}}^{N^{\beta_k}}\frac{z^j e^{-\ii j h_k}+z^{-j} e^{\ii j h_k}}{j},
\end{equation}
so that $\sum_{ \ell=1}^N V(e^{\ii\theta_\ell}) = \sum_{k=1}^m \xi_k \sum_{j=N^{\alpha_k}}^{N^{\beta_k}}\Re\left(\frac{e^{-\ii j h_k} \Tr(U_N^j)}{j}\right)$.
We will also consider in linear combinations of such functions.

For a general smooth function on the unit circle we define
\begin{equation}\label{eqn:f normRH}
\sigma^2(V)=\sum_{j=1}^\infty j V_jV_{-j}.
\end{equation}
We have the following estimates, for general $V$, not necessarily of type (\ref{eqn:fRH}).

\begin{prop}[Fourier transform asymptotics]\label{thm:FourierRH}
Let $0<\delta<1$ be fixed. Assume that   $V$ is analytic in $||z|-1|<N^{-1+\delta}$ and $|V(z)|<N^{\delta-\e}$ in this domain, for some fixed $\e>0$. Then
\begin{equation}\label{eqn:essgaus}
\E\left(e^{\sum_{\ell=1}^NV(e^{\ii\theta_\ell})}\right)=e^{\sigma^2(V)}\left(1+\OO(e^{-cN^{\delta}})\right).
\end{equation}
\end{prop}

\begin{proof}
We follow the method of Subsections \ref{subsec:Set}, \ref{subsec:RH}, \ref{subsec:anaRH} and \ref{subsec:proof} with a notable difference: there is no singularity at $1$,
that is $\alpha=0$. 
The contour $\Gamma$ of the  $X$-Riemann-Hilbert problem is just 
$\Sigma_{\rm out}\cup\Sigma''_{\rm out}$ (their closed version, i.e. two full circles).

We cannot apply the previous method directly. Our new choice for $V$ has much greater amplitude than (\ref{eqn:Uext}), so if we adopt the interpolation (\ref{eqn:Vt}), there is no not guarantee that
$1-t+t e^{V(z)}\neq 0$ for all $0\leq t\leq 1$ and $||z|-1|<N^{-1+\delta}$. Lemma \ref{lem:analytic} does not hold anymore and we cannot consider any Riemann-Hilbert problem.

To circumvent this problem, we adopt a different interpolation in many steps, following an idea from 
\cite[Section 5.4]{DeiItsKra2014}. Define $N$ functions $(V^{(k)})_{1\leq k\leq N}$ (any number on a polynomial scale greater than $N^\delta$ would actually work) simply defined as $V^{(k)}(z)=\frac{k}{N}V(z)$. 
For fixed $k$, we consider the interpolation
$$
e^{V^{(k,t)}(z)}=(1-t)e^{V^{(k-1)}(z)}+t e^{V^{(k)}(z)}.
$$
Then the analogue of Lemma \ref{lem:analytic} holds: for any $k$ and $t$, $V^{(k,t)}$ admits an analytic continuation to
$||z|-1|<\kappa$, because, on that domain, we have $|{\rm Im} V^{(k-1)}(z)-{\rm Im} V^{(k)}(z)|\leq \frac{1}{N}\sup_{||z|-1|<\kappa}{|V(z)|}\leq N^{-1+\delta}<1/10$.
We therefore can consider the following Riemann-Hilbert problem for the matrix $X$ (we do not mention the dependence of $X$ in $t$ and $k$):
\begin{enumerate}
\item $X$ is analytic for $z\in\mathbb{C}\backslash \Gamma$.
\item The boundary values of $R$ are related by the jump condition
$$
X_+(z)=X_-(z)Q(z)
$$
where the jump matrix $Q$ is given by
\begin{equation*}
\begin{array}{ll}
Q(z)=\left(
\begin{array}{cc}
1&0\\
0&e^{V^{(k,t)}_0}
\end{array}
\right)N(z)\left(\begin{array}{cc}1&0\\f^{(k,t)}(z)^{-1}z^{-N}&1\end{array}\right)N(z)^{-1}\left(
\begin{array}{cc}
1&0\\
0&e^{-V^{(k,t)}_0}
\end{array}
\right)&{\rm if\ } z\in\Sigma_{\rm out},\\
Q(z)=\left(
\begin{array}{cc}
1&0\\
0&e^{V^{(k,t)}_0}
\end{array}
\right)N(z)\left(\begin{array}{cc}1&0\\f^{(k,t)}(z)^{-1}z^{N}&1\end{array}\right)N(z)^{-1}\left(
\begin{array}{cc}
1&0\\
0&e^{-V^{(k,t)}_0}
\end{array}
\right)&{\rm if\ } z\in\Sigma''_{\rm out},
\end{array}
\end{equation*}
where we used the notations
\begin{align*}
f^{(k,t)}(z)&=e^{V^{(k,t)}(z)}\\
N(z)&=\left\{
\begin{array}{ll}
e^{g(z)\sigma_3}&{\rm if}\ |z|>1\\
e^{g(z)\sigma_3}\left(
\begin{array}{cc}
0&1\\
-1&0
\end{array}
\right)&{\rm if}\ |z|<1
\end{array}\right.\\
\exp(g(z))&=\exp\frac{1}{2\pi \ii}\int_{\mathscr{C}}\frac{\ln V^{(k,t)}(s)}{s-z}\rd s
=
\left\{
\begin{array}{ll}
e^{V^{(k,t)}_0}b^{(k,t)}_+(z)&{\rm if}\ |z|<1,\\
b^{(k,t)}_-(z)^{-1}&{\rm if}\ |z|>1.\\
\end{array}
\right.
\end{align*}
\item $X(z)={\rm Id}+{\rm O}(1/z)$ as $z\to\infty$.
\end{enumerate}
Define
\begin{equation}\label{eqn:DNk}
{\rm D}_N(V^{(k})=\E\left(\prod_{j=1}^N e^{V^{(k)}(e^{\ii\theta_j})}\right)\ .
\end{equation}
From \cite[Equation (5.106)]{DeiItsKra2014}, for any $1\leq k\leq N$ we have
\begin{equation}\label{eqn:errorK}
\log {\rm D}_N(f^{(k)})-\log{\rm D}_N(f^{(k-1)})
=
V_0+\frac{2k-1}{N^2}\sum_{k=1}^\infty k V_kV_{-k}
+\int_0^1 E(t) \rd t.
\end{equation}
where
\begin{align}
E(t)=&\int_{\Sigma_{\rm out}} z^{-N} \frac{e^{-2g+V_0^{(k,t)}}}{f^{(k,t)}}\left(X_{22}'X_{12}-X_{12}'X_{22}\right)_+\frac{\dot f^{(k,t)}}{f^{(k,t)}}\frac{\rd z}{\ii2\pi }\notag
+\int_{\Sigma''_{\rm out}}z^{N} \frac{e^{2g-V_0^{(k,t)}}}{f^{(k,t)}}\left(X_{21}'X_{11}-X_{11}'X_{21}\right)_-\frac{\dot f^{(k,t)}}{f^{(k,t)}}\frac{\rd z}{\ii2\pi }\\
&+\int_{\Sigma'_{\rm out}}\left((X_{11}'X_{22}-X_{12}X_{21}')-(X_{11}X_{22}'-X_{12}'X_{21})\right)\frac{\dot f^{(k,t)}}{f^{(k,t)}}\frac{\rd z}{\ii 2\pi}\label{eqn:errorK2}.
\end{align}
The formula (\ref{eqn:errorK2}) comes from \cite[Section 5.4]{DeiItsKra2014}, in the same way as we followed 
\cite[Section 5.3]{DeiItsKra2014} for the proof of Proposition \ref{prop:fundamentalerror} and Corollary \ref{cor:fundamentalerror}. 

Note that $\sum_{k=1}^N(2k-1)/N^2=1$, so that Proposition \ref{thm:FourierRH}  will be proved by summation of (\ref{eqn:errorK}) if, uniformly in $k$ and $t$, we have
\begin{equation}\label{eqn:erorsmall}
|E(t)|\leq e^{-cN^{\delta}}.
\end{equation}
We now bound all terms in (\ref{eqn:fRH}), uniformly in $k$ and $t$.
\begin{enumerate}[(a)]
\item 
Reproducing the reasoning in Lemma \ref{lem:estlem}, we have, on $\Sigma_{\rm out}$, 
$$
\log \left|\frac{e^{-2g+V_0^{(k,t)}}}{f^{(t)}}(z)\right|
\leq
C\int_{\mathscr{C}_{1+2\kappa}}\left(\frac{V^{(k,t)}(s)}{s-z}-\frac{V^{(k,t)}(s)}{s-\bar z^{-1}}\right)\rd s\leq C\|V^{(k,t)}\|_{{\rm L}^{\infty}(\Sigma_{\rm out})}\leq C N^{\delta-\e}.
$$
\item Clearly $|\dot f^{(k,t)}/f^{(k,t)}|<\frac{1}{N}\sup_{||z||-1|<\kappa} |V(z)|\leq 1$.
\item The analogue of Lemma \ref{lem:jumpestimate} now just involves $\Sigma_{\rm out}$ and $\Sigma_{\rm out}''$. By equation (\ref{eqn:QbounOut}), using (a) we now have 
$Q(z)-{\rm Id}=\OO(e^{C N^{\delta-\e}}e^{-c N^\delta})\leq e^{-cN^{\delta}}$. Following the reasoning of Proposition \ref{prop:Qestimate}, this implies that
$$
X(z)_{\pm}-{\rm Id}={\rm O}\left(e^{-cN^{\delta}}\right),X'(z)_{\pm}={\rm O}\left(e^{-cN^{\delta}}\right)
$$
uniformly  in  $z\in\Sigma_{{\rm out}}\cup\Sigma''_{{\rm out}}$,
and the same bounds hold for $X(z)$ and $X'(z)$  uniformly in $||z|-1|<\kappa$.
\end{enumerate}
The estimates (a), (b) and (c) in (\ref{eqn:errorK}) imply the result (\ref{eqn:erorsmall}) and therefore conclude the proof of the proposition.
\end{proof}

The Fourier decomposition of $V$ is 
$$
V(e^{\ii\theta})=\sum_{j\in\mathbb{Z}} V_j e^{\ii j\theta},\ 
V_j = \frac{1}{2} \sum_{ k=1}^m \xi_k \frac{e^{-\ii jh_k}}{|j|}\1_{N^{\alpha_k}\leq |j|\leq N^{\beta_k}}.
$$

\begin{proof}[Proof of Lemma \ref{lem: gaussian UB}]
We apply Proposition \ref{thm:FourierRH}
to
$$
V(z)=\frac{1}{2} \xi\sum_{1 \le j < N^{1-\delta}} -\frac{e^{-\ii jh}z^j+e^{\ii jh}z^{-j}}{j}
$$
We have $|V(z)|<C|\xi|\log N<N^{\delta/2}$ on $||z|-1|<N^{-1+\delta}$. Therefore (\ref{eqn:essgaus}) holds and the proof is concluded by noting that
$2\sigma^2(V)=\xi^2\sigma^2$, where we used both notations (\ref{eqn:f normRH}) and (\ref{eq: exp mom of trunc sum}).
\end{proof}

\begin{proof}[Proof of Lemma \ref{lem: two point fourier bound}]
The proof is identical to the above proof of Lemma \ref{lem: gaussian UB}.
\end{proof}

\begin{proof}[Proof of Lemma \ref{lem: exp moment continuity}]
We apply Proposition \ref{thm:FourierRH}
to
$$
V(z)=\frac{1}{2}\xi \left( \sum_{j=1}^{N^{1-\delta}} -\frac{ ( e^{-\ii jh} z^j+e^{\ii jh} z^{-j} ) }{j} 
        - \sum_{j=1}^{N^{1-\delta}} -\frac{(z^j+z^{-j} )}{j} \right)
$$
As $|\xi h N^{1-\delta}| \le N^{\delta-\e}$, we have $|V(z)|<N^{\delta-\frac{\e}{2}}$ for $||z|-1|<N^{-1+\delta}$. Therefore (\ref{eqn:essgaus}) holds and the proof is concluded by noting that
$\sigma^2(V)/2=2\xi^2(\sigma^2-\rho)$, where we used both notations (\ref{eqn:f normRH}) and (\ref{eqn: Sigma 2}).
\end{proof}

\subsection{Tail estimate for the imaginary part.\ } \label{subsec:Im}
The following is the strict analogue of Proposition \ref{prop:tail}, for the imaginary part of the series.

\begin{prop}
\label{prop:tailIm}
For any fixed (small) $\delta \in (0,1)$ and (large) $\alpha$, for large enough $N$ we have
\begin{equation}\label{eq: mgf ub tail}
  \E\left( \exp\left(2\alpha \sum_{j \ge N^{1-\delta} } \frac{{\rm Im}(\Tr {\rm U}_N^j)}{j} \right) \right)
 = \exp\left( (1+\oo(1))\alpha^2 \delta \log N \right).
\end{equation}
\end{prop}

\begin{proof}
We follow the ideas from subsections 
\ref{subsec:Set} to \ref{subsec:proof}. The previous definition of $V$ in (\ref{eqn:Uext}) is now replaced with
\begin{equation}\label{eqn:UextIm}
V(z)=\lambda\sum_{j=1}^{N^{1-\delta}}\frac{z^j-z^{-j}}{\ii j},\ {\rm i.e.}\ V_j=\frac{\lambda}{\ii j}\1_{|j|\leq N^{1-\delta}},
\end{equation}
where we will choose $\lambda=-\alpha$.
Denoting $z=e^{\ii \theta}$ with $0\leq \theta<2\pi$, we now have
$$
f^{(t)}(z)=e^{V^{(t)}(z)}e^{-\alpha(\theta-\pi)}.
$$
With the notations in \cite[equation (1.2)]{DeiItsKra2011}, this coincides with the choice $\beta_0=\ii\alpha$,  $\beta_j=0$ for $j\geq 1$ and $\alpha_j\equiv0$. Proposition \ref{prop:fundamentalerror} is now changed into
\begin{equation}\label{eqn:fundamentalerrorIm}
\log {\rm D}_N(f^{(1)})-\log{\rm D}_N(f^{(0)})
=
N V_0+\sum_{k=1}^\infty k V_kV_{-k}+\ii\alpha \log (b^1_+(1)/b^1_-(1))
+\int_0^1 \widetilde E(t) \rd t,
\end{equation}
as shown in \cite[equation (5.93)]{DeiItsKra2014}.
The new error term $\widetilde E(t)$ will be defined and shown to be negligible at the end of this proof. We first estimate all other terms in (\ref{eqn:fundamentalerrorIm}). Clearly, $V_0=0$, $\sum_{k\geq 1}kV_kV_{-k}=\alpha^2(1-\delta)\log N+\OO(1)$, and from  (\ref{eqn:b+b-}) we have 
$\log (b^1_+(1)/b^1_-(1))=\ii 2\alpha(1-\delta)\log N+\OO(1)$. Moreover, from Selberg's integrals it is well known that
(see \cite{KeaSna2000})
$$
{\rm D}_N(f^{(0)})
=
\prod_{k=1}^N\frac{\Gamma(k)^2}{\Gamma(k+\ii\alpha)\Gamma(k-\ii\alpha)}.
$$
With the asymptotic expansion (\ref{eqn:asympt2}), this yields $\log {\rm D}_N(f^{(0)})=\alpha^2\log N+\OO(1)$.
To summarize, we just proved that 
$$
\log {\rm D}_N(f^{(1)})=\alpha^2(1-\delta)\log N+\OO(1)+\int_0^1 \widetilde E(t) \rd t,
$$
so that the proof of Proposition \ref{prop:tailIm} will be complete after bounding $\widetilde E(t)$.

This error term $\widetilde E(t)$ is defined exactly as in Corollary \ref{cor:fundamentalerror} after
replacing $X$ with $\widetilde{X}$, 
the latter being 
the solution of the same Riemann-Hilbert problem as $X$ (see before Corollary \ref{cor:fundamentalerror}) except that in the definition of
the jump matrix $M$ (see (\ref{eqn:matrixM})) one needs to be replace
$\exp(g)$ from (\ref{eqn:Szego}),  $F$ from (\ref{eqn:F}), $E$ from (\ref{eqn:E}) and $\Psi$ from (\ref{eqn:Psi}) as follows:
we now choose $m=0$, $z_0=1$, $\alpha_0=0$ and $\beta_0=\ii\alpha$ in 
the 
formulas 
\cite[equations (4.9) and (4.10)]{DeiItsKra2011} 
for $\mathcal{D}=\exp(g)$, 
\cite[equation (4.17)]{DeiItsKra2011} for $F$, 
\cite[equations (4.47) till (4.50)]{DeiItsKra2011} 
for $E$, and 
\cite[Proposition 4.1 and equations (4.25) till (4.29)]{DeiItsKra2011} 
for $\Psi$. 

For the new definition of $\exp(g)$, Lemmas \ref{eqn:rewriting} and \ref{lem:estlem} do not hold anymore, but
the following can be shown easily: for some $C=C(\delta)>0$ we have
$
N^{-C}<\left|\frac{e^{-2g(z)+V^{(t)}_0}}{f^{(t)}(z)}\right|<N^C$ for $ |z|>1$, and
$N^{-C}<\left|\frac{e^{2g(z)-V^{(t)}_0}}{f^{(t)}(z)}\right|<N^C$ for $ |z|<1$.
As a consequence, in the definition $\widetilde E(t)$, the integrals along $\Sigma_{\rm out}$ and $\Sigma''_{\rm out}$ are easily shown to be negligible. 

To control the other terms, we need to bound the solution of the Riemann-Hilbert problem.
Lemma \ref{lem:jumpestimate} still holds for our new jump matrix: this is a simple calculation from
\cite[equation (4.55)]{DeiItsKra2011}. Once this estimate is known, Proposition \ref{prop:Qestimate} and Corollary \ref{cor:boundE} still hold and their proof is unchanged. This concludes the proof of Proposition \ref{prop:tailIm}.
\end{proof}


\begin{bibdiv}
\begin{biblist}




\bib{AddRee08}{collection}{
    AUTHOR = {Addario-Berry, L.},
    author={Reed, B. A.},
     TITLE = {Ballot theorems, old and new},
 BOOKTITLE = {Horizons of combinatorics},
    SERIES = {Bolyai Soc. Math. Stud.},
    VOLUME = {17},
     PAGES = {9--35},
 PUBLISHER = {Springer, Berlin},
      YEAR = {2008}
}


\bib{AddRee09}{article}{
    AUTHOR = {Addario-Berry, L.},
    author={Reed, B. A.},
     TITLE = {Minima in branching random walks},
   JOURNAL = {Ann. Probab.},
    VOLUME = {37},
      YEAR = {2009},
    NUMBER = {3},
     PAGES = {p. 1044--1079}
}





\bib{Aid13}{article}{
    AUTHOR = {A{\"{\i}}d{\'e}kon, E.},
     TITLE = {Convergence in law of the minimum of a branching random walk},
   JOURNAL = {Ann. Probab.},
    VOLUME = {41},
      YEAR = {2013},
    NUMBER = {3A},
     PAGES = {1362--1426}
}



\bib{ArgBelHar15}{article}{
    AUTHOR = {Arguin, L.-P.},
    AUTHOR = {Belius ,D.},
    AUTHOR = {Harper, A.J.},
     TITLE = {Maxima of a randomized Riemann zeta function, and branching random walks},
   JOURNAL = {Preprint arXiv:1506.00629},
      YEAR = {2015},
}

\bib{ArgZin15}{article}{
   author={Arguin, Louis-Pierre},
   author={Zindy, Olivier},
   title={Poisson-Dirichlet statistics for the extremes of the
   two-dimensional discrete Gaussian free field},
   journal={Electron. J. Probab.},
   volume={20},
   date={2015},
   pages={no. 59, 19}
}

\bib{ArgZin14}{article}{
   author={Arguin, Louis-Pierre},
   author={Zindy, Olivier},
   title={Poisson-Dirichlet statistics for the extremes of a log-correlated
   Gaussian field},
   journal={Ann. Appl. Probab.},
   volume={24},
   date={2014},
   number={4},
   pages={1446--1481}
}

\bib{Bac00}{article}{
    AUTHOR = {Bachmann, M.},
     TITLE = {Limit theorems for the minimal position in a branching random
              walk with independent logconcave displacements},
   JOURNAL = {Adv. in Appl. Probab.},
    VOLUME = {32},
      YEAR = {2000},
    NUMBER = {1},
     PAGES = {159--176}
}



\bib{BatRao}{book}{
   author={Bhattacharya, R. N.},
   author={Ranga Rao, R.},
   title={Normal approximation and asymptotic expansions},
   note={Wiley Series in Probability and Mathematical Statistics},
   publisher={John Wiley \& Sons, New York-London-Sydney},
   date={1976}
}


\bib{BelKis14}{article}{
    AUTHOR = {Belius, D.},
    AUTHOR = {Kistler, N.},
     TITLE = {The subleading order of two dimensional cover times},
   JOURNAL = {Preprint arXiv:1405.0888},
      YEAR = {2014}
}


\bib{BenBou2013}{article}{
   author={Ben Arous, G.},
   author={Bourgade, P.},
   title={Extreme gaps between eigenvalues of random matrices},
   journal={Ann. Probab.},
   volume={41},
   pages={2648--2681},
   date={2013}
}


\bib{BisLou13}{article}{
    AUTHOR = {Biskup, M.},
    AUTHOR = {Louidor, O.},
     TITLE = {Extreme local extrema of two-dimensional discrete Gaussian free field},
   JOURNAL = {Preprint arXiv:1306.2602},
      YEAR = {2013}
}




\bib{BolDeuGia01}{article}{
    AUTHOR = {Bolthausen, E.},
   AUTHOR = {Deuschel, J.-D.},
   AUTHOR = {Giacomin,  G.},
     TITLE = {Entropic repulsion and the maximum of the two-dimensional
              harmonic crystal},
   JOURNAL = {Ann. Probab.},
    VOLUME = {29},
      YEAR = {2001},
    NUMBER = {4},
     PAGES = {1670--1692}
}

\bib{BovKur04}{article}{
   author={Bovier, Anton},
   author={Kurkova, Irina},
   title={Derrida's generalized random energy models. II. Models with
   continuous hierarchies},
   language={English, with English and French summaries},
   journal={Ann. Inst. H. Poincar\'e Probab. Statist.},
   volume={40},
   date={2004},
   number={4},
   pages={481--495}
}



%



\bib{BourgadeMesoscopic}{article}{
    AUTHOR = {Bourgade, P.},
     TITLE = {Mesoscopic fluctuations of the zeta zeros},
   JOURNAL = {Probab. Theory Related Fields},
    VOLUME = {148},
      YEAR = {2010},
    NUMBER = {3-4},
     PAGES = {479--500}
}

%
%
%
%


\bib{Bra78}{article}{
   author={Bramson, M. D.},
   title={Maximal displacement of branching Brownian motion},
   journal={Comm. Pure Appl. Math.},
   volume={31},
   date={1978},
   number={5},
   pages={531--581}
}







\bib{BramDinZei13}{article}{
    AUTHOR = {Bramson, M.},
    AUTHOR = {Ding, J.},
    AUTHOR = {Zeitouni, O.},
     TITLE = {Convergence in law of the maximum of the two-dimensional discrete Gaussian free field},
   JOURNAL = {Preprint arXiv:1301.6669},
      YEAR = {2013}
}




\bib{BraDinZei14}{article}{
  title = {Convergence in law of the maximum of nonlattice branching random walk}, 
AUTHOR = {Bramson, M.},
    AUTHOR = {Ding, J.},
    AUTHOR = {Zeitouni, O.},
  journal = {Preprint arxiv: 1404.3423},
  year = {2014}
}


\bib{BraZei09}{article}{
   author={Bramson, M.},
   author={Zeitouni, O.},
   title={Tightness for a family of recursion equations},
   journal={Ann. Probab.},
   volume={37},
   date={2009},
   number={2},
   pages={615--653},
}




\bib{BramZei12}{article}{
AUTHOR = {Bramson, M.},
    AUTHOR = {Zeitouni, O.},
     TITLE = {Tightness of the recentered maximum of the two-dimensional
              discrete {G}aussian free field},
   JOURNAL = {Comm. Pure Appl. Math.},
    VOLUME = {65},
      YEAR = {2012},
    NUMBER = {1},
     PAGES = {1--20}
}



\bib{CacMalSch2015}{article}{
      author={Cacciapuoti, C.},
      author={Maltsev, A.},
      author={Schlein, B.},
       title={Bounds for the Stieltjes transform and the density of states of Wigner matrices},
        date={2015},
     journal={Probability Theory and Related Fields},
      volume={163},
      number={1},
       pages={1\ndash 59}
}


\bib{CarLed2001}{article}{
  author = {Carpentier, D.},
    author = {Le Doussal, P.},
  title = {Glass transition of a particle in a random potential, front selection in nonlinear renormalization group, and entropic phenomena in Liouville and sinh-Gordon models},
  journal = {Phys. Rev. E},
  volume = {63},
  issue = {2},
  pages = {026110, 33pp.},
  year = {2001},
  month = {Jan},
}



\bib{ClaKra2014}{article}{
AUTHOR = {Claeys, T.},
    AUTHOR = {Krasovsky, I.},
     TITLE = {Toeplitz determinants with merging singularities},
   JOURNAL = {to appear in Duke Math. Journal, Preprint arxiv:1403.3639},
      YEAR = {2014}
}




\bib{Dei1999}{article}{
   author={Deift, P.},
   title={Integrable operators},
   journal={Amer. Math. Soc. Transl.},
   pages={69--84},
   volume={189},
   number={2},
   date={1999}
}





\bib{DeiItsKra2011}{article}{
   author={Deift, P.},
   author={Its, A.},
   author={Krasovsky, I.},
   title={Asymptotics of Toeplitz, Hankel, and Toeplitz+Hankel determinants with Fisher-Hartwig singularities},
   journal={Annals of Mathematics},
   pages={1243--1299},
   volume={174},
   date={2011}
}


\bib{DeiItsKra2014}{article}{
   author={Deift, P.},
   author={Its, A.},
   author={Krasovsky, I.},
   title={On the asymptotics of a Toeplitz determinant with singularities},
   journal={Mathematical Sciences Research Institute Publications},
   volume={65},
   date={2014},
    PAGES = {93--146}
}



\bib{DeiKriMcLVenZho1999}{article}{
   author={Deift, P.},
   author={Kriecherbauer, T.},
   author={McLaughlin, K. T-R},
   author={Venakides, S.},
   author={Zhou, X.},
   title={Strong asymptotics of orthogonal polynomials with respect to
   exponential weights},
   journal={Comm. Pure Appl. Math.},
   volume={52},
   date={1999},
   number={12},
   pages={1491--1552}
}


\bib{DemPerRosZeo04}{article}{
    AUTHOR = {Dembo, A.},
    AUTHOR = {Peres, Y.},
    AUTHOR = {Zeitouni, O.},
     TITLE = {Cover times for {B}rownian motion and random walks in two
              dimensions},
   JOURNAL = {Ann. of Math. (2)},
    VOLUME = {160},
      YEAR = {2004},
    NUMBER = {2},
     PAGES = {433--464}
}

\bib{DerSpo88}{article}{
   author={Derrida, B.},
   author={Spohn, H.},
   title={Polymers on disordered trees, spin glasses, and traveling waves},
   note={New directions in statistical mechanics (Santa Barbara, CA, 1987)},
   journal={J. Statist. Phys.},
   volume={51},
   date={1988},
   number={5-6},
   pages={817--840}
}

\bib{DiaEva01}{article}{
    AUTHOR = {Diaconis, P.},
    AUTHOR = {Evans, S. N.},
     TITLE = {Linear functionals of eigenvalues of random matrices},
   JOURNAL = {Trans. Amer. Math. Soc.},
    VOLUME = {353},
      YEAR = {2001},
    NUMBER = {7},
     PAGES = {2615--2633}
}

	

\bib{DiaSha1994}{article}{
   author={Diaconis, P.},
   author={Shahshahani, M.},
   title={On the eigenvalues of random matrices},
   note={Studies in applied probability},
   journal={J. Appl. Probab.},
   volume={31A},
   date={1994},
   pages={49--62}
}


\bib{DinRoyZei15}{article}{
  title = {Convergence of the centered maximum of log-correlated Gaussian fields},
  author = {Ding, J.},
  author = {Roy, R.},
  author = {Zeitouni, O.},
  journal = {Preprint arXiv:1503.04588},
  year = {2015}
}



\bib{DinZei14}{article}{
    AUTHOR = {Ding, J.},
    AUTHOR = {Zeitouni, O.},
     TITLE = {Extreme values for two-dimensional discrete {G}aussian free
              field},
   JOURNAL = {Ann. Probab.},
    VOLUME = {42},
      YEAR = {2014},
    NUMBER = {4},
     PAGES = {1480--1515}
}


\bib{ErdYauYin2012Rig}{article}{
      author={Erd{\H{o}}s, L.},
      author={Yau, H.-T.},
      author={Yin, J.},
       title={Rigidity of eigenvalues of generalized {W}igner matrices},
        date={2012},
     journal={Adv. Math.},
      volume={229},
      number={3},
       pages={1435\ndash 1515},
}





\bib{FyoBou08}{article}{
    AUTHOR = {Fyodorov, Y. V.},
    AUTHOR = {Bouchaud, J.-P.},
     TITLE = {Freezing and extreme-value statistics in a random energy model
              with logarithmically correlated potential},
   JOURNAL = {J. Phys. A},
    VOLUME = {41},
      YEAR = {2008},
    NUMBER = {37},
     PAGES = {372001, 12 pp.}
}



\bib{FyoHiaKea12}{article}{
  title = {Freezing Transition, Characteristic Polynomials of Random Matrices, and the {Riemann} Zeta Function},
  author = {Fyodorov, Y. V.},
  author = {Hiary, G. A.},
  author = {Keating, J. P.},
  journal = {Phys. Rev. Lett.},
  volume = {108},
  pages = {170601, 5pp.},
  year = {2012},
  publisher = {American Physical Society}
}



\bib{FyoKea14}{article}{
    AUTHOR = {Fyodorov, Y. V.},
    AUTHOR = {Keating, J. P.},
     TITLE = {Freezing transitions and extreme values: random matrix theory,
              and disordered landscapes},
   JOURNAL = {Phil. Trans. R. Soc. A},
    VOLUME = {372},
      YEAR = {2014},
     PAGES = {20120503, 32 pp. }
}


Yan Fyodorov, Boris , and Nicholas Simm


\bib{FyoKhoSim16}{article}{
    AUTHOR = {Fyodorov, Y. V.},
    AUTHOR = {Khoruzhenko, B.},
author={Simm, N.},
     TITLE = {Fractional Brownian motion with Hurst index $H=0$ and the Gaussian Unitary Ensemble},
   JOURNAL = {to appear in Annals of Probability},
      YEAR = {2016}
}



\bib{FyoLedRos09}{article}{
    AUTHOR = {Fyodorov, Y. V.},
    AUTHOR = {Le Doussal, P.},
author={Rosso, A.},
     TITLE = {Statistical mechanics of logarithmic {REM}: duality, freezing
              and extreme value statistics of {$1/f$} noises generated by
              {G}aussian free fields},
   JOURNAL = {J. Stat. Mech. Theory Exp.},
      YEAR = {2009},
    NUMBER = {10},
     PAGES = {P10005, 32 pp.}
}

\bib{FyoLedRos12}{article}{
   author={Fyodorov, Y. V.},
   author={Le Doussal, P.},
   author={Rosso, A.},
   title={Counting function fluctuations and extreme value threshold in
   multifractal patterns: the case study of an ideal $1/f$ noise},
   journal={J. Stat. Phys.},
   volume={149},
   date={2012},
   number={5},
   pages={898--920},
}



\bib{FyoSim2015}{article}{
    AUTHOR = {Fyodorov, Y. V.},
    AUTHOR = {Simm, N. J.},
     TITLE = {On the distribution of maximum value of the characteristic polynomial of GUE random matrices},
   JOURNAL = {preprint, arXiv:1503.07110},
      YEAR = {2015}
      }




%
%


\bib{Har13}{article}{
    AUTHOR = {Harper, A. J.},
     TITLE = {A note on the maximum of the Riemann zeta function, and log-correlated random variables},
   JOURNAL = {Preprint arXiv:1304.0677},
      YEAR = {2013}
}



\bib{Hei1878}{article}{
   author={Heine, H.},
   title={Kugelfunktionen},
   journal={Berlin, 1878 and 1881; reprinted, Physica Verlag, W{\"u}rzburg, 1961.}
}









\bib{HugKeaOco2001}{article}{
   author={Hughes, C. P.},
   author={Keating, J.},
   author={O'Connell, N.},
   title={On the Characteristic Polynomial of a Random Unitary Matrix},
   journal={Comm. Math. Phys.},
   number={2},
   pages={429--451},
   date={2001}
}



\bib{KeaSna2000}{article}{
   author={Keating, J.},
   author={Snaith, N.},
   title={Random Matrix Theory and $\zeta(1/2+\ii t)$},
   journal={Comm. Math. Phys.},
   number={1},
   pages={57--89},
   date={2000}
}





\bib{Kis15}{collection}{
    AUTHOR = {Kistler, N.},
     TITLE = {Derrida's Random Energy Models},
 BOOKTITLE = {Correlated Random Systems: Five Different Methods},
    SERIES = {Lecture Notes in Math.},
    VOLUME = {2143},
     PAGES = {71--120},
 PUBLISHER = {Springer, Berlin},
      YEAR = {2015},
}

\bib{Mad14}{article}{
    AUTHOR = {Madaule, T.},
     TITLE = {Maximum of a log-correlated Gaussian field},
   JOURNAL = {Preprint arXiv:1307.1365},
      YEAR = {2014}
}




\bib{MarMcLSaf2006}{article}{
   author={Mart{\'{\i}}nez-Finkelshtein, A.},
   author={McLaughlin, K. T.-R.},
   author={Saff, E. B.},
   title={Asymptotics of orthogonal polynomials with respect to an analytic
   weight with algebraic singularities on the circle},
   journal={Int. Math. Res. Not.},
   date={2006},
   PAGES = {Art. ID 91426, 43 pp.}
}




%
%


\bib{RhoVar13}{article}{
author = {Rhodes, R.},
author = {Vargas, V.},
journal = {Probab. Surveys},
pages = {315--392},
publisher = {The Institute of Mathematical Statistics and the Bernoulli Society},
title = {Gaussian multiplicative chaos and applications: A review},
volume = {11},
year = {2014}
}


%
%

\bib{Ste1970}{book}{
   author={Stein, E. M.},
   title={Singular integrals and differentiability properties of functions},
   series={Princeton Mathematical Series, No. 30},
   publisher={Princeton University Press, Princeton, N.J.},
   date={1970}
}


\bib{SteSha03}{book}{
    AUTHOR = {Stein, E. M.},
    AUTHOR = {Shakarchi, R.},
     TITLE = {Fourier analysis},
    SERIES = {Princeton Lectures in Analysis},
    VOLUME = {1},
      NOTE = {An introduction},
 PUBLISHER = {Princeton University Press, Princeton, NJ},
      YEAR = {2003},
     PAGES = {xvi+311 pp.}
}




\bib{Vor1987}{article}{
   author={Voros, A.},
   title={Spectral Functions, Special Functions and the Selberg Zeta Function},
   journal={Commun. Math. Phys.},
   volume={110},
   pages={439--465},
   date={1987}
}



\bib{Web15}{article}{
   author={Webb, C.},
   title={The characteristic polynomial of a random unitary matrix and Gaussian multiplicative chaos - The ${\rm L}^2$-phase},
   journal={axXiv:1410.0939},
   volume={},
   date={2014},
   number={},
   pages={}
}









\end{biblist}
\end{bibdiv}
\end{document}